\documentclass[11pt,english]{article}
\usepackage{amsmath}
\usepackage{amsthm}
\usepackage{amssymb}
\usepackage[authoryear]{natbib}

\makeatletter
\theoremstyle{plain}
\newtheorem{thm}{\protect\theoremname}
\theoremstyle{plain}
\newtheorem{lem}[thm]{\protect\lemmaname}
\theoremstyle{plain}

\newtheorem{prop}{Proposition}

\newtheorem{definition}{Definition}

\@ifundefined{date}{}{\date{}}
\usepackage{amsfonts}\usepackage{amsthm}\usepackage{graphicx}\usepackage{bm}
\usepackage{subcaption}
\usepackage{graphicx}
\usepackage{soul}
\usepackage[font=small,labelfont=bf]{caption}

\usepackage{verbatim}
\usepackage{rotating}
\usepackage{caption}
\usepackage{paralist}
\usepackage{booktabs}
\usepackage{hyperref}

\usepackage{epsfig}

\usepackage{bbm}
\usepackage{indentfirst}
\usepackage{float}
\usepackage{authblk}
\theoremstyle{plain}

\makeatother
\providecommand{\keywords}[1]
{
  \small	
  \textbf{\textit{Keywords---}} #1
}

\providecommand{\assumptionname}{Assumption}
\providecommand{\lemmaname}{Lemma}
\providecommand{\theoremname}{Theorem}

\usepackage{algorithm}
\usepackage{algorithmicx}
\usepackage{algpseudocode}
\makeatother

\usepackage{babel}
\providecommand{\lemmaname}{Lemma}

\providecommand{\theoremname}{Theorem}
\newcommand{\algorithmiclastcon}{\textbf{Lastcon:}}
\newcommand{\lastcon}{\item[\algorithmiclastcon]}

\renewcommand{\algorithmiclastcon}{\textbf{Output:}}

\date{}

\begin{document}

\title{Higher-Order Coverage Errors of Batching Methods via Edgeworth Expansions on $t$-Statistics}
\author{Shengyi He$^{1}$, Henry Lam$^{1}$  \\
        \small $^{1}$ Department of Industrial Engineering and Operations Research, Columbia University \\
}

\maketitle

\begin{abstract}
While batching methods have been widely used in simulation and statistics, it is open regarding their higher-order coverage behaviors and whether one variant is better than the others in this regard.
We develop techniques to obtain higher-order coverage errors for batching methods by building Edgeworth-type expansions on $t$-statistics. The coefficients in these expansions are intricate analytically, but we provide algorithms to estimate the coefficients of the $n^{-1}$ error term via Monte Carlo simulation. We provide insights on the effect of the number of batches on the coverage error where we demonstrate generally non-monotonic relations. We also compare different batching methods both theoretically and numerically, and argue that none of the methods is uniformly better than the others in terms of coverage. However, when the number of batches is large, sectioned jackknife has the best coverage among all. 
\end{abstract}
\keywords{Batching methods, Edgeworth expansion, coverage errors, simulation analysis}

    

\section{Introduction}
Batching methods are widely used in simulation analysis. The basic idea of these methods is to divide the data into batches and quantify the variability of point estimates by suitably combining the batch estimates. They are especially useful tools to construct confidence intervals (CIs) when the variance of the output is hard to compute, such as quantile \citep{Nakayama2014confidence} whose variance estimation involves density estimation, and in serially dependent problems and steady-state estimation \citep{Asmussen2007,NAKAYAMA20071330}. 



While widely used, the detailed coverage behaviors of batching methods, beyond the well-known asymptotic limits, are quite open. To understand and to better compare the statistical performances of these methods, however, this question seems imminent. To put things in perspective, note that a good CI should have a small coverage error, which refers to the difference between the empirical coverage and the prescribed target coverage. It is known that under regularity conditions, all batching methods are asymptotically exact, in the sense that they achieve the prescribed coverage level as the sample size increases. In other words, each method only have ``higher-order'' coverage errors that would converge to zero as the sample size increases. However, when the sample size is finite, this higher-order coverage error could be significant and also affected by the number of batches and method used, which is what we investigate mainly in this paper.

 There are very few studies on the higher-order coverage errors of batching methods. The challenge is that the statistics used in these methods have an asymptotic $t$-distribution rather than a normal distribution, so conventional Edgeworth expansion cannot be directly applied. The most relevant result is the heuristic argument given in \cite{Nakayama2014confidence}, which argue that since the estimator based on the entire empirical distribution has a smaller bias, the so-called sectioning appears to lead to a better coverage compared to batching. \cite{Nakayama2014confidence} supports this claim with numerical results.
 
 In this paper, we develop tools to study higher-order expansions for the coverage probabilities of batching methods. Let $K$ be the number of batches and suppose that each batch has $n$ i.i.d. samples. When $K\geq r+3$ for some positive integer $r$, under regularity conditions, we show that the coverage errors of batching methods can be expanded as a series of $n^{-1/2}$ with residual $O(n^{-(r+1)/2})$. For symmetric CIs, we show that batching methods have coverage errors of order $O(n^{-1})$. To support the necessity of the assumption $K\geq r+3$, we provide examples where such expansion does not exist when $K=2,r=1$. The coefficients in the expansion involve some integration that cannot be explicitly calculated in general, but for symmetric CIs, we design a simulation scheme to generate unbiased estimates for the coefficient of the $n^{-1}$ error term, which is the leading term of the coverage error.
 
 Our argument for coverage error expansions can be extended to more general cases without the i.i.d. assumption as long as there is an Edgeworth expansion for the joint distribution of the batch averages. We provide extensions to the scenario where the sample is a dependent data sequence with proper recurrence and mixing conditions, for which batching methods are commonly used. To ensure an Edgeworth expansion, however, we need to run variants of the batching methods that are slightly different from standard practice, one by leaving a gap between successive batches, and one using regenerative cycles. These variants allow us to explicitly analyze the coverage errors under data dependence.

In terms of methodology, our analysis utilizes Edgeworth expansion and Taylor's expansion techniques combined with oddness and evenness arguments for functions. More precisely, we approximate the event that the target value is covered by the CI using a Taylor expansion argument where the coefficients are given by the implicit function theorem. Then, we study the probability of the approximated event  by integrating with the Edgeworth expansion that leverages embedded normality. For a symmetric CI, we use an oddness and evenness argument to show that the coefficient of the $n^{-1/2}$ term is 0. To our best knowledge, our line of techniques in deriving $t$-statistic expansions and applying it to the open question on higher-order coverage errors of batching methods appear the first in the literature.

By numerically comparing the theoretical coverages using our expansion (where the coefficients are estimated via simulation) and the actual coverages (estimated from repeated experiments), we show that our expansions give close approximations to the actual coverages. This allows us to compare the coverage errors of different batching methods and draw insights on the effect of the number of batches based on our expansion. From our analyses, we conclude that which method has smaller higher-order coverage error depends on problem parameters and none of them is uniformly better. However, for a fixed problem, we show that when $K$ is large and $n$ is fixed, batching suffers from a significant bias and has asymptotically 0 coverage, sectioning has an asymptotically incorrect coverage, while sectioned jackknife has asymptotically correct coverage.
When the total number of data is fixed, coverage errors tend to increase as $K$ increases. However, there is no monotonicity, meaning that for any batching method that we consider, when $K$ is small, the coefficient could either increase or decrease as $K$ increases depending on the problem parameters.

We summarize the main results and contributions of this paper:
\begin{enumerate}
    \item\emph{Building expansion techniques for $t$-statistics:} We develop techniques to study higher-order expansions for the distribution of statistics whose limiting distribution is $t$. Our techniques utilize the implicit function theorem combining with embedded normality in the considered statistic that allows us to leverage established Edgeworth expansions as a building block.
    \item\emph{Higher-order expansions for coverage errors of batching methods:} With the $t$-statistic expansions, we obtain expansions for the coverage errors of batching methods, both in the case of i.i.d. data and when data comes from a dependent sequence with proper recurrence and mixing conditions. In the latter case, we provide two approaches to run the batching methods that allow explicit analyses on the coverage errors.
 \item\emph{Sufficient conditions on the number of batches $K$:} To obtain our expansions with residuals of order $O(n^{-(r+1)/2})$, our theorem requires $K\geq r+3$. We provide an example where an expansion with residual $O(n^{-1})$ does not exist when $K=2$, which implies that the condition $K\geq 3$ is necessary when $r=1$.
      \item\emph{Simulation-based  algorithms to estimate coefficients in the expansions:} The coefficients in the expansions of coverage errors are intricate analytically, but amenable to simulation. We provide algorithms to estimate the coefficients of the $n^{-1}$ term. We prove that the algorithms give unbiased estimates for the coefficients.
    \item\emph{Insights on methods used and the effect of the number of batches:} We show that none of the batching methods are uniformly better than the others, but when the number of data in each batch is fixed and $K\rightarrow\infty$, the coverage probability of batching goes to 0, the coverage probability of sectioning goes to a limit that is different from the nominal level, while the coverage probability of sectioned jackknife converges to the nominal level. 
    When the total number of data is fixed, the coverage error tends to empirically increase as $K$ increases, but the error coefficient in our expansion is not monotone in $K$ and thus no theoretical monotonicity is achieved in general.
\end{enumerate}

The rest of this paper is as follows. Section \ref{subsec: literature} reviews related literature. Section \ref{sec: methods} introduces all considered batching methods. Section \ref{sec: validity} presents our higher-order coverage error expansions. Section \ref{sec: dependent} extends these expansions to models with dependent data. Section \ref{sec: alg} provides an algorithm to generate unbiased estimates for the $n^{-1}$ error terms for symmetric CIs. Section \ref{sec: asymptotic} discusses the asymptotic coverages as the number of batches grows. Section \ref{sec: numerics} validates our theoretical results via numerical experiments and uses our results to compare different methods and batch number choices. Technical proofs, some computation details, an alternative algorithm for batching, and additional numerical experiments are provided in the Appendix.

\section{Literature review}\label{subsec: literature}
We briefly review the literature on batching methods. \cite{pope1995improving} analyzes the coverage error of sectioning using Edgeworth expansion, but it focuses on the case when the number of batches goes to infinity so that the problem statistic can be approximated by normal. This is different from our analysis for the $t$ distribution approximation, which is our key novelty and faced challenge in this problem.  For the CI half width, \cite{Schmeiser1982batch} shows that if we assume the data size is large enough so that the non-normality of the batch estimators is negligible, then the expected half width would decrease as the number of batches increases, but the rate of decrease would become much slower when the number of batches is large. Similar observations are also made in \cite{glynn2018constructing}. Jackknife can be used to reduce small-sample bias within sections, but at the cost of greater computation time and uncertainty about the variance inflation \citep{lewis1989simulation}. 

Batching methods are commonly used in simulation output analysis, especially for steady state estimation where the data come from a dependent process. \cite{alexopoulos1996implementing} use batching to study steady-state means and numerically test different strategies for choosing the number of batches. They also study overlapping batch means where different batches can overlap with each other. \cite{steiger2002improved} also study steady-state means and proposes an algorithm called ASAP to progressively increase the batch size until the batch means pass the independence test or multivariate normality test. Some refinements (ASAP2,ASAP3) are provided in \cite{asap2,asap3}. \cite{Tafazzoli2008skart,Tafazzoli2011nskart} propose SKART and N-SKART which are adjusted batching methods based on skewness and autoregression.
Other algorithms concerning the CIs for steady-state means include \cite{LADA20061769} (WASSP) and \cite{LADA20061769} (SBATCH). \cite{Alexopoulos2014sequest} use the idea of sectioning and batching to develop algorithms to build CIs for steady-state quantiles. \cite{munoz1997batch} uses batching methods for the estimation of non-linear functions of steady-state means. \cite{dong2019new,dong2019asymptotic} use batching methods to decide the stopping time in sequential procedures and characterize the limiting distributions of the estimators at stopping times. Batching methods can also be used to estimate the variances of Markov chain Monte Carlo \citep{geyer1992practical} and the consistency of these variance estimators is established in  \cite{flegal2010batch,jones2006fixed}. Finally,   \cite{song1995optimal,flegal2010batch} analyze optimal batch sizes that minimize the mean squared errors of batching variance estimators.

Batching can also be seen as a special type of the more general umbrella technique of standardized time series (STS), as shown in Example 3.1 of \cite{Glynn1990simulation}. Similar to batching methods, the idea of STS is to cancel out the variance term by taking the ratio between a point estimator and a variance estimator, but in a more general way where the variance estimator is viewed as a finite-sample approximation of a general functional of Brownian motions that is asymptotically independent of the point estimator. \cite{Schruben1983confidence} proposes STS as a method to construct CIs for the steady-state mean of a stationary process. Generalizations and properties of STS-type methods are studied extensively in the literature, including \cite{Glynn1990simulation,Goldsman1990new,calvin2006permuted,antonini2009area,Alexopoulos2016spsts}. STS can also be used for constructing CIs for steady-state quantiles. \cite{calvin2013confidence} and \cite{Alexopoulos2019sequential,alexopoulos2020steady} use STS to study steady-state quantiles and establish the asymptotic validity of the STS CI under different conditions. 

Lastly, we mention a preliminary conference version of this work \cite{he2021error}. All contents in this paper are new except the argument on evenness and oddness that we utilize and the computation of expansions on some simple examples that we include in Appendix \ref{sec: explicit examples}.

\section{Batching methods}\label{sec: methods}

Consider the problem of constructing a CI for $\psi(P)$ where $P$ is an unknown distribution, $\psi$ is a known statistical functional  and we have data $X_1.\dots,X_{N}$ drawn i.i.d. from $P$. Suppose the data size is $N=nK$. Divide the data into $K$ batches each with size $n$ and denote $\hat{P}_i$ as the empirical distribution for the $i$-th batch where $i=1,2,\dots,K$. Denote $\hat{P}$ as the entire empirical distribution using all of the $N=nK$ data. By batching methods, in this paper we mean the following four variants:

\begin{itemize}
    \item\underline{Batching:} The batching CI is given by 
$$CI_{B} := \left(\frac{1}{K}\sum_{i}\psi\left(\hat{P}_{i}\right) \pm t_{K-1,\alpha/2}\frac{S_{\text{batch}}}{\sqrt{K}}\right)$$
where $S_{\text{batch}}^2=\frac{1}{K-1}\sum_{i=1}^{K}\left(\psi\left(\hat{P}_{i}\right)-\frac{1}{K}\sum_{j}\psi\left(\hat{P}_{j}\right)\right)^{2}$ and $t_{K-1,\alpha/2}$ is the upper $\alpha/2$-quantile of the $t_{K-1}$, the $t$-distribution with degree of freedom $K-1$. That is, batching uses the batch estimates $\psi(\hat P_i)$ as primitives and the sample mean and sample variance of these batch estimates to construct a CI.
\\

\item \underline{Sectioning:} The sectioning CI is given by 
$$CI_S := \left(\psi\left(\hat{P}\right) \pm t_{K-1,\alpha/2}\frac{S_{\text{sec}}}{\sqrt{K}}\right)$$
where $S_{\text{sec}}^2=\frac{1}{K-1}\sum_{i=1}^{K}\left(\psi\left(\hat{P}_{i}\right)-\psi\left(\hat{P}\right)\right)^{2}$. Compared with batching, sectioning uses the point estimate $\psi(\hat P)$ constructed from the entire empirical distribution in both the center of the interval and the center in the variance estimator $S_{sec}^2$.
\\

\item \underline{Sectioning-batching (SB):} The SB CI is 
$$CI_{SB} := \left(\psi\left(\hat{P}\right) \pm t_{K-1,\alpha/2}\frac{S_{\text{batch}}}{\sqrt{K}}\right).$$
SB is a modified sectioning \citep{Nakayama2014confidence} that is viewed as a middle ground between batching and sectioning. It uses the same variance estimator $S_{batch}^2$ as batching, but the same interval center $\psi(\hat P)$ as sectioning. 
\\

\item \underline{Sectioned jackknife (SJ):}  Let
$\hat{P}_{(i)}$ be the empirical distribution of all samples except for
those from the $i$-th section. Let $J_{i}=K\psi(\hat{P})-(K-1)\psi(\hat{P}_{(i)})$. The SJ CI is given by $$CI_{SJ}:=\left(\bar{J}\pm t_{K-1,\alpha/2}\frac{S_{\text{SJ}}}{\sqrt{K}}\right)$$
where $S_{\text{SJ}}^2=\frac{1}{K-1}\sum_{i=1}^{K}(J_{i}-\bar{J})^{2}$. SJ works in a similar way as conventional jackknife, but instead of considering the leave-one-out data, we leave one section or batch out. Like the conventional jackknife, SJ is also known to be bias-corrected, but requires less computational cost (Section III.5b, \cite{Asmussen2007}). 
\\
\end{itemize}

Under regularity conditions, all four methods above are asymptotically exact, i.e., $P(\psi(P)\in CI_\cdot)\rightarrow 1-\alpha$ as $n\to\infty$ where $\cdot$ can be $B$, $S$, $SB$ or $SJ$. For batching, this can be seen by the fact that the corresponding statistic has a limiting $t_{K-1}$ distribution:
\[
W_{B}:=\frac{\sqrt{nK}\left(\frac{1}{K}\sum_{i}\psi\left(\hat{P}_{i}\right)-\psi\right)}{\sqrt{\frac{1}{K-1}\sum_{i=1}^{K}\left(\sqrt{n}\psi\left(\hat{P}_{i}\right)-\frac{1}{K}\sum_{j}\sqrt{n}\psi\left(\hat{P}_{j}\right)\right)^{2}}} \Rightarrow t_{K-1}.
\]
Here $\psi:=\psi(P)$ is the target value and the limit is as $n\rightarrow\infty$ with $K$ fixed. Sectioning is also asymptotically exact since 
\[
W_{S}:=\frac{\sqrt{nK}\left(\psi\left(\hat{P}\right)-\psi\right)}{\sqrt{\frac{1}{K-1}\sum_{i=1}^{K}\left(\sqrt{n}\psi\left(\hat{P}_{i}\right)-\sqrt{n}\psi\left(\hat{P}\right)\right)^{2}}} = W_B + o_p(1) \Rightarrow t_{K-1}.
\] 
Similar to sectioning and batching, SB is also asymptotically exact since
\[
W_{SB}:=\frac{\sqrt{nK}\left(\psi\left(\hat{P}\right)-\psi\right)}{\sqrt{\frac{1}{K-1}\sum_{i=1}^{K}\left(\sqrt{n}\psi\left(\hat{P}_{i}\right)-\frac{1}{K}\sum_{j}\sqrt{n}\psi\left(\hat{P}_{j}\right)\right)^{2}}} = W_B + o_p(1) \Rightarrow t_{K-1}.
\]
Finally, the asymptotic exactness of SJ can be seen from
\[
W_{SJ}:=\frac{\sqrt{nK}\left(\bar{J}-\psi_{0}\right)}{\sqrt{\frac{1}{K-1}\sum_{i=1}^{K}(\sqrt{n}J_{i}-\sqrt{n}\bar{J})^{2}}}\Rightarrow t_{K-1}.
\]
To check that the last convergence indeed holds, consider the case $\psi(P)=E_{P}X$ first where $E_P$ denotes the expectation under $P$. In this
case, $J_{i}=\bar{X}_{i}$ so SJ is equivalent to
batching and sectioning whose statistic has limit $t_{K-1}$. More generally, with proper differentiability of $\psi$,
we can approximate $\psi(\hat{P})$ and $\psi(\hat{P}_{(i)})$ with
$E_{\hat{P}}IF(X)$ and $E_{\hat{P}_{(i)}}IF(X)$ where $IF(\cdot)$ is the influence function of $\psi$ at $P$. From this, we can get
the same asymptotic distribution.

The four CIs we have introduced above are two-sided and symmetric in the sense that each of the CI has a midpoint at the respective point estimate. With the limiting distributions of the statistics given above, it is not hard to see that the following lower one-sided CIs are also valid: $\widetilde{CI}_{B} := \left(-\infty,\frac{1}{K}\sum_{i}\psi\left(\hat{P}_{i}\right) + t_{K-1,\alpha}\frac{S_{\text{batch}}}{\sqrt{K}}\right)$, $\widetilde{CI}_S := \left(-\infty, \psi\left(\hat{P}\right) + t_{K-1,\alpha}\frac{S_{\text{sec}}}{\sqrt{K}}\right)$, $\widetilde{CI}_{SB} := \left(-\infty, \psi\left(\hat{P}\right) + t_{K-1,\alpha}\frac{S_{\text{batch}}}{\sqrt{K}}\right)$, $\widetilde{CI}_{SJ}:=\left(-\infty, \bar{J}+ t_{K-1,\alpha}\frac{S_{\text{SJ}}}{\sqrt{K}}\right)$. Similarly, one can consider upper one-sided CIs with shape $(\psi_l,\infty)$ for some $\psi_l$ or two-sided CIs that are not symmetric.

Although each of the four batching methods introduced above has asymptotically correct coverage, the coverage might be poor when we only have a finite sample. Note that $-t_{K-1,\alpha/2}\leq W_B\leq t_{K-1,\alpha/2} \Leftrightarrow \psi\in CI_{B}$ and similar arguments hold for sectioning, SJ and SB and for one-sided CIs. Therefore, to study the higher-order coverage errors, it suffices to study the distributions of $W_B,W_S,W_{SB}$ and $W_{SJ}$, and how much they deviate from $t_{K-1}$. 


\section{Expansions on batching methods with $t$-limits}\label{sec: validity}

For batching, by integrating over the Edgeworth expansion for batched estimates, we obtain the following.
\begin{thm}[Coverage error expansion for batching]\label{thm: batching_expansion}

Suppose that $\psi\left(\hat{P}_1\right)$ has a valid Edgeworth expansion, in the sense that for some $0<\sigma<\infty$,
\begin{equation}\label{eq: Edgeworth assumption}
P\left(\frac{\sqrt{n}\left(\psi\left(\hat{P}_{1}\right)-\psi\right)}{\sigma}\leq q\right)=\Phi(q)+\sum_{j=1}^{r}n^{-j/2}p_{j}(q)\phi(q)+O\left(n^{-(r+1)/2}\right)
\end{equation}
holds uniformly over $q\in\mathbb{R}$, and $p_j$ is an even polynomial when $j$ is odd and is an odd polynomial when $j$ is even. Here $\Phi$ and $\phi$ are the distribution function and density function of standard normal. Then: 
\begin{itemize}
    \item For any $q\in\mathbb{R}$, there exists $c_j^{(B,K)}\in\mathbb{R},j=1,2,\dots,r$, such that  $$P(W_{B}\leq q)=P(t_{K-1}\leq q)+\sum_{j=1}^rn^{-j/2}c_j^{(B,K)} + O(n^{-(r+1)/2})$$
    Here the coefficients $c_{j}^{(B,K)}$ depends on $K$, the distribution $P$, the objective function $\psi$ and the value of $q$, but do not depend on $n$.
    \item $P\left(-q\leq W_{B}\leq q\right)=P(-q\leq t_{K-1}\leq q)+O(n^{-1})$.
\end{itemize}
\end{thm}
\begin{proof}
As long as we have a valid Edgeworth expansion for $\sqrt{n}\left(\psi\left(\hat{P}_{i}\right)-\psi\right)$,
noting that $W_{B}$ is a function of $\left(\sqrt{n}\left(\psi\left(\hat{P}_{i}\right)-\psi\right)\right)_{i=1}^{K}$,
we can evaluate the probability $P\left(W_{s}\leq q\right)$ based
on integration. Let $f(\mathbf{z})=\frac{\sqrt{K}\frac{1}{K}\sum_{i=1}^Kz_i}{\sqrt{\frac{1}{K-1}\sum_{i=1}^{K}(z_i-\frac{1}{K}\sum_{j}z_j)^2}}$. Then one can check that
\begin{equation}\label{eq: WB function}
    W_{B}=f\left(\left(\sqrt{n}\left(\psi\left(\hat{P}_{i}\right)-\psi\right)/\sigma\right)_{i=1}^{K}\right)
\end{equation}
Therefore, from \eqref{eq: Edgeworth assumption}, by integration we can get (here, the reminder term below has the claimed order because the testing region $f(\mathbf{z})\leq q$ has a nice form: given $z_1,\dots,z_{i-1},z_{i+1},\dots,z_{K}$, the values of $z_i$ that makes $f(\mathbf{z})\leq q$ can be described by at most two intervals, whose measures are uniformly controlled by \eqref{eq: Edgeworth assumption}. See Appendix \ref{sec: proof other} for a detailed proof on this.)
\begin{equation}\label{eq: WB integration}
    P\left(W_{B}\leq q\right)=\int_{f(\mathbf{z})\leq q}\Pi_{j=1}^{K}d\left(\Phi(z_{j})+\sum_{j=1}^{r}n^{-j/2}p_{j}(z_{j})\phi(z_{j})\right)+O\left(n^{-(r+1)/2}\right).
\end{equation}

For a symmetric CI, by a similar integration, we have that
\begin{align*}
& P\left(-q\leq W_{B}\leq q\right) \\
 = & \int_{-q\leq f(\mathbf{z})\leq q}\Pi_{j=1}^{K}d\left(\Phi(z_j)+n^{-1/2}p_{1}(z_j)\phi(z_j)\right)+O\left(n^{-1}\right)\\
 = & P\left(-q\leq t_{K-1}\le q\right)+Kn^{-1/2}\int_{-q\leq f(\mathbf{z})\leq q}\phi(z_{1})\phi(z_{2})\dots\phi(z_{K})\left(-z_{1}p_{1}(z_{1})+p_{1}^{\prime}(z_{1})\right)d\mathbf{z}\\&+O(n^{-1}).
\end{align*}
Here $p_1^{\prime}$ is the derivative of $p_1$. Since $p_{1}$ is an
even polynomial, $-z_{1}p_{1}(z_{1})+p_{1}^{\prime}(z_{1})$ is an odd polynomial.
In addition, note that the area $\{-q\leq f(\mathbf{z})\leq q\}$ is symmetric around
$\mathbf{0}$ since $f(\mathbf{z})=-f(-\mathbf{z})$. Thus, the integration
in the RHS above is 0. As a result, we have that $P\left(-q\leq W_{s}\leq q\right)$
can be expanded as a power of $n^{-1/2}$ and its leading term is
of order $n^{-1}$.
\end{proof}

Theorem \ref{thm: batching_expansion} implies that the distribution of the batching statistic can be expanded as a series of $n^{-1/2}$, and that the coverage error for the symmetric CI (i.e., a CI centered at the point estimate as introduced in Section \ref{sec: methods}) given by batching is of order $O(n^{-1})$. The reasonableness of the condition \eqref{eq: Edgeworth assumption} can be checked from Theorem 2.2 of \cite{Hall1992}. Relation \eqref{eq: WB function} is important for the proof, which says that $W_B$ can be written as a function of $\left(\sqrt{n}\left(\psi\left(\hat{P}_{i}\right)-\psi\right)\right)_{i=1}^{K}$ whose distribution is well understood by Edgeworth expansion. Moreover, this function does not depend on $n$, so in \eqref{eq: WB integration}, when we integrate with respect to the Edgeworth expansion, the area of integration remains fixed when $n$ changes. 

However, for other batching methods, the above is no longer the case. Consider sectioning for example. $W_{S}$ cannot
be expressed as merely a function of $\left(\sqrt{n}\left(\psi\left(\hat{P}_{i}\right)-\psi\right)\right)_{i=1}^{K}$, but is also dependent on $\psi(\hat{P})$.
Moreover, it is difficult to study the joint distribution of
$$\Lambda:=\left(\sqrt{nK}\left(\psi\left(\hat{P}\right)-\psi\right),\left(\sqrt{n}\left(\psi\left(\hat{P}_{i}\right)-\psi\right)\right)_{i=1}^{K}\right)$$
via Edgeworth expansion, since its asymptotic joint distribution is degenerate. By the latter we mean that, under regularity conditions, $\sqrt{nK}\left(\psi\left(\hat{P}\right)-\psi\right) - \sqrt{K}\frac{1}{K}\sum_{i=1}^K \sqrt{n}\left(\psi\left(\hat{P}_{i}\right)-\psi\right) =o_p(1) $, which implies that the limiting distribution of $\Lambda\in\mathbb{R}^{K+1}$ has only $K$ degrees of freedom, hence is degenerate. Similarly for SJ, we note that $SJ$ depends on $\psi(\hat{P}),\psi(\hat{P}_{(1)}),\dots,\psi(\hat{P}_{(K)})$, but again it appears difficult to study
its joint distribution
because of the asymptotic degeneracy. 

To handle this issue, we focus on the smooth
function model where $\psi(P)=f(E_{P}X)$. For this model, it is not difficult to see that each of $W_S,W_{SB}$ and $W_{SJ}$ can be written as a function of $\tilde{\Lambda}:=\left(\sqrt{n}\left(\bar{X}_i-E_PX\right)\right)_{i=1}^{K}$ whose distribution can be approximated by Edgeworth expansion. However, we still face some difficulties. To illustrate this, let's take sectioning as an example and denote this function as $f_S$, which gives $W_S=f_S(\tilde{\Lambda})$. As one can check, the function $f_S(\cdot)$ would depend on $n$ (as long as $f$ is nonlinear), so the set of $\tilde{\Lambda}$ that makes $f_S(\tilde{\Lambda})\leq q$ depends on $n$. Therefore, a direct integration as in the proof of Theorem \ref{thm: batching_expansion} will not give an expansion whose coefficients are free of $n$. To address this, we approximate the event $W_S\leq q$ by using a conditioning argument and studying the Taylor expansion for the critical value. More precisely, letting $A_i=\sqrt{n}(\bar{X}_i-E_PX)$, we consider the vector  $A:=(\bar{A},A_1-\bar{A},\dots,A_{K-1}-\bar{A})\in\mathbb{R}^{dK}$ which can be seen as a linear transformation of $\tilde{\Lambda}$ (so its Edgeworth expansion is also known). Let $A_{0,1}\in\mathbb{R}$ denote the first coordinate of $A$ and let $A^{\prime}\in\mathbb{R}^{dK-1}$ denote the rest of the coordinates of $A$. We reformulate the event $W_{S}\leq q$ (whose probability we are interested in) as $ A_{0,1}\leq F_{+}^{n}$ where $F_{+}^{n}$ is a function of $A^{\prime}$ (and depend on $n$). The expansion for $F_{+}^{n}$ w.r.t. $n^{-1/2}$ can be studied by the implicit function theorem. Moreover, the smoothness of (the Edgeworth approximation for) the distribution of $A_{0,1}$ can be used to argue that the contribution of the higher-order terms in the expansion for $F_{+}^{n}$ to the final expectation is also of higher order. Therefore, with the expansion for $F_{+}^{n}$, by integrating with the Edgeworth expansion for $A$ in a proper way, we get an expansion for the probability of $A_{0,1}\leq F_{+}^{n}$ which is what we want. The detailed argument is given in Appendix \ref{sec: proof_validity}. We can show the following theorem regarding the validity and order of expansion for other batching methods. 
\begin{thm}[Coverage error expansions for all methods]
\label{thm: validity}Suppose that $\psi(\cdot)$ is a statistical functional mapping from distributions in $\mathbb{R}^d$ to $\mathbb{R}$ defined by $\psi\left(\bar{P}\right)=f\left(E_{\bar{P}}X\right)$
for a vector $X\sim \bar{P},X\in\mathbb{R}^{d}$. Assume the following Cramer's condition
holds for the distribution of $X_1\in\mathbb{R}^d$ (recall that $X_1,\dots,X_{nK}$ are drawn i.i.d. from $P$):
\begin{equation}\label{eq: cramer}
\limsup_{\left|t\right|\rightarrow\infty}\left|E_{P}(\exp\{i\left<t,X\right>\})\right|<1,
\end{equation}
Suppose that for some positive integer $r$, $X$ has finite moments up to order
$r+2$ with nonsingular covariance, and $f$ is $r+1$ times differentiable in a neighborhood of
$E_{P}X$ with $\nabla f(E_{P}X)\neq 0$. Then:
\begin{itemize}
    \item If $K\geq r+3$, then for any $q\in\mathbb{R}$, there exists $c_j^{(SJ,K)}\in\mathbb{R},j=1,2,\dots,r$, such that  $$P(W_{SJ}\leq q)=P(t_{K-1}\leq q)+\sum_{j=1}^rn^{-j/2}c_j^{(SJ,K)} + O(n^{-(r+1)/2})$$
    Here the coefficients $c_{j}^{(SJ,K)}$ depends on $K$, the distribution $P$ that generates each $X_i$, the objective function $f$ and the value of $q$, but does not depend on $n$.
\item Suppose that $K\geq 4$. Then we have $P\left(-q\leq W_{SJ}\leq q\right)=P(-q\leq t_{K-1}\leq q)+O(n^{-1})$.
\end{itemize}
The same result holds if $W_{SJ}$ is replaced by $W_{S}$, $W_{SB}$ or $W_{B}$ and the coefficients $c_j^{(SJ,K)},j=1,2,\dots,r$ are replaced with a different set of coefficients corresponding to each method.
\end{thm}

Theorem \ref{thm: validity} gives results similar to Theorem \ref{thm: batching_expansion}, but has different assumptions and works for all batching methods. The assumptions on the Cramer's condition, finiteness of moments, and the differentiability of $f$ are common in the literature of Edgeworth expansion. The assumption that seems nontrivial is the requirement $K\geq r+3$. In our proof, this condition
is used to guarantee the finiteness of some terms that involve the
moment of an inverse chi-squared distribution with $K-1$ degrees of
freedom. This condition may not be necessary in some cases, e.g., the simple examples in Appendix \ref{sec: explicit examples}. However, we also find examples where the claim of Theorem \ref{thm: validity} does not hold when $K$ is too small. The following proposition reveals that when $K=2$, the magnitude of error for sectioning in the symmetric case is indeed larger than $n^{-1}$. This implies that for Theorem \ref{thm: validity} with $r=1$ to hold, $K$ should be at least 3 (unless there are other additional assumptions). 
\begin{prop}[Magnitude of error when $K=2$]\label{prop: errorK2}
Let $\psi(\bar{P})=f(E_{\bar{P}}[X],E_{\bar{P}}[{Y}]),f(x,y)=x+y^{2}$
and $(X,Y)\sim N(0,2I_{2})$ under $P$ where $I_2$ is the 2-dimensional identity matrix,  
$$\left| P(-q\leq W_S\leq q)-P(-q\leq t_1\leq q)\right|=\omega(n^{-1})$$ but $$\left| P(-q\leq W_S\leq q)-P(-q\leq t_1\leq q)\right|=o(n^{-1/2}).$$ 
\end{prop}
We also remark that in Theorem \ref{thm: validity} the condition $K\geq r+3$ could be potentially relaxed to $K\geq r+2$ if we only need a residual $o(n^{-r/2})$ instead of $O(n^{-(r+1)/2})$. This could be done by conducting all the expansions in our proofs with reminder terms of the Peano type. A detailed argument of this would require some delicate but similar analysis as this paper.

\section{Extensions to dependent data}\label{sec: dependent}

Following from the discussion before Theorem \ref{thm: validity},
for the smooth function model, as long as the (scaled) batch averages
have a valid joint Edgeworth
expansion and other regularity conditions in Theorem \ref{thm: validity}
hold, we can show the validity of Edgeworth expansion for
the resulting statistic. In what follows, we extend Theorem \ref{thm: validity} to the setting where the data is a dependent sequence.

Consider a stationary dependent sequence (e.g., a Markov chain) $X_1,X_2,\dots$ and suppose we are interested in estimating $f(Eg(X_1))$ where the expectation is taken under the stationary measure and $g$ maps from the the state space of the data sequence to $\mathbb{R}$. Here we introduce $g$ because the Edgeworth expansion literature only works for the average of $g(X_i)$'s (which is one-dimensional) but not for the average of $X_i$'s (which can be multidimensional). Suppose we divide the data into batches each having $n$ samples. Let $\bar{g}_k$ denote the average of $g(X_i)$ among the $X_i$'s in the $k$-th batch. For this model, under some conditions, Edgeworth expansion for the distribution of $\sqrt{n}(\bar{g}_1-Eg(X_1))$ is known (\citealt{jensen1989asymtotic,malinovskii1987limit}). However, this is not sufficient to inform the joint distribution of $(\sqrt{n}(\bar{g}_1-Eg(X_1)),\dots,\sqrt{n}(\bar{g}_K-Eg(X_1)))$ since we do not know the dependence among batches. Indeed, noting that
$Cor(g(X_{n}),g(X_{n+1}))$ (here $Cor$ means correlation) would contribute $\Phi(n^{-1})$ to $Cor\left(\sqrt{n}\left(\bar{g}_{1}-Eg(X_1)\right),\sqrt{n}\left(\bar{g}_{2}-Eg(X_1)\right)\right)$,
we expect that $Cor\left(\sqrt{n}\left(\bar{g}_{1}-Eg(X_1)\right),\sqrt{n}\left(\bar{g}_{2}-Eg(X_1)\right)\right)$ is of order $n^{-1}$. Therefore, the dependence among batches would contribute $n^{-1}$ to the Edgeworth expansion for the joint distribution  which is not negligible for the higher-order coverage error. Without further information, it seems challenging to explicitly analyze the effect of dependence among batches. For this reason, we consider some modified schemes where different batch averages are independent or nearly independent. 

\subsection{Approach 1: Leave a gap between successive batches}

Suppose that, between each two adjacent batches, we discard $n^{\delta}$
data for some $0<\delta<1$ (for convenience, suppose that after this
operation, each batch has $n$ data). Intuitively, with this gap, there is more independence between adjacent batches. After this operation, we construct the CIs in the same way as introduced in Section \ref{sec: methods}. More precisely, following the convention therein, let $\hat{P}_{i},i=1,2,\dots,K$ denote the empirical distribution of $X_j$'s in the $i$-th batch (after leaving the gap). Let $\hat{P}=\frac{1}{K}\sum_{i=1}^K\hat{P}_i$ and $\hat{P}_{(i)} = \frac{1}{K-1}\sum_{1\leq k\leq K,k\neq i}\hat{P}_{k}$. Let $\psi(\cdot)$ be defined as $\psi({\bar{P}})=f(E_{X\sim {\bar{P}}} g(X))$. With these, we construct the CIs $CI_B,CI_S,CI_{SB},CI_{SJ}$ and consider the associated statistics $W_{SJ},W_{S},W_{B},W_{SB}$ in terms of $\psi(\cdot),\hat{P},\hat{P}_{i},\hat{P}_{(i)},i=1,2,\dots,K$ using the same formulas as given in Section \ref{sec: methods}. 


To proceed, we introduce the notion of mixing coefficient and Harris recurrence.
\begin{definition}[Mixing coefficient]
For any two Borel $\sigma$-fields $\mathcal{A}$ and $\mathcal{B}$, define the mixing coefficient $\alpha(\mathcal{A},\mathcal{B}):=\sup_{A\in\mathcal{A},B\in\mathcal{B}}\left|P(A\cap B)-P(A)P(B)\right|$. For a stationary sequence X with $\mathcal{F}_j^l:=\sigma(X_j,\dots,X_l)$, define $\alpha(n):=\sigma(\mathcal{F}_{-\infty}^0,\mathcal{F}_{n}^{\infty})$.
\end{definition}
\begin{definition}[Harris recurrence]\label{def: Harris}
A Markov chain $X$ defined on state space $\mathcal{X}$ is said to be Harris recurrent if there exists a measure $\mu$ such that for any measurable set $B\subset\mathcal{X}$ satisfying $\mu(B)>0$, we have $P_x(X_n\in B,i.o.)=1,\forall x\in\mathcal{X}$.
\end{definition}
 With the Edgeworth expansion proposed in Theorem 1 of \citet{jensen1989asymtotic}, we can show the following validity of the coverage error expansion: 
\begin{thm}[Coverage error expansions for batching methods on dependent data bearing recurrent atoms with gaps between successive batches] \label{thm: validity dependent}
Suppose that $X_{i}$, $i=1,2,\dots$ is a Harris-recurrent and strictly stationary
Markov chain with a recurrent atom $A_{0}$ and stationary distribution $P$. Moreover, suppose
that 
$n^{\frac{r+1}{2\delta}}\alpha(n)\rightarrow 0$ as $n\rightarrow 0$ for some positive integer $r$. 
Let
\[
\tau=\min\{n>0:X_{n}\in A_{0}\},G=\sum_{i=1}^{\tau}\left|g(X_{i})\right|
\]
Suppose the uniform Cramer condition holds: there exists $\delta^{\prime}<1$
such that $\left|E_P\exp(iuX_{1}+iv\tau)\right|<\delta^{\prime}$ for all
$v\in\mathbb{R}$ and all $u\in\mathbb{R}^{d}$ with $\left\Vert u\right\Vert >c$.
Suppose further that $E\tau^{r+3}<\infty,EG^{r+3}<\infty$.
Suppose that $f$ is $r$ times differentiable in a neighborhood of
$E_Pg(X_1)$ and $\nabla f(E_Pg(X_1))\neq 0$. Moreover, suppose that $K\geq r+3$. Then for any $q\in\mathbb{R}$, there exists $\bar{c}_j^{(SJ,K)}\in\mathbb{R},j=1,2,\dots,r$, such that  $$P(W_{SJ}\leq q)=P(t_{K-1}\leq q)+\sum_{j=1}^rn^{-j/2}\bar{c}_j^{(SJ,K)} + O(n^{-(r+1)/2})$$
    Here the coefficients $\bar{c}_{j}^{(SJ,K)}$ depends on $K$, the stationary distribution $P$, the objective function $g$ and the value of $q$, but does not depend on $n$.
The same result holds if $W_{SJ}$ is replaced by $W_{S}$, $W_{SB}$ or $W_{B}$ and the coefficients $\bar{c}_j^{(SJ,K)},j=1,2,\dots,r$ are be replaced with a different set of coefficients corresponding to each method.
\end{thm}

Theorem \ref{thm: validity dependent} stipulates that for a dependent data sequence satisfying proper conditions, the coverage error also has a valid expansion similar to Theorems \ref{thm: batching_expansion} and \ref{thm: validity}. In Theorem \ref{thm: validity dependent}, we didn't make a claim that the coverage error of a symmetric CI is of order $O(n^{-1})$ since the expansion in \cite{jensen1989asymtotic} doesn't discuss the oddness and evenness of the polynomials in the Edgeworth expansion. But as long as we have this oddness and evenness, we can show that the coverage error for a symmetric CI is order $O(n^{-1})$ (see the second claim of Theorem \ref{thm: validity dependent 2} next). 

We discuss the conditions introduced in Theorem \ref{thm: validity dependent}. The Harris recurrence condition can be seen as a generalization
of positive recurrent Markov chains with countable states while preserving
some good properties for a rich theory. This condition can be easily verified in some settings beyond countable
state space: for example, for the waiting time process in a queue, we can let $\mu$ in Definition \ref{def: Harris} (and also the recurrent atom as required in Theorem \ref{thm: validity dependent}) be concentrated on the state
``0'', then we can verify the Harris recurrence as long as ``0'' is a recurrent state. Similar setups using positive recurrent Harris
chains also appear in IV.6b of \citet{Asmussen2007}. 
More generally, the Doeblin chains can be seen as special cases of
recurrent Harris chains \citet{malinovskii1987limit}. Practical examples
include M/G/1  waiting time processes, AR(1)
processes, and more general queueing systems
as considered in the numerical examples of \citet{chien1994batch,alexopoulos1996implementing,steiger2002improved,alexopoulos2013sequential}.
Conditions that can guarantee the mixing condition $n^{\frac{r+1}{2\delta}}\alpha(n)\rightarrow 0$ in the setting of Harris recurrent Markov chains can be found
in  Section 3.2 of \citet{duchi2021statistics}. Under those conditions, we actually have $\alpha(n)=c^n$ for some $c\in(0,1)$, so $n^{\frac{r+1}{\delta}}\alpha(n)\rightarrow 0$ holds for arbitrary $\delta>0$. Other assumptions on the finiteness
of moments and Cramer condition are common in the study of Edgeworth
expansion.

The following theorem does not require the existence of
atom as in Theorem \ref{thm: validity dependent}, but involves more complex notations in the statement. Moreover, we are able to conclude that the coverage error for a symmetric CI is of order $O(n^{-1})$ thanks to the knowledge of the oddness and evenness of the polynomials in the Edgeworth expansion in \cite{malinovskii1987limit}. 
\begin{thm}[Coverage error expansions for batching methods on dependent data with gaps between successive batches]\label{thm: validity dependent 2}
Suppose that $X_{i}$, $i=1,2,\dots$ is a Harris-recurrent Markov chain and is strictly stationary. Moreover, suppose that $n^{\frac{r+1}{2\delta}}\alpha(n)\rightarrow 0$ for some positive integer $r$. 
Let $P$ be the stationary distribution. Suppose that there exist
a set $C$, a positive $\lambda$ and a probability measure $\varphi_{C}$
concentrated on $C$ such that
\[
P_{x}\left( \cup_{i=1}^{\infty}\{X_{i}\in C\}\right) =1
\]
for any $x$ and 
\[
P_{x}(X_{2}\in A)\geq\lambda\varphi_{C}(A)
\]
 for any $x\in C,A\subset C$. For any distribution $\alpha$,
define $P_{\alpha,\lambda}$ as the measure of the process $\{(X_i,b_i),i=1,2,\dots\}$ with initial distribution
\[
P_{\alpha,\lambda}(X_{1}\in dx,b_{1}=\delta)=\alpha(dx)(\lambda\delta+(1-\lambda)(1-\delta))
\]
and transition probability
\[
P_{\alpha,\lambda}\left((x,1),A\times\{\delta\}\right)=\begin{cases}
(\lambda\delta+(1-\lambda)(1-\delta))P(x,A) & x\notin C\\
(\lambda\delta+(1-\lambda)(1-\delta))\varphi_{C}(A) & x\in C
\end{cases}
\]
\[
P_{\alpha,\lambda}\left((x,0),A\times\{\delta\}\right)=\begin{cases}
(\lambda\delta+(1-\lambda)(1-\delta))P(x,A) & x\notin C\\
(\lambda\delta+(1-\lambda)(1-\delta))Q(x,A) & x\in C
\end{cases}
\]
where $P(x,\dot)$ is the transition probability of $X_i,i=1,2,\dots$ and $Q(x,\cdot)=(1-\lambda)^{-1}(P(x,\cdot)-\lambda\varphi_{C}(\cdot))$.
Let
\[
\tau=\min\{n>0:X_{n}\in C,b_{n}=1\},\tilde{g}(x)=g(x)-Eg(X_{1}),\Sigma_{g,n}=n^{-1/2}\sum_{i=1}^{n}\tilde{g}(X_{i}),G=\sum_{i=1}^{\tau}\tilde{g}(X_{i}).
\]
Suppose that

(i) $\limsup_{\left|t\right|\rightarrow\infty}\left|E_{\varphi_{C},\lambda}\exp(itG)\right|<1$.

(ii) $\sigma_{G}:=E_{\varphi_{C},\lambda}G^{2}>0$

(iii) $E_{\varphi_{C},\lambda}\tau^{r+3}<\infty,E_{P,\lambda}\tau^{r+1}<\infty$

(iv) $E_{\varphi_{C},\lambda}\left(\sum_{i=1}^{\tau}\left|g(X_{i})\right|\right)^{r+3}<\infty,E_{P,\lambda}\left(\sum_{i=1}^{\tau}\left|g(X_{i})\right|\right)^{r+1}<\infty$

Then,
\begin{itemize}
    \item If $K\geq r+3$, then for any $q\in\mathbb{R}$, there exists $\bar{\bar{c}}_j^{(SJ,K)}\in\mathbb{R},j=1,2,\dots,r$, such that  $$P(W_{SJ}\leq q)=P(t_{K-1}\leq q)+\sum_{j=1}^rn^{-j/2}\bar{\bar{c}}_j^{(SJ,K)} + O(n^{-(r+1)/2})$$
    Here the coefficients $\bar{\bar{c}}_{j}^{(SJ,K)}$ depends on $K$, the stationary distribution $P$, the objective function $g$ and the value of $q$, but does not depend on $n$.
\item  Suppose that $K\geq 4$. Then $P\left(-q\leq W_{SJ}\leq q\right)=P(-q\leq t_{K-1}\leq q)+O(n^{-1})$.
\end{itemize}
The same result holds if $W_{SJ}$ is replaced by $W_{S}$, $W_{SB}$ or $W_{B}$ and the coefficients $\bar{\bar{c}}_j^{(SJ,K)},j=1,2,\dots,r$ are be replaced with a different set of coefficients corresponding to each method.

\end{thm}

While the choices of $C$ and $\lambda$ in the statement of Theorem \ref{thm: validity dependent 2} seem to have some freedom,
it is shown in \citet{malinovskii1987limit} that any choice would
lead to the same Edgeworth expansion. 

\subsection{Approach 2: Decompose into regenerative cycles}

Suppose that $X_{1},X_{2},\dots$ has a recurrent state $a_{0}$.
Let $T_{1}=\inf\left\{ k>0:X_{k}=a_{0}\right\} $ and for $i\geq2$,
let $T_{i}=\inf\left\{ k>T_{i-1}:X_{k}=a_{0}\right\} $. Let $Y_{i}=\sum_{k=T_{i}}^{T_{i+1}-1}g(X_{k})$.
We may regard $\left(Y_{i},T_{i+1}-T_{i}\right)$
as the new data and perform batching for these data. For example,
suppose that we are interested in $E_{\pi}g(X_{1})$, where $\pi$ is the stationary distribution. Note that  $E_{\pi}g(X_{1})=\frac{EY_{i}}{E\left(T_{i+1}-T_{i}\right)}$ which belongs to
the family of smooth function models. Therefore, if we suppose that
$\left(Y_{i},T_{i+1}-T_{i}\right)$ satisfies the same assumptions
as $X_{i}$ in Theorem \ref{thm: validity} and other conditions therein are satisfied, then all the claims in Theorem
\ref{thm: validity} would hold. 

More concretely, we construct the CIs for $\psi:=E_{\pi}g(X_{1})$ in the following way. Let $Q_i:=\left(Y_{i},T_{i+1}-T_{i}\right),i=1,2,\dots,nK$. Let $\hat{P},\hat{P}_i$ and $\hat{P}_{(i)}$ denote the entire, batched and leave-one-batch-out empirical distributions of $\{Q_i\}_{i=1}^{nK}$ as introduced in Section \ref{sec: methods}. For any two-dimensional distribution ${\bar{P}}$, we define $\psi({\bar{P}})=\frac{E_{\bar{P}} [Q^{(1)}]}{E_{\bar{P}}[Q^{(2)}]}$ where $Q\sim {\bar{P}}$ and for $j=1,2$, $Q^{(j)}$ stands for the $j$-th coordinate of $Q$. Note that the target value can be written as $\psi=\frac{EQ_1^{(1)}}{EQ_1^{(2)}}$ (where the expectation is taken under the true distribution of $Q_1$). Construct $CI_B,CI_S,CI_{SB},CI_{SJ}$ and define the corresponding statistics $W_{B},W_{S},W_{SB},W_{SJ}$ in terms of $\psi(\cdot),\hat{P},\hat{P}_i,\hat{P}_{(i)}$ using the same formulas as in Section \ref{sec: methods}. 

We have the following theorem regarding the validity of expansion:
\begin{thm}[Coverage error expansions for batching methods on dependent data using regenerative cycles]\label{thm: validity regenerate}
 Suppose that 
the following Cramer's condition
holds for the distribution of $Q_i$: 
\begin{equation}\label{eq: cramerQ}
\limsup_{\left|t\right|\rightarrow\infty}\left|E[\exp\{i\left<t,Q_1\right>\}]\right|<1,
\end{equation}
Suppose that for some positive integer $r$, $Q_1$ has finite moments up to order
$r+2$ with nonsingular covariance. Suppose that $E[T_2-T_1]>0$.  Then:
\begin{itemize}
    \item If $K\geq r+3$, then for any $q\in\mathbb{R}$, there exists $\tilde{c}_j^{(SJ,K)}\in\mathbb{R},j=1,2,\dots,r$, such that  $$P(W_{SJ}\leq q)=P(t_{K-1}\leq q)+\sum_{j=1}^rn^{-j/2}\tilde{c}_j^{(SJ,K)} + O(n^{-(r+1)/2})$$
    Here the coefficients $\tilde{c}_{j}^{(SJ,K)}$ depends on $K$, the true distribution of $Q_1$, and the value of $q$, but does not depend on $n$.
\item Suppose that $K\geq 4$. Then we have $P\left(-q\leq W_{SJ}\leq q\right)=P(-q\leq t_{K-1}\leq q)+O(n^{-1})$.
\end{itemize}
The same result holds if $W_{SJ}$ is replaced by $W_{S}$, $W_{SB}$ or $W_{B}$ and the coefficients $\tilde{c}_j^{(SJ,K)},j=1,2,\dots,r$ are be replaced with a different set of coefficients corresponding to each method.
\end{thm}



\section{An algorithm to estimate coefficients of the $n^{-1}$ error via simulation}\label{sec: alg}

For this section, we mainly focus on the symmetric CIs whose errors are of order $n^{-1}$ by Theorem \ref{thm: validity}. Consider the smooth function model where we want to construct a CI for $f(EX)$. Recall that Theorem \ref{thm: validity} implies that 
\begin{equation}\label{eq: c def}
    P(-q\leq W_{\cdot}\leq q)=P(-q\leq t_{K-1}\leq q) + cn^{-1} + O(n^{-3/2})
\end{equation}
where $W_{\cdot}$ can be either one of $W_S,W_B,W_{SB}$ and $W_{SJ}$ and $c$ is a constant that depends on the method used, the number of batches $K$, the underlying distribution $P$, the function $f$, and the critical point $q$, but does not depend on $n$. For some very simple examples, we can calculate $c$ explicitly (see Section \ref{sec: explicit examples}). However, in general, $c$ involves some integration that cannot be calculated explicitly. In this section, we give
a simulation scheme to generate an unbiased estimator for $c$ for each batching method. Our proposed algorithm is Algorithm \ref{alg: error estimation} whose construction follows closely from the proof of Theorem \ref{thm: validity}. The polynomials $p_1$ and $p_2$ in Algorithm \ref{alg: error estimation} are defined as follows, which are due to \cite{skovgaard1986multivariate}:
\begin{equation}\label{eq: Edgeworth formula}
    \begin{array}{l}
    p_1(x)=\phi_{\Sigma}(x)\left[1+\frac{1}{6} \chi_{i j k} \sigma^{i \alpha} \sigma^{j \beta} \sigma^{k \gamma} x_{\alpha} x_{\beta} x_{\gamma}-\frac{1}{2} \sigma^{i j} \chi_{i j k} \sigma^{k \alpha} x_{\alpha}\right]\\
p_2(x)=\phi_{\Sigma}(x)\left[\frac{1}{24} \chi_{i j k l}\left\{\sigma^{i \alpha} \sigma^{j \beta} \sigma^{k \gamma} \sigma^{l \delta} x_{\alpha} x_{\beta} x_{\gamma} x_{\delta}-6 \sigma^{i \alpha} \sigma^{j \beta} \sigma^{k l} x_{\alpha} x_{\beta}+3 \sigma^{i j} \sigma^{k l}\right\}\right. \\
\quad+\frac{1}{72} \chi_{i j k} \chi_{l m n}\left\{\sigma^{i \alpha} \sigma^{j \beta} \sigma^{k \gamma} \sigma^{l \delta} \sigma^{m \varepsilon} \sigma^{n \tau} x_{\alpha} x_{\beta} x_{\gamma} x_{\delta} x_{\varepsilon} x_{\tau}\right. \\
\quad-\left(6 \sigma^{i \alpha} \sigma^{j \beta} \sigma^{k \gamma} \sigma^{l \delta} \sigma^{m n}+9 \sigma^{i \alpha} \sigma^{j \beta} \sigma^{l \gamma} \sigma^{m \delta} \sigma^{k n}\right) x_{\alpha} x_{\beta} x_{\gamma} x_{\delta} \\
\quad+\left(9 \sigma^{i \alpha} \sigma^{l \beta} \sigma^{j k} \sigma^{m n}+18 \sigma^{i \alpha} \sigma^{j \beta} \sigma^{k l} \sigma^{m n}+18 \sigma^{i \alpha} \sigma^{l \beta} \sigma^{j m} \sigma^{k n}\right) x_{\alpha} x_{\beta} \\
\left.\left.\quad-\left(9 \sigma^{i j} \sigma^{k l} \sigma^{m n}+6 \sigma^{i l} \sigma^{j m} \sigma^{k n}\right)\right\}\right]
\end{array}
\end{equation}

Here $\phi_{\Sigma}(\cdot)$ is the density of $N(0,\Sigma)$, $\sigma^{ij}$ are the coordinates of matrix $\Sigma^{-1}$, and $\chi_{ijk}$ and $\chi_{ijkl}$ stand for the joint cumulants
of the coordinates of $X$. The Einstein summation convention is used here, which means we sum over repeated indices. For example, we omit the summation symbol $\sum_{1\leq i,j,k,\alpha,\beta,\gamma\leq d}$ when we write $\chi_{i j k} \sigma^{i \alpha} \sigma^{j \beta} \sigma^{k \gamma} x_{\alpha} x_{\beta} x_{\gamma}$. 

The following proposition asserts the correctness of the algorithm. The proof uses similar techniques as the proof of Theorem \ref{thm: validity}, but with more explicit algebra.

\begin{thm}[Unbiasedness of simulation algorithm to estimate coefficients of $n^{-1}$]\label{prop: alg correctness}
Suppose that the conditions of Theorem \ref{thm: validity} hold with $r=2$, and the expansion is given as \eqref{eq: c def} where $W_\cdot$ can be any of $W_S$, $W_{B}$, $W_{SB}$ and $W_{SJ}$. Then, Algorithm \ref{alg: error estimation} returns an unbiased estimator for $c$.
\end{thm}

\begin{algorithm}[htp]
\caption{An unbiased simulation scheme to compute the coefficient of the $n^{-1}$ error term}
\label{alg: error estimation}
\begin{algorithmic}[1]
\Ensure Derivatives of $f$ at $EX_1$ up to order 3, cumulants of $X_1$ up to order 4, testing statistic $W\cdot$ ($W_\cdot$ can be either of $W_S,W_B,W_{SB}$ and $W_{SJ}$ and is regarded as a function of batch averages and $n$)

\State Let $\Sigma=Var(X_1)$. Simulate $X_{1},X_{2},\dots,X_{K}\stackrel{\text{i.i.d.}}{\sim}N(0,\Sigma)$.
Let $A_{0}=\frac{X_{1}+\dots+X_{K}}{K},B_{i}=X_{i}-A_{0},i=1,2,\dots,K$. 
\State Perform Taylor's expansion on $W_\cdot$ regarding $X_1,\dots,X_K$ as the batch averages,
and write it as $W_\cdot=\sqrt{K(K-1)}\frac{a+n^{-1/2}b_{1}+n^{-1}b_{2}}{\sqrt{E_{2}+n^{-1/2}\lambda+n^{-1}e}}+O_p(n^{-3/2})$.
Express each of $a, b_{1},b_{2},\lambda,e,E_{2}$ as a polynomial of $A_{0},B_{1},B_{2},\dots,B_{K}$ (the expressions depend on the batching scheme we use. The expressions for sectioning is included in Section \ref{sec: alg}. The expressions for other schemes are given in Section \ref{sec: expansion other schemes}). 
\State Compute the derivative of $b_{1}$ and $\lambda$
w.r.t. $A_{0,1}$ and denote them by $b_{1}^{\prime}$ and $\lambda^{\prime}$ respectively. Then let (denote $u=\nabla f(EX_1)$)
\[
F_{x}=\frac{b_{1}}{\sqrt{E_{2}}}-\frac{1}{2}\frac{\lambda a}{E_{2}\sqrt{E_{2}}},F_{xx}=\frac{1}{\sqrt{E_{2}}}\left(a\left(-\frac{e}{E_{2}}+\frac{3}{4}\frac{\lambda^{2}}{E_{2}^{2}}\right)-b_{1}\frac{\lambda}{c}+2b_{2}\right),
\]
\[
F_{y}=\frac{a^{\prime}}{\sqrt{E_{2}}},F_{xy}=\frac{b_{1}^{\prime}}{\sqrt{E_{2}}}-\frac{1}{2}\frac{\left(\lambda a\right)^{\prime}}{E_{2}\sqrt{E_{2}}},
\]
\[
F_{+}=\frac{q\sqrt{E_2}}{\sqrt{K(K-1)}}-\sum_{i=2}^d u_i A_{0,i}, F_{-}=\frac{-q\sqrt{E_2}}{\sqrt{K(K-1)}}-\sum_{i=2}^d u_i A_{0,i}.
\]
\State Compute
\[
y_{x}=-F_{x}/F_{y}|_{A_{0,1}=F_{+}},y_{xx}=-\left(F_{xx}+2F_{xy}y_{x}\right)/F_{y}|_{A_{0,1}=F_{+}},
\]
\[
y_{x}^{(-)}=-F_{x}/F_{y}|_{A_{0,1}=F_{-}},y_{xx}^{(-)}=-\left(F_{xx}+2F_{xy}y_{x}\right)/F_{y}|_{A_{0,1}=F_{-}}.
\]
\State Derive the error term estimator:
\begin{align*}
 ER = & I_C(A)\left[\sum_{1\leq i<j\leq K}p_{1}(x_{i}(A))p_{1}(x_{j}(A))+\sum_{1\leq i\leq K}p_{2}(x_{i}(A))\right]\\
 & \left(\begin{array}{c}
\phi_{\tilde{\sigma}_{0}}(F_{+}-\mu)\left(p_{1}(x(A))\right)|_{A_{0,1}=F_{+}}\left(y_{x}\right)\\
-\phi_{\tilde{\sigma}_{0}}(F_{-}-\mu)\left(p_{1}(x(A))\right)|_{A_{0,1}=F_{-}}\left(y_{x}^{(-)}\right)
\end{array}\right)+\left(\begin{array}{c}
\phi_{\tilde{\sigma}_{0}}(F_{+}-\mu)\left(\frac{1}{2}y_{xx}\right)\\
-\phi_{\tilde{\sigma}_{0}}(F_{-}-\mu)\left(\frac{1}{2}y_{xx}^{(-)}\right)
\end{array}\right)\\
 & +\frac{1}{2}\left(\begin{array}{c}
\phi_{\tilde{\sigma}_{0}}(F_{+}-\mu)(-\left(F_{+}-\mu\right)/\tilde{\sigma}_{0})y_{x}^{2}\\
-\phi_{\tilde{\sigma}_{0}}^{\prime}(F_{-}-\mu)(-\left(F_{-}-\mu\right)/\tilde{\sigma}_{0})y_{x}^{(-)2}
\end{array}\right)
\end{align*}
Here, $(x_{1}(A),\dots,x_{K}(A))^{\top}=\left(\begin{array}{ccccc}
1 & 1\\
1 &  & 1\\
\dots &  &  & \dots\\
1 &  &  &  & 1\\
1 & -1 & -1 & \dots & -1
\end{array}\right)\left(\begin{array}{c}
A_{0}\\
B_{1}\\
\dots\\
B_{K-2}\\
B_{K-1}
\end{array}\right)$ and $C$ represents $-q\leq\sqrt{K(K-1)}\frac{[u,A_{0}]}{\sqrt{\sum_{i=1}^{K}\left(\begin{array}{c}
[u,A_{i}-A_{0}]\end{array}\right)^{2}}}\leq q$. The polynomials $p_1$ and $p_2$ are given in \eqref{eq: Edgeworth formula}. $\tilde{\sigma}_{0}=\left(\sigma_{0}-\sigma_{01}\sigma_{11}^{-1}\sigma_{10}\right)/K$,
$\sigma_{0}$ is the variance of the first coordinate of $X_1$, $\sigma_{11}$ is the variance of the last $d-1$ coordinates of $X_1$, and $\sigma_{01}$ is their covariance. $\mu=\sigma_{01}\sigma_{11}^{-1}A_{0}^{\prime}$ where $A_0^\prime$ is the vector of the last $d-1$ coordinates of $A_0$. $\phi_{\sigma}(\cdot)$ is the density of $N(0,\sigma)$.

\lastcon{An unbiased estimator of $c$ in \eqref{eq: c def} given by $ER$
}
\end{algorithmic}
\end{algorithm}

In what follows, we illustrate how to compute the polynomials $a,b_1,b_2,d,e,E_2,b_1^{\prime},d^{\prime}$ as required by Steps 2 and 3 of Algorithm \ref{alg: error estimation} using sectioning as an example. The computation for other schemes can be found in Section \ref{sec: expansion other schemes}.

Let $u=\nabla f,v=\nabla^{2}f/2,w=\nabla^{3}f/6$. Recall the definition of $A_0,B_i,i=1,2,\dots,K$ in Algorithm \ref{alg: error estimation}. Then 
\begin{align*}
W_{S} & =\frac{\sqrt{nK}\left(f\left(m+n^{-1/2}A_{0}\right)-f_{0}\right)}{\sqrt{\frac{1}{K-1}\sum_{i=1}^{K}\left(\sqrt{n}f\left(m+n^{-1/2}X_{i}\right)-\sqrt{n}f\left(m+n^{-1/2}A_{0}\right)\right)^{2}}}\\
 & =\sqrt{K(K-1)}\frac{[u,A_{0}]+n^{-1/2}[v,A_{0},A_{0}]+n^{-1}\left[w,A_{0},A_{0},A_{0}\right]}{\sqrt{\sum_{i=1}^{K}\left(\begin{array}{c}
[u,A_{i}-A_{0}]+n^{-1/2}[v,A_{i},A_{i}]-n^{-1/2}[v,A_{0},A_{0}]\\
+n^{-1}\left[w,A_{i},A_{i},A_{i}\right]-n^{-1}\left[w,A_{0},A_{0},A_{0}\right]
\end{array}\right)^{2}}}+O_p(n^{-3/2})
\end{align*}
Here, $[u,A_{m}];=u_{i}A_{m,i},[v,A_{m},A_{m}]=:v_{ij}A_{m,i}A_{m,j},[w,A_{m},A_{m},A_{m}]=:w_{ijk}A_{m,i}A_{m,j}A_{m,k}$.
The leading term of the denominator is $\sum_{i}[u,A_{i}-A_{0}]^{2}$.
The coefficient for $n^{-1/2}$ is 
\[
\lambda=2\sum_{i}[u,B_{i}]\left([v,B_{i}+A_{0},B_{i}+A_{0}]-[v,A_{0},A_{0}]\right)
\]
The coefficient for $n^{-1}$ is 
\[
e=\sum_{i}\left[\begin{array}{c}
     \left([v,B_{i}+A_{0},B_{i}+A_{0}]-[v,A_{0},A_{0}]\right)^{2}  \\
     +2[u,B_{i}]\left(\left[w,B_{i}+A_{0},B_{i}+A_{0},B_{i}+A_{0}\right]-\left[w,A_{0},A_{0},A_{0}\right]\right) 
\end{array}\right]
\]
We also have $a=[u,A_{0}],b_{1}=[v,A_{0},A_{0}],b_{2}=\left[v,A_{0},A_{0},A_{0}\right]$. 
The derivatives w.r.t. $A_{0,1}$ can be computed as
\[
b_{1}^{\prime}=2v_{11}A_{0,1}+2\sum_{j=2}^{d}v_{1j}A_{0,j}=2\sum_{j=1}^{d}v_{1j}A_{0,j}
\]
\[
\lambda^{\prime}=2\sum_{i}[u,B_{i}]2\left(\sum_{j=1}^{d}v_{1j}\left(B_{i,j}+A_{0,j}\right)-\sum_{j=1}^{d}v_{1j}A_{0,j}\right)=4\sum_{i}[u,B_{i}]\sum_{j=1}^{d}v_{1j}B_{i,j}.
\]
This gives all polynomials required by Steps 2 and 3 of Algorithm \ref{alg: error estimation}.

We also point out that when $K=2$, the proposed algorithm may not work for batching, although in this case the error of batching has a valid expansion by Theorem \ref{thm: batching_expansion}.  This is described in the next proposition.
\begin{prop}\label{prop: batching small K}
Consider the same setting as in Proposition \ref{prop: errorK2}. In this case, Algorithm \ref{alg: error estimation} does not give the correct coefficient of the $n^{-1}$ coverage error term of batching.
\end{prop}
For batching (only), we can use an alternative algorithm that leverages Theorem \ref{thm: batching_expansion}. This algorithm, which is detailed in Appendix \ref{sec: batching algorithm}, works regardless of the value of $K$. When $K\geq 5$, both this algorithm and Algorithm \ref{alg: error estimation} would give unbiased estimators for the coefficient of $n^{-1}$ error term, but Algorithm \ref{alg: error estimation} would have a smaller variance due to the conditioning argument in its construction.


\section{Asymptotic coverages as number of batches grows}\label{sec: asymptotic}
We close our theoretical developments with a discussion on the behaviors of batching methods when the number of batches $K$ grows. For a general $\psi(\cdot)$ and distribution $P$, we have the following theorem regarding the asymptotic as $K\rightarrow\infty$ when the number of samples in each batch is fixed:

\begin{thm}[Asymptotic coverages as number of batches grows]\label{thm: K_large}
Suppose that we are under the setting introduced in Section \ref{sec: methods}. In particular, we have i.i.d. data $X_1,\dots,X_{nK}$ drawn from $P$. Suppose that $\psi(\cdot)$ is continuously Gateaux differantiable at $P$ with influence function $IF$. Suppose that $E\psi(\hat{P}_{1})-\psi\neq 0$ (i.e., $\psi(\hat{P}_1)$ is a biased estimator of $\psi$),  $\text{Var}(\psi(\hat{P}_1))<\infty$, and $Var_P IF=\sigma^2,0<\sigma<\infty$.  Fix $n$ and let $K\rightarrow \infty$. Then for any $q>0$, $$P(-q\leq W_B\leq q)\rightarrow 0,$$ $$P(-q\leq W_S\leq q) \rightarrow \Phi\left(q\sqrt{nE(\psi(\hat{P}_1)-\psi)^2}/\sigma\right)-\Phi\left(-q\sqrt{nE(\psi(\hat{P}_1)-\psi)^2}/\sigma\right),$$ while $$P(-q\leq W_{SJ}\leq q)\rightarrow \Phi(q)-\Phi(-q).$$
\end{thm}

Theorem \ref{thm: K_large} implies that when $K$ is large, batching has a significant bias which could lead to an extremely small coverage. The sectioning statistic converges in distribution to a limit that is not standard normal, so its asymptotic coverage is different from the nominal level. In contrast, the SJ statistic converges to the standard normal, so SJ is the only method among all considered ones that is consistent in this regime. 
This observation is also consistent with the numerical findings in \cite{Nakayama2014confidence}, who considers $K=10$ and $K=20$ and observes severe under-coverage for batching when $K=20$ in some examples. 



Finally, suppose we fix the total number of data and let $K$ increases. In this case, for sectioning and batching, the performance could be bad because
the number of data in each batch could be extremely small. On the other hand, SJ it reduces to a usual jackknife. So SJ is evidently superior in this case.

\section{Numerical results}\label{sec: numerics}
In this section, we run experiments to validate our expansions. Moreover, based on the coefficients of the $n^{-1}$ error term and the derived theoretical coverage probabilities, we investigate the coverage performances of our considered batching methods and in terms of the number of batches in two regimes: when the number of data in each batch is fixed, and when the total data size is fixed.

\subsection{Validation of expansions and simulation algorithm}\label{subsec: validation}

In this subsection, we validate the correctness of our expansions and Algorithm \ref{alg: error estimation} by comparing the theoretical with the actual coverage probabilities. Moreover, we investigate the effect of the number of batches on coverage error and compare different batching methods. We consider smooth function models where we want to construct a CI for $f(EX)$. In what follows, we present the numerical findings for different choices of $f$ and the distribution of $X$.

\subsubsection{Normal}\label{subsec: validation normal}

Consider the model $f(EX)$ where $X\sim N(0,1)$ and $f(x) = x + x^2$. We estimate both the coverage errors given by our theoretical expansion
and empirical experiments. Recall that $c$ is the coefficient for the $n^{-1}$ error term defined in \eqref{eq: c def}. For the theoretical coverage
probabilities, we first estimate $c$ via $10^{4}$ replications of Algorithm \ref{alg: error estimation}. Then an estimate of the theoretical coverage is given by the nominal level plus $cn^{-1}$. On the other hand, the actual
coverage probabilities are estimated via $10^{6}$ experimental repetitions, where in each repetition we generate a new data set and compute the CI, and at the end we estimate the empirical coverages by taking the proportions of times where the CI covers the truth. 

Figure \ref{fig: fix_n} shows the result
with a fixed number of samples per batch $n=30$, and the number of batches $K$ ranging from 2 to 30 (note that although Theorem \ref{prop: alg correctness} requires $K\geq 5$ in general, for this example we also have the validity of expansion for $K=2,3,4$). We see that for each batching method, the
estimated theoretical coverage probabilities are close to the
estimated actual coverage probabilities. This validates the correctness of our estimation for the error. When $K$ becomes large, the coverage probability of batching decreases quickly, while SJ has the smallest error among the four methods. For example, when the nominal level is 80\% and $K=30$, the estimated coverage probability of batching is only about 70\%, which is 10\% smaller than the nominal level, while SJ has the smallest error. On the other hand, when $K$ is small, this observation does not hold. In fact, when $K\leq 5$, batching has the smallest error while SJ has the largest error, which is exactly the opposite of the observation when $K$ is large. In addition, we observe that the critical value of $K$ where the comparison among methods changes depends on the nominal level: When the nominal level is 80\%, batching has the largest coverage error when $K=10$; however, when the nominal level is 95\%, batching still has the smallest coverage error when $K=10$. Moreover, we can see that for each method, the theoretical coverage (and thus the coefficient $c$) is not monotone in $K$. For example, when the nominal level is 95\%, none of the curves is monotonically increasing or decreasing.

\begin{figure}[htbp]
\begin{subfigure}{0.32\textwidth} \includegraphics[width=1\textwidth]{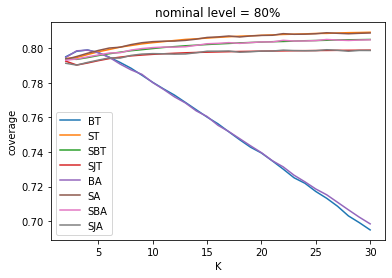}
\end{subfigure} \begin{subfigure}{0.32\textwidth} \includegraphics[width=1\textwidth]{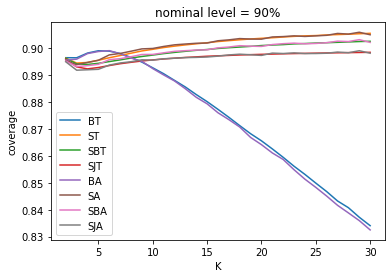}
\end{subfigure}
\begin{subfigure}{0.32\textwidth} \includegraphics[width=1\textwidth]{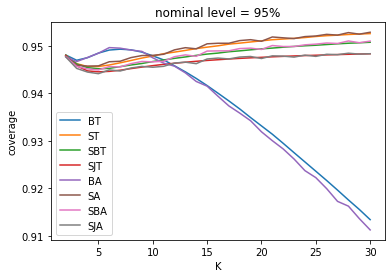}
\end{subfigure} 

\caption{Theoretical and actual coverage probabilities for the model in Section \ref{subsec: validation normal}. BT, ST, SBT, and SJT represents the estimated theoretical coverage probabilities of batching, sectioning, SB, and SJ, respectively. BA, SA, SBA, and SJA represents the estimated actual coverage probabilities of batching, sectioning, SB, and SJ, respectively. \label{fig: fix_n}}
\end{figure}

\subsubsection{Exponential and chi-square}\label{subsec: validation exp}

Consider the model $\psi({\bar{P}})=f(E_{\bar{P}}X,E_{\bar{P}}Y)$. Let $X\sim exp(1)-1$, $Y\sim\frac{\chi_{1}^{2}-1}{\sqrt{2}}$
and they are independent under the true distribution $P$. Let $f(x,y)=x+2y+y^{2}+x^{3}$. This is an example where the samples are multidimensional with non-normal distribution.  The cumulants
are given by (note that there are no cross-cumulants by independence)
\[
\kappa_{2}(X)=1,\kappa_{3}(X)=2,\kappa_{4}(X)=6,\kappa_{2}(Y)=1,\kappa_{3}(Y)=2\sqrt{2},\kappa_{4}(Y)=12
\]
We use $4\times 10^4$ replications of Algorithm \ref{alg: error estimation} to estimate the theoretical coverage probabilities, and $10^6$ experimental repetitions to estimate the actual coverage probabilities. We set $n=30$ and let $K$ range from 4 to 30. The results are shown in Figure \ref{fig: expchi}. The estimated theoretical coverage probabilities are again close to the actual coverage probabilities, although the differences appear slightly larger than the first example where the data is exactly normal. The increased differences are probably due to the need to estimate the error term induced by the non-normality of the underlying distribution. When $K$ is large, we observe that SJ has the smallest coverage error. In comparison, batching has more significant under-coverage, while sectioning and SB have significant over-coverage issues. For example, when the nominal level is 80\%, the estimated coverage probability of batching is around 77\%, the estimated coverage probabilities of sectioning and SB are above 82\%, while the estimated coverage probability of SJ is close to the nominal level. But as in the previous example, this observation may not hold when $K$ is small. For example, when the nominal level is 95\% and $K=5$, the estimated coverage probability of batching, sectioning and SB are between 94.75\% and 95\%, while the estimated coverage probability of SJ is smaller than 94.75\%, indicating that SJ has the largest error in this case.

\begin{figure}[htbp]
\begin{subfigure}{0.32\textwidth} \includegraphics[width=1\textwidth]{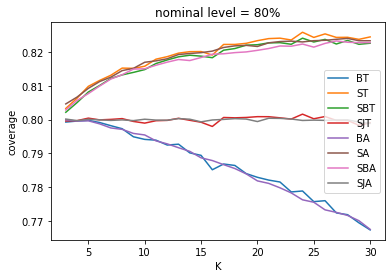}
\end{subfigure} \begin{subfigure}{0.32\textwidth} \includegraphics[width=1\textwidth]{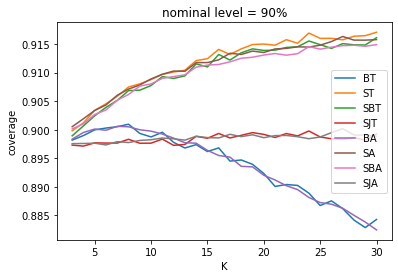}
\end{subfigure}
 \begin{subfigure}{0.32\textwidth} \includegraphics[width=1\textwidth]{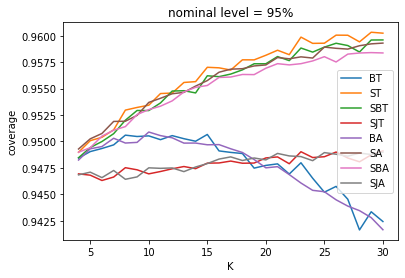}
\end{subfigure} 

\caption{Theoretical and actual coverage probabilities for the model in Section \ref{subsec: validation exp}. The definitions of BT, ST, etc are the same as in Figure \ref{fig: fix_n}.\label{fig: expchi}}
\end{figure}

\subsubsection{Normal and its square}\label{subsec: validation norm chi}

Again consider the model $\psi({\bar{P}})=f(E_{\bar{P}}X,E_{\bar{P}}Y)$. Let $X\sim N(0,1)$, $Y=\frac{X^{2}-1}{\sqrt{2}}$ under the true distribution $P$, and let $f(x,y)=sin(x+y^{2})$. This is an example where the coordinates of the samples are dependent and $f$ is not a polynomial.
The joint cumulants can be computed as follows:
\[
\kappa(X,X)=\kappa(Y,Y)=1,\kappa(X,Y)=0
\]
\[
\kappa(X,X,X)=\kappa(X,Y,Y)=0,\kappa(X,X,Y)=\sqrt{2},\kappa(Y,Y,Y)=2\sqrt{2}
\]
\[
\kappa(X,X,X,Y)=\kappa(X,Y,Y,Y)=0,\kappa(X,X,Y,Y)=4,\kappa(Y,Y,Y,Y)=12
\]
Other setups are the same as the previous example. The results are shown in Figure \ref{fig: normal_chi}. Again, the actual coverage probabilities are close to the estimated theoretical coverage probabilities. In this example, when $K$ is large, we also observe that SJ has the smallest error and batching has significant under-coverage issues, which is consistent with the findings in the previous two examples. A difference with the previous examples is that, for sectioning, SB, and SJ, the coverage probabilities do not change much when $K$ changes, which can be seen since the curves for these methods are close to horizontal.

\begin{figure}[htbp]
\begin{subfigure}{0.32\textwidth} \includegraphics[width=1\textwidth]{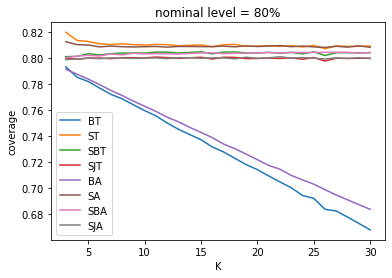}
\end{subfigure} \begin{subfigure}{0.32\textwidth} \includegraphics[width=1\textwidth]{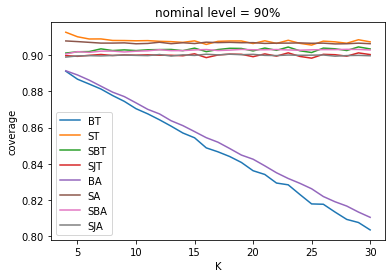}
\end{subfigure}
\begin{subfigure}{0.32\textwidth} \includegraphics[width=1\textwidth]{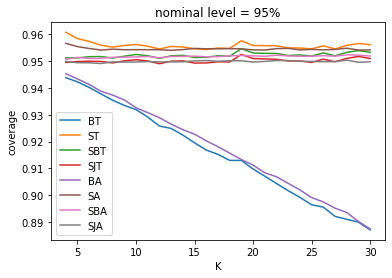}
\end{subfigure} 

\caption{Theoretical and actual coverage probabilities for the model in Section \ref{subsec: validation norm chi}. The definitions of BT, ST, etc are the same as in Figure \ref{fig: fix_n}.\label{fig: normal_chi}}
\end{figure}


\subsection{Coverage error comparisons when fixing total data size}

We investigate the coverage error when the total data size $N$ is fixed and $K$ changes. In this case, the number of data in each section is $N/K$ which will become smaller when $K$ is larger. We consider $N=1000$ and $K$ ranges from 4 to 30. Since $N/K>30$ which is large, we expect our asymptotic approximation to be accurate and estimate the coverage probability by the nominal level plus $c(N/K)^{-1}$. We consider the models in Section \ref{subsec: validation}. The results are plotted in Figure \ref{fig: fix N}. When $K$ is large, we observe a similar trend as in Section \ref{subsec: validation}: SJ has the smallest coverage error, sectioning and SB have over-coverage issues, while batching has under-coverage issues. When $K$ is smaller, the coverage probabilities of all methods are close to the nominal level. This can be attributed to that when $K$ is small, there are sufficient samples in each section which makes the coverage error small.

\subsection{Non-monotonicity of coverage errors in number of batches}

Lastly, we also investigate the trend of the coefficient when fixing total data size (i.e., $Kc$). In general, the coefficient does not exhibit monotonicity behavior as $K$ increases. This can be seen from the estimated error coefficients reported in Table \ref{tab: coef} which contains two experiments with different setups. In particular, from Experiment 1, we can see that the error coefficient of batching increases when $K$ goes from 5 to 7, and then decreases as $K$ increases from 7. The error coefficient of sectioning decreases when $K$ goes from 5 to 7, and then increases as $K$ increases from 7. The error coefficient of SB decreases when $K$ increases from 5 to 9, and then increases as $K$ increases from 9. So we conclude that the error coefficients of batching, sectioning and SB are not monotone. In Experiment 1, the error coefficient of SJ keeps decreasing as $K$ increases from 5 to $7$, but in Experiment 2, we see that the coefficient of SJ increases as $K$ increases from 5 to 7. Therefore, the error coefficient of SJ is also not monotone in a particular direction.



\begin{figure}[htbp]
\begin{subfigure}{0.32\textwidth} \includegraphics[width=1\textwidth]{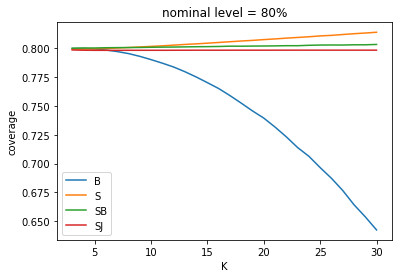}\caption{}
\end{subfigure} \begin{subfigure}{0.32\textwidth} \includegraphics[width=1\textwidth]{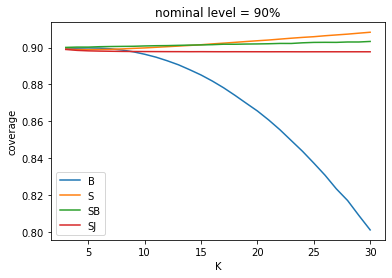}\caption{}
\end{subfigure}
\begin{subfigure}{0.32\textwidth} \includegraphics[width=1\textwidth]{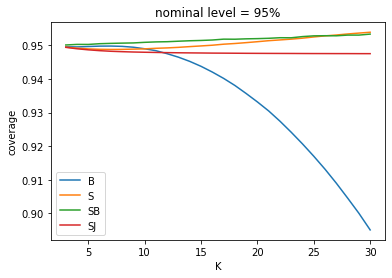}\caption{}
\end{subfigure}
\begin{subfigure}{0.32\textwidth} \includegraphics[width=1\textwidth]{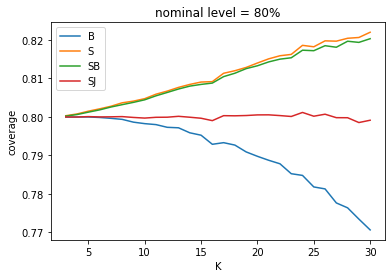}\caption{}
\end{subfigure} \begin{subfigure}{0.32\textwidth} \includegraphics[width=1\textwidth]{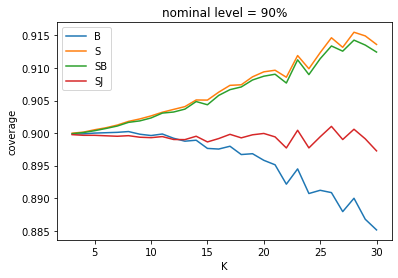}\caption{}
\end{subfigure}
\begin{subfigure}{0.32\textwidth} \includegraphics[width=1\textwidth]{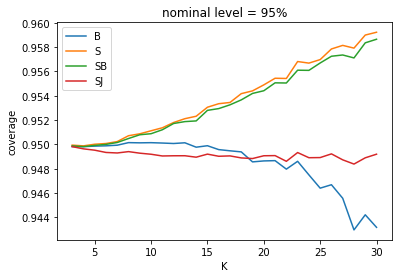}\caption{}
\end{subfigure}
\begin{subfigure}{0.32\textwidth} \includegraphics[width=1\textwidth]{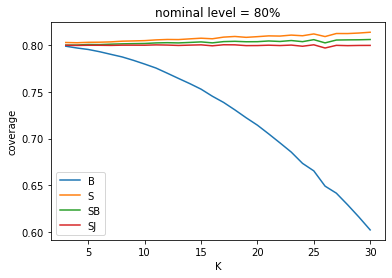}\caption{}
\end{subfigure} \begin{subfigure}{0.32\textwidth} \includegraphics[width=1\textwidth]{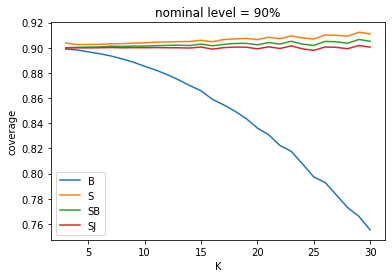}\caption{}
\end{subfigure}
\begin{subfigure}{0.32\textwidth} \includegraphics[width=1\textwidth]{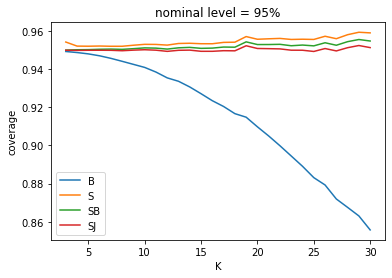}\caption{}
\end{subfigure}

\caption{Theoretical coverage probabilities when $N$ is fixed. (a)-(c), (d)-(f), and (g)-(i) correspond to the models in Sections \ref{subsec: validation normal}, \ref{subsec: validation exp}, and \ref{subsec: validation norm chi}, respectively. The definitions of BT, ST, etc are the same as in Figure \ref{fig: fix_n}.\label{fig: fix N}}
\end{figure}

\begin{table}[]
\caption{Estimated error coefficients. Here $c_{\text{B}}$ represents the coefficient of the $n^{-1}$ coverage error term for batching as defined in \eqref{eq: c def} where we added subscript ``B'' to highlight its dependence on the method used. $c_{\text{S}},c_{\text{SJ}},c_{\text{SB}}$ are defined similarly. ``CI'' represents the half width of a 95\% CI for the target quantity.}
\centering
\textbf{Experiment 1}: $f(x) = x + x^2,X\sim N(0,1)$, nominal level = 95\% \\
\begin{tabular}{|l|l|l|l|l|l|l|l|l|}
\hline
K  & $Kc_{\text{B}}$        & CI-B     & $Kc_{\text{S}}$        & CI-S     & $Kc_{\text{SJ}}$       & CI-SJ    & $Kc_{\text{SB}}$       & CI-SB    \\ \hline
5  & -0.383  & 0.008 & -1.088 & 0.014 & -1.368 & 0.016 & -1.207 & 0.015 \\ \hline
6  & -0.266  & 0.011 & -1.209 & 0.016 & -1.613 & 0.014 & -1.423 & 0.017 \\ \hline
7  & -0.246  & 0.015 & -1.251 & 0.019 & -1.795 & 0.013 & -1.532 & 0.019 \\ \hline
8  & -0.335  & 0.019 & -1.233 & 0.022 & -1.917 & 0.012 & -1.611 & 0.022 \\ \hline
9  & -0.592  & 0.023 & -1.157 & 0.025 & -2.014 & 0.012 & -1.673 & 0.026 \\ \hline
10 & -1.051  & 0.028 & -1.076 & 0.029 & -2.085 & 0.011 & -1.651 & 0.028 \\ \hline
11 & -1.628  & 0.030 & -0.911 & 0.033 & -2.143 & 0.011 & -1.639 & 0.032 \\ \hline
12 & -2.523  & 0.037 & -0.792 & 0.036 & -2.197 & 0.010 & -1.580 & 0.035 \\ \hline
13 & -3.559  & 0.041 & -0.626 & 0.039 & -2.244 & 0.010 & -1.520 & 0.038 \\ \hline
14 & -4.805  & 0.045 & -0.409 & 0.041 & -2.277 & 0.009 & -1.407 & 0.039 \\ \hline
15 & -6.273  & 0.050 & -0.207 & 0.044 & -2.305 & 0.009 & -1.299 & 0.040 \\ \hline
16 & -7.974  & 0.059 & 0.007  & 0.045 & -2.339 & 0.009 & -1.221 & 0.045 \\ \hline
17 & -9.854  & 0.062 & 0.308  & 0.045 & -2.355 & 0.009 & -1.055 & 0.046 \\ \hline
18 & -11.956 & 0.068 & 0.514  & 0.049 & -2.376 & 0.008 & -0.951 & 0.046 \\ \hline
19 & -14.347 & 0.074 & 0.774  & 0.051 & -2.398 & 0.008 & -0.798 & 0.047 \\ \hline
20 & -16.858 & 0.082 & 1.052  & 0.051 & -2.415 & 0.008 & -0.634 & 0.049 \\ \hline
21 & -19.538 & 0.090 & 1.352  & 0.052 & -2.427 & 0.008 & -0.495 & 0.055 \\ \hline
22 & -22.580 & 0.098 & 1.580  & 0.053 & -2.438 & 0.008 & -0.317 & 0.052 \\ \hline
23 & -25.873 & 0.105 & 1.812  & 0.058 & -2.447 & 0.008 & -0.143 & 0.053 \\ \hline
24 & -29.338 & 0.115 & 2.099  & 0.059 & -2.461 & 0.007 & 0.017  & 0.056 \\ \hline
25 & -33.002 & 0.123 & 2.463  & 0.057 & -2.473 & 0.007 & 0.159  & 0.055 \\ \hline
26 & -36.867 & 0.133 & 2.781  & 0.059 & -2.477 & 0.007 & 0.358  & 0.060 \\ \hline
27 & -41.014 & 0.144 & 3.040  & 0.059 & -2.485 & 0.007 & 0.574  & 0.057 \\ \hline
28 & -45.427 & 0.151 & 3.350  & 0.059 & -2.488 & 0.007 & 0.770  & 0.058 \\ \hline
29 & -49.949 & 0.166 & 3.611  & 0.063 & -2.500 & 0.007 & 0.926  & 0.059 \\ \hline
30 & -54.900 & 0.174 & 3.862  & 0.062 & -2.504 & 0.007 & 1.132  & 0.058 \\ \hline
\end{tabular}

\textbf{Experiment 2}: $f(x) = x_1+x_2+x_3-0.4x_1^2-0.06x_1x_2-2.13x_2^2+1.6x_3^2-1.79x_1^3-0.84x_1^2x_2+0.5x_1x_2^2-1.25x_2^3, X\sim N(0,I_3)$, nominal level = 80\%. 

\begin{tabular}{|c|c|c|c|}
\hline
$K$ & 5     & 6      & 7     \\
\hline
$Kc_{\text{SJ}}$  & -0.87 & -0.786 & -0.74 \\
\hline
CI-SJ & 0.015 & 0.013  & 0.012 \\
\hline
\end{tabular}
\label{tab: coef}
\end{table}

\section*{Acknowledgments}
We gratefully acknowledge support from the National Science Foundation under grants CAREER CMMI-1834710 and IIS-1849280.


\bibliographystyle{plainnat}
\bibliography{arXiv}

\appendix

\section{Explicit Coverage Error Expansions for Simple Examples}\label{sec: explicit examples}
Suppose that $K=2$, $\psi(P)=f(E_{P}X):=E_{P}X+\lambda\left(E_{P}X\right)^{2}$
and $P_{0}$ is standard normal. In this case, the batched estimates
are $\psi(\hat{P}_{i})=f(E_{\hat{P}_{i}}X)\stackrel{d}{=}f(\frac{1}{\sqrt{n}}U_{i})$
where $U_{i}\sim N(0,1)$. For this model, the higher-order coverage errors can be computed explicitly via the following lemma.
\begin{lem}\label{lem: K2_computation}
With the model introduced above, the higher-order coverage errors for batching, sectioning, SB and SJ can be expressed as

\begin{equation}
P(-q\leq W_{B}\leq q)-P(-q\leq t_{1}\leq q)=\frac{\lambda^{2}}{n}\left(-q\left(q^{2}-1\right)^{2}\left(\frac{1}{q^{2}+1}\right)^{3}\frac{4}{\pi}\right)+O(n^{-3/2}),\label{eq: K2_b_err}
\end{equation}

\begin{equation}
P(-q\leq W_{S}\leq q)-P(-q\leq t_{1}\leq q) = \frac{\lambda^{2}}{n}\left(-q^{5}\left(\frac{1}{q^{2}+1}\right)^{3}\frac{4}{\pi}+q\left(\frac{1}{q^{2}+1}\right)^{2}\frac{1}{\pi}\right)+O(n^{-3/2}),\label{eq: K2_s_err}
\end{equation}

\begin{equation}\label{eq: K2_sb_err}
P(-q\leq W_{SB}\leq q)-P(-q\leq t_{1}\leq q)=\frac{\lambda^{2}}{n}\left(-q^5\left(\frac{1}{q^{2}+1}\right)^{3}\frac{4}{\pi}\right)+O(n^{-3/2}),
\end{equation}

\begin{equation}\label{eq: K2 sj err}
    P(-q\leq W_{SJ}\leq q)-P(-q\leq t_{1}\leq q)=\frac{\lambda^{2}}{n}\left(-q(q^{2}+1)^{-1}\frac{4}{\pi}\right)+O(n^{-3/2}).
\end{equation}

\end{lem}
\begin{proof}

The test statistic for batching can be
expressed as 
\begin{equation*}
W_{B}  =\frac{\sqrt{nK}\left(\frac{f\left(\frac{1}{\sqrt{n}}U_{1}\right)+f\left(\frac{1}{\sqrt{n}}U_{2}\right)}{2}\right)}{\sqrt{\frac{n}{2}\left(f\left(\frac{1}{\sqrt{n}}U_{1}\right)-f\left(\frac{1}{\sqrt{n}}U_{2}\right)\right)^{2}}}
  =\frac{\sqrt{2}\left(\frac{U_{1}+U_{2}}{2}+\frac{\lambda}{\sqrt{n}}\frac{U_{1}^{2}+U_{2}^{2}}{2}\right)}{\frac{\sqrt{2}}{2}\left|U_{1}-U_{2}\right|\sqrt{\left(1+\frac{\lambda}{\sqrt{n}}\left(U_{1}+U_{2}\right)\right)^{2}}}
\end{equation*}
Denote $A_{0}=\sqrt{2}\frac{U_{1}+U_{2}}{2}$ and $A_{1}=\sqrt{2}\frac{U_{1}-U_{2}}{2}$
(note that they are independent). Then from the above, 
\[
W_{B}=\frac{A_{0}+\frac{\lambda}{\sqrt{2}\sqrt{n}}\left(A_{0}^{2}+A_{1}^{2}\right)}{\left|A_{1}\right|\sqrt{\left(1+\frac{\sqrt{2}\lambda}{\sqrt{n}}A_{0}\right)^{2}}}=\frac{A_{0}+\frac{\lambda}{\sqrt{2}\sqrt{n}}\left(A_{0}^{2}+A_{1}^{2}\right)}{\left|A_{1}\right|\left(1+\frac{\sqrt{2}\lambda}{\sqrt{n}}A_{0}\right)}
\]
Here the second equality holds as long as $1+\frac{\sqrt{2}\lambda}{\sqrt{n}}A_{0}>0$
which happens with probability $1-O(e^{-n})$. Based on this expression for $W_S$, we study the event $W_B\leq q$:
\begin{align*}
W_{B} & \leq q\Leftrightarrow A_{0}+\frac{\lambda}{\sqrt{2}\sqrt{n}}\left(A_{0}^{2}+A_{1}^{2}\right)\leq q\left|A_{1}\right|\left(1+\frac{\sqrt{2}\lambda}{\sqrt{n}}A_{0}\right)
\end{align*}
which is a quadratic function in $A_{0}$.  It can be equivalently
written as
\begin{equation}\label{eq: A_1_formula}
A_{0}+\frac{\lambda}{\sqrt{2n}}\left(A_{0}^{2}-2q\left|A_{1}\right|A_{0}\right)\leq q\left|A_{1}\right|-\frac{\lambda}{\sqrt{2}\sqrt{n}}A_{1}^{2}
\end{equation}
We want to write the above as $A_0<V+O(n^{-3/2})$ for some critical value $V$ that is independent of $A_0$. From the above inequality, $V$ satisfy 
\[
V=q\left|A_{1}\right|+O_p(n^{-1/2})=:q\left|A_{1}\right|+V_{1}
\]
for some $V_1=O_p(n^{-1/2})$. Plugging this in \eqref{eq: A_1_formula} and solving for $V_1$, we get
\[
V_{1}=\left(q^{2}-1\right)\frac{\lambda}{\sqrt{2}\sqrt{n}}A_{1}^{2}+O(n^{-1})=:\left(q^{2}-1\right)\frac{\lambda}{\sqrt{2}\sqrt{n}}A_{1}^{2}+V_{2}.
\]
Again, by plugging this in \eqref{eq: A_1_formula} and solving for $V_2$, we have that $V_2$ satisfy
\[
V_{2}=0+O_p(n^{-3/2}).
\]

So as a conclusion, with exponentially small error, $W_{B}\leq t\Leftrightarrow A_{0}\leq q\left|A_{q}\right|+\left(q^{2}-1\right)\frac{\lambda}{\sqrt{2}\sqrt{n}}A_{1}^{2}+O_p(n^{-3/2}).$
Similarly, $W_{B}\geq-q\Leftrightarrow A_{0}\geq-q\left|A_{1}\right|+\left(q^{2}-1\right)\frac{\lambda}{\sqrt{2}\sqrt{n}}A_{1}^{2}+O_p(n^{-3/2})$. Based on this, for the coverage error we have
\begin{align*}
 & P(-q\leq W_{B}\leq q)-P(-q\leq t_{1}\leq q)\nonumber \\
= & P\left(-q\left|A_{1}\right|+\left(q^{2}-1\right)\frac{\lambda}{\sqrt{2}\sqrt{n}}A_{1}^{2}\leq A_{0}\leq q\left|A_{1}\right|+\left(q^{2}-1\right)\frac{\lambda}{\sqrt{2}\sqrt{n}}A_{1}^{2}\right)\\ 
& -P\left(-q\left|A_{1}\right|\leq A_{0}\leq q\left|A_{1}\right|\right)+O(n^{-3/2})\\
= & E_{A_{1}}\left[\begin{array}{c}
     \Phi\left(q\left|A_{1}\right|+\left(q^{2}-1\right)\frac{\lambda}{\sqrt{2}\sqrt{n}}A_{1}^{2}\right)-\Phi\left(-q\left|A_{1}\right|+\left(q^{2}-1\right)\frac{\lambda}{\sqrt{2}\sqrt{n}}A_{1}^{2}\right)  \\
     -\Phi\left(q\left|A_{1}\right|\right)+\Phi\left(-q\left|A_{1}\right|\right) 
\end{array}\right]+O(n^{-3/2})\\
= & E_{A_{1}}\left[-\phi\left(q\left|A_{1}\right|\right)\frac{q\lambda^{2}\left(q^{2}-1\right)^{2}A_{1}^{5}}{2n}\right]+O(n^{-3/2})\\
= & \frac{1}{\sqrt{2\pi}}\frac{\lambda^{2}}{n}\left(-\frac{q\left(q^{2}-1\right)^{2}}{2}\left(\frac{1}{q^{2}+1}\right)^{3}\mu_{5}\right)+O(n^{-3/2}) \\
= & \frac{\lambda^{2}}{n}\left(-q\left(q^{2}-1\right)^{2}\left(\frac{1}{q^{2}+1}\right)^{3}\frac{4}{\pi}\right)+O(n^{-3/2}).
\end{align*}
Here in the second equality, we condition on $A_2$ first and use that $A_0$ and $A_1$ are independent standard normals. Also $\mu_i$ is the $i$-th absolute moment of the standard normal. So we have shown \eqref{eq: K2_b_err}.

For sectioning, we can do a similar computation 
\[
W_{S}=\frac{A_{0}+\frac{\lambda}{\sqrt{2}\sqrt{n}}A_{0}^{2}}{\left|A_{1}\right|\sqrt{\left(1+\frac{\sqrt{2}\lambda}{\sqrt{n}}A_{0}\right)^{2}+\frac{\lambda^{2}A_{1}^{2}}{2n}}}=\frac{A_{0}+\frac{\lambda}{\sqrt{2}\sqrt{n}}A_{0}^{2}}{\left|A_{1}\right|\left(1+\frac{\sqrt{2}\lambda}{\sqrt{n}}A_{0}\right)\left(1+\frac{\lambda^{2}A_{1}^{2}}{4n}\right)}+O_{p}(n^{-3/2}).
\]
and 
\[
W_{S}\leq q\Leftrightarrow A_{0}+\frac{\lambda}{\sqrt{2n}}\left(A_{0}^{2}-2q\left|A_{1}\right|A_{0}\right)\leq q\left|A_{1}\right|+q\left|A_{1}\right|\frac{\lambda^{2}A_{1}^{2}}{4n}.
\]
After some algebra, we get 
\[
W_{S}\leq q\Leftrightarrow A_{0}\leq q\left|A_{1}\right|+\frac{\lambda}{\sqrt{2n}}q^{2}A_{1}^{2}+q\left|A_{1}\right|\frac{\lambda^{2}A_{1}^{2}}{4n}+O_p(n^{-3/2})
\]
and 
\[
W_{S}\geq-q\Leftrightarrow A_{0}\geq-q\left|A_{1}\right|+\frac{\lambda}{\sqrt{2n}}q^{2}A_{1}^{2}-q\left|A_{1}\right|\frac{\lambda^{2}A_{1}^{2}}{4n}+O_p(n^{-3/2}).
\]
Then we have that 
\begin{align*}
 & P(-q\leq W_{S}\leq q)\\
= & P\left(-q\left|A_{1}\right|+q^{2}\frac{\lambda}{\sqrt{2}\sqrt{n}}A_{1}^{2}-q\left|A_{1}\right|\frac{\lambda^{2}A_{1}^{2}}{4n}\leq A_{0}\leq q\left|A_{1}\right|+q^{2}\frac{\lambda}{\sqrt{2}\sqrt{n}}A_{1}^{2}+q\left|A_{1}\right|\frac{\lambda^{2}A_{1}^{2}}{4n}\right)+O(n^{-3/2})\\
= & E_{A_{1}}\left[\Phi\left(q\left|A_{1}\right|+q^{2}\frac{\lambda}{\sqrt{2}\sqrt{n}}A_{1}^{2}+q\left|A_{1}\right|\frac{\lambda^{2}A_{1}^{2}}{4n}\right)-\Phi\left(-q\left|A_{1}\right|+q^{2}\frac{\lambda}{\sqrt{2}\sqrt{n}}A_{1}^{2}-q\left|A_{1}\right|\frac{\lambda^{2}A_{1}^{2}}{4n}\right)\right]+O(n^{-3/2})\\
= & E_{A_{1}}\left[\Phi\left(q\left|A_{1}\right|\right)-\Phi\left(-q\left|A_{1}\right|\right)+\frac{\lambda^{2}}{n}\phi\left(q\left|A_{1}\right|\right)\left(-\frac{q^{5}\left|A_{1}\right|^{5}}{2}+\frac{1}{2}q\left|A_{1}\right|^{3}\right)\right]+O(n^{-3/2})
\end{align*}
Here the second equality follows by conditioning on $A_{1}$. Thus, the coverage
error of sectioning is given by

\begin{align*}
P(-q\leq W_{S}\leq q)-P(-q\leq t_{1}\leq q) & =E_{A_{1}}\left[\frac{\lambda^{2}}{n}\phi\left(q\left|A_{1}\right|\right)\left(-\frac{q^{5}\left|A_{1}\right|^{5}}{2}+\frac{1}{2}q\left|A_{1}\right|^{3}\right)\right]+O(n^{-3/2})\\
 & =\frac{\lambda^{2}}{n}\left(-q^{5}\left(\frac{1}{q^{2}+1}\right)^{3}\frac{4}{\pi}+q\left(\frac{1}{q^{2}+1}\right)^{2}\frac{1}{\pi}\right)+O(n^{-3/2})
\end{align*}
so \eqref{eq: K2_s_err} is proved. The algebra for SB is quite similar to sectioning. Starting from the following expression for the sectioning statistic:  \[
W_{SB}=\frac{A_{0}+\frac{\lambda}{\sqrt{2}\sqrt{n}}A_{0}^{2}}{\left|A_{1}\right|\sqrt{\left(1+\frac{\sqrt{2}\lambda}{\sqrt{n}}A_{0}\right)^{2}}},
\]
We can do similar computations as above and get \eqref{eq: K2_sb_err}. For SJ, we have
$$W_{SJ}\stackrel{d}{=}\frac{A_{0}+\frac{\lambda}{\sqrt{2n}}\left(A_{0}^{2}-A_{1}^{2}\right)}{\left|A_{1}\right|\sqrt{\left(1+\frac{\sqrt{2}\lambda}{\sqrt{n}}A_{0}\right)^{2}}}$$
It follows that $$P(-q\leq W_{SJ}\leq q)=P(-q\left|A_{1}\right|+(q^{2}+1)\frac{\lambda}{\sqrt{2n}}A_{1}^{2}\leq A_{0}\leq q\left|A_{1}\right|+(q^{2}+1)\frac{\lambda}{\sqrt{2n}}A_{1}^{2})$$ which leads to \eqref{eq: K2 sj err}. 
\end{proof}

Lemma \ref{lem: K2_computation} indicates that these three methods have different higher-order coverage errors. More specifically, their leading term ($n^{-1}$ order term) in the error expansion is different.  When $q\geq1$ (which is usually the case; since 1 is the 75-percentile of the $t_{1}$ distribution), the RHS of each of \eqref{eq: K2_b_err}-\eqref{eq: K2 sj err} is negative, which implies that the actual coverage probability is smaller than the nominal coverage probability. With a litte algebra, we can show that RHS of \eqref{eq: K2 sj err} $<$ RHS of \eqref{eq: K2_sb_err} $<$ RHS of \eqref{eq: K2_s_err} $<$ RHS of \eqref{eq: K2_b_err} $<0$. Thus, batching has the smallest higher-order coverage error and SJ has the largest higher-order coverage error. 

However, if the underlying distribution is not normal, then we also need to
consider the error induced by that. The joint density of $\left(\sqrt{n}\bar{X}_{1},\sqrt{n}\bar{X}_{2}\right)$
admits an Edgeworth expansion where the coefficients are determined
by the cumulants of $X$. For simplicity, consider the case when $EX^{3}=EX=0$
and $VarX=1$. Let $\kappa_{4}=EX^{4}-3$ be the 4-th cumulant. Then the density of $\sqrt{n}X_i$ has Edgeworth expansion $p_{X_i}(x) = \phi(x)(1+\frac{1}{24n}\kappa_4 He_4(x)) + O(n^{-3/2})$. Here $He_4$ is the 4-th Hermite polynomial given by $He_4(x)=x^4-6x^2+3$.  Noting that all of $W_B,W_S,W_{SB}$ and $W_{SJ}$ can be expressed as $\frac{\bar{X}_1+\bar{X}_2}{\left\vert\bar{X}_1-\bar{X}_2 \right\vert}+O_p(n^{-1/2})$, we have that the contribution of the error term in the Edgeworth expansion to the coverage error is given by (for both sectioning and
batching):
\begin{align}\label{eq: err_distribution}
& P\left(-q\leq\frac{\bar{X}_1+\bar{X}_2}{\left\vert\bar{X}_1-\bar{X}_2 \right\vert}\leq q  \right) - P\left(-q\leq\frac{U_1+U_2}{\left\vert U_1-U_2 \right\vert}\leq q  \right)\\ \nonumber = & 2n^{-1}\int_{-q\leq f(\mathbf{z})\leq q}\phi(z_{1})\phi(z_{2})\frac{1}{24}\kappa_{4}He_{4}(z_{1})d\mathbf{z}+O(n^{-3/2})
\end{align}
where $f(z_{1},z_{2})=\frac{z_{1}+z_{2}}{\left|z_{1}-z_{2}\right|}$.
Therefore, the coverage errors become the RHS of \eqref{eq: K2_b_err}-\eqref{eq: K2 sj err} plus the above term. Noting
that $\kappa_{4}$ can be positive or negative, we have that after
adding the above term, it could be the case that $0<$ RHS of \eqref{eq: K2 sj err} + \eqref{eq: err_distribution} $<$ \eqref{eq: K2_sb_err} + \eqref{eq: err_distribution} $<$ RHS of \eqref{eq: K2_s_err} + \eqref{eq: err_distribution} $<$ RHS of \eqref{eq: K2_b_err} + \eqref{eq: err_distribution}. If this is the case, then the coverage error of SJ is the smallest.

\section{Expansions for Other Schemes}\label{sec: expansion other schemes}
In this section, we do Steps 2 and 3 of Algorithm \ref{alg: error estimation} for SJ, batching, and SB. We will continue to use the notations introduced in Section \ref{sec: alg}.
\subsection{Sectioned jackknife}

We know that 
\[
W_{SJ}=\frac{\sqrt{nK}\left(\bar{J}-\psi_{0}\right)}{\sqrt{\frac{1}{K-1}\sum_{i=1}^{K}(\sqrt{n}J_{i}-\sqrt{n}\bar{J})^{2}}}
\]
where 
\begin{align*}
J_{i} & =Kf(\bar{\bar{X}})-(K-1)f\left(\frac{K\bar{\bar{X}}-\bar{X}_{i}}{K-1}\right)\\
 & =Kf(m+n^{-1/2}A_{0})-(K-1)f\left(\frac{K\left(m+n^{-1/2}A_{0}\right)-\left(m+n^{-1/2}\left(A_{0}+B_{i}\right)\right)}{K-1}\right)\\
 & =Kf(m+n^{-1/2}A_{0})-(K-1)f\left(m+n^{-1/2}\left(A_{0}-\frac{B_{i}}{K-1}\right)\right)\\
 & =f_{0}+K\left(n^{-1/2}[u,A_{0}]+n^{-1}[v,A_{0},A_{0}]+n^{-3/2}[w,A_{0},A_{0},A_{0}]\right)\\
 & -(K-1)\left(\begin{array}{c}
n^{-1/2}[u,A_{0}-\frac{B_{i}}{K-1}]+n^{-1}[v,A_{0}-\frac{B_{i}}{K-1},A_{0}-\frac{B_{i}}{K-1}]\\
+n^{-3/2}[w,A_{0}-\frac{B_{i}}{K-1},A_{0}-\frac{B_{i}}{K-1},A_{0}-\frac{B_{i}}{K-1}]
\end{array}\right)\\
 & =f_{0}+n^{-1/2}\left[u,A_{0}+B_{i}\right]+n^{-1}\left(K[v,A_{0},A_{0}]-(K-1)[v,A_{0}-\frac{B_{i}}{K-1},A_{0}-\frac{B_{i}}{K-1}]\right)\\
 & +n^{-3/2}\left(K[w,A_{0},A_{0},A_{0}]-(K-1)[w,A_{0}-\frac{B_{i}}{K-1},A_{0}-\frac{B_{i}}{K-1},A_{0}-\frac{B_{i}}{K-1}]\right)
\end{align*}
and 
\begin{align*}
\bar{J} & =\frac{1}{K}\sum_{i}J_{i}\\
 & =f_{0}+n^{-1/2}\left[u,A_{0}\right]+n^{-1}\left(K[v,A_{0},A_{0}]-\frac{K-1}{K}\sum_{i}[v,A_{0}-\frac{B_{i}}{K-1},A_{0}-\frac{B_{i}}{K-1}]\right)\\
 & +n^{-3/2}\left(K[w,A_{0},A_{0},A_{0}]-\frac{K-1}{K}\sum_{i}[w,A_{0}-\frac{B_{i}}{K-1},A_{0}-\frac{B_{i}}{K-1},A_{0}-\frac{B_{i}}{K-1}]\right)
\end{align*}
Therefore, 
\begin{align*}
 & \sum_{i=1}^{K}(\sqrt{n}J_{i}-\sqrt{n}\bar{J})^{2}\\
= & \sum_{i=1}^{K}\left(\begin{array}{c}
[u,B_{i}]-n^{-1/2}\left[(K-1)[v,A_{0}-\frac{B_{i}}{K-1},A_{0}-\frac{B_{i}}{K-1}]-\frac{K-1}{K}\sum_{i}[v,A_{0}-\frac{B_{i}}{K-1},A_{0}-\frac{B_{i}}{K-1}]\right]\\
-n^{-1}\left[\begin{array}{c}
     (K-1)[w,A_{0}-\frac{B_{i}}{K-1},A_{0}-\frac{B_{i}}{K-1},A_{0}-\frac{B_{i}}{K-1}]  \\
     -\frac{K-1}{K}\sum_{i}[w,A_{0}-\frac{B_{i}}{K-1},A_{0}-\frac{B_{i}}{K-1},A_{0}-\frac{B_{i}}{K-1}] 
\end{array}\right]
\end{array}\right)^{2}
\end{align*}
The coefficient for $n^{-1/2}$ is 
\[
\lambda=-2\sum_{i}[u,B_{i}]\left[(K-1)[v,A_{0}-\frac{B_{i}}{K-1},A_{0}-\frac{B_{i}}{K-1}]\right]
\]
The coefficient for $n^{-1}$ is 
\begin{align*}
e & =-2\sum_{i}[u,B_{i}](K-1)[w,A_{0}-\frac{B_{i}}{K-1},A_{0}-\frac{B_{i}}{K-1},A_{0}-\frac{B_{i}}{K-1}]\\
 & +\sum_{i}\left[(K-1)[v,A_{0}-\frac{B_{i}}{K-1},A_{0}-\frac{B_{i}}{K-1}]-\frac{K-1}{K}\sum_{i}[v,A_{0}-\frac{B_{i}}{K-1},A_{0}-\frac{B_{i}}{K-1}]\right]^{2}
\end{align*}
We also have that $a=\left[u,A_{0}\right],b_{1}=K[v,A_{0},A_{0}]-\frac{K-1}{K}\sum_{i}[v,A_{0}-\frac{B_{i}}{K-1},A_{0}-\frac{B_{i}}{K-1}]=[v,A_{0},A_{0}]-\frac{1}{K(K-1)}\sum_{i}[v,B_{i},B_{i}],$
\[
b_{2}=K[w,A_{0},A_{0},A_{0}]-\frac{K-1}{K}\sum_{i}[w,A_{0}-\frac{B_{i}}{K-1},A_{0}-\frac{B_{i}}{K-1},A_{0}-\frac{B_{i}}{K-1}]
\]
\[
b_{1}^{\prime}=2v_{11}A_{0,1}+2\sum_{j=2}^{d}v_{1j}A_{0,j}=2\sum_{j=1}^{d}v_{1j}A_{0,j}\ 
\]
\[
\lambda^{\prime}=-4(K-1)\sum_{i}[u,B_{i}]\left(\sum_{j=1}^{d}v_{1j}(A_{0,j}-\frac{B_{i,j}}{K-1})-\sum_{j=1}^{d}v_{1j}A_{0,j}\right)=4\sum_{i}[u,B_{i}]\sum_{j=1}^{d}v_{1j}B_{i,j}\ 
\]
We observe that the formulas for $b_1^\prime$ and $\lambda^\prime$ are the same as the corresponding formulas for sectioning.

\subsection{Batching}

We have that 
\begin{align*}
W_{B} & =\frac{\sqrt{nK}\left(\frac{1}{K}\sum_{i}f\left(m+n^{-1/2}X_{i}\right)-f_{0}\right)}{\sqrt{\frac{1}{K-1}\sum_{i=1}^{K}\left(\sqrt{n}f\left(m+n^{-1/2}X_{i}\right)-\sqrt{n}\frac{1}{K}f\left(m+n^{-1/2}X_{i}\right)\right)^{2}}}
\end{align*}
Since 
\begin{align*}
f\left(m+n^{-1/2}X_{i}\right) & =f_{0}+n^{-1/2}[u,A_{0}+B_{i}]+n^{-1}[v,A_{0}+B_{i},A_{0}+B_{i}]\\
 & +n^{-3/2}[w,A_{0}+B_{i},A_{0}+B_{i},A_{0}+B_{i}]
\end{align*}
and 
\begin{align*}
\frac{1}{K}\sum_{i}f\left(m+n^{-1/2}X_{i}\right) & =f_{0}+n^{-1/2}[u,A_{0}]+n^{-1}\frac{1}{K}\sum_{i}[v,A_{0}+B_{i},A_{0}+B_{i}]\\
 & +n^{-3/2}\frac{1}{K}\sum_{i}[w,A_{0}+B_{i},A_{0}+B_{i},A_{0}+B_{i}],
\end{align*}
we have $a=[u,A_{0}]$,
\[
b_{1}=\frac{1}{K}\sum_{i}[v,A_{0}+B_{i},A_{0}+B_{i}],
\]
\[
b_{2}=\frac{1}{K}\sum_{i}[w,A_{0}+B_{i},A_{0}+B_{i},A_{0}+B_{i}]
\]
and the denominator is
\[
\sum_{i}\left(\begin{array}{c}
[u,B_{i}]+n^{-1/2}\left([v,A_{0}+B_{i},A_{0}+B_{i}]-\frac{1}{K}\sum_{i}[v,A_{0}+B_{i},A_{0}+B_{i}]\right)\\
+n^{-1}\left([w,A_{0}+B_{i},A_{0}+B_{i},A_{0}+B_{i}]-\frac{1}{K}\sum_{i}[w,A_{0}+B_{i},A_{0}+B_{i},A_{0}+B_{i}]\right)
\end{array}\right)^{2}
\]
Therefore, 
\[
\lambda=2\sum_{i}[u,B_{i}][v,A_{0}+B_{i},A_{0}+B_{i}]
\]
\[
e=2\sum_{i}\left[\begin{array}{c}
     [u,B_{i}][w,A_{0}+B_{i},A_{0}+B_{i},A_{0}+B_{i}]  \\
     +\left([v,A_{0}+B_{i},A_{0}+B_{i}]-\frac{1}{K}\sum_{i}[v,A_{0}+B_{i},A_{0}+B_{i}]\right)^{2} 
\end{array} \right]
\]
The expressions for $\lambda^{\prime}$ and $b_{1}^{\prime}$ are the same
as for sectioning.

\subsection{Sectioning-batching}

Just replace $\lambda$ and $e$ for sectioning with the $\lambda$ and $e$ for
batching. Other steps are the same as before.

\section{An Alternative Algorithm to Estimate Coefficient of the $n^{-1}$ Error for Batching}\label{sec: batching algorithm}

For batching (only), we propose an alternative algorithm given in Algorithm \ref{alg: batching}. It is not hard to see that Theorem \ref{thm: batching_expansion} along with the Edgeworth expansion given in Section 2.3 of \cite{Hall1992} imply the correctness of Algorithm \ref{alg: batching}. To illustrate how to use Algorithm \ref{alg: batching} and compare with Algorithm \ref{alg: error estimation}, we consider the example in Proposition \ref{prop: errorK2}. More precisely, consider $f(x,y)=x+\lambda y^{2}$ where $X,Y\sim N(0,1)$. Then $A_{1}\stackrel{d}{=}\left(X,Y\right)$. The cumulants of $X+\lambda n^{-1/2}Y^{2}$
are given by 
\[
\kappa_{1,n}=\lambda n^{-1/2},\kappa_{2,n}=1+2\lambda^{2}n^{-1},\kappa_{3,n}=O(n^{-3/2}),\kappa_{4,n}=O(n^{-3/2}),
\]
so we have $k_{1,2}=\lambda,k_{2,2}=2\lambda^{2}$. Therefore, 
\[
h_{1}(x)=-\lambda\phi(x),h_{2}(x)=-x\frac{3}{2}\lambda^{2}\phi(x).
\]
Then, following Step 3 of Algorithm \ref{alg: batching},
\[
p_{1}(x)=\lambda x\phi(x),p_{2}(x)=-\frac{3}{2}\lambda^{2}\left(1-x^{2}\right)\phi(x).
\]

\begin{algorithm}[htp]
\caption{An unbiased simulation scheme to compute coefficient of the $n^{-1}$ error for batching}
\label{alg: batching}
\begin{algorithmic}[1]
\Ensure derivatives of $f$ at $EX_1$ up to order 3, cumulants of $X_1$ up to order 4, testing statistic $W_B$. 

\State Let $A_1=\sqrt{n}(\bar{X}_1-EX_1)/Var([\nabla f,X_1-EX_1])$. Algebraically compute the cumulants of $[\nabla f,A_1]+n^{-1/2}[\nabla^2 f,A_1,A_1]+n^{-1}[\nabla^3 f,A_1,A_1,A_1]$ up to order 4 (with residual $O(n^{-3/2})$), in terms of $n$ and the cumulants of $X_1$. Moreover, write the cumulants $\kappa_{i,n},i=1,2,3,4$ (defined as the cumulants of order 1,2,3,4 respectively) as a series of $n^{-1/2}$: $\kappa_{1,n}=n^{-1/2}k_{1,2},\kappa_{2,n}=1+n^{-1}k_{2,2}+O(n^{-3/2}),\kappa_{3,n}=n^{-1/2}k_{3,n}+O(n^{-3/2}),\kappa_{4,n}=n^{-1}k_{4,1}+O(n^{-3/2})$.

\State Let $h_1(x)=-\{k_{1,2}+\frac{1}{6}k_{3,1}(x^2-1)\}\phi(x)$, 
$$h_2(x)=-x\{\frac{1}{2}(k_{2,2}+k_{1,2}^2)+\frac{1}{24}(k_{4,1}+4k_{1,2}k_{3,1})(x^2-3)+\frac{1}{72}k_{3,1}^2(x^4-10x^2+15)\}\phi(x). $$ Here $\phi(\cdot)$ is the pdf of standard normal

\State Let $p_1(x)=h_1^\prime(x)/\phi(x)$ and $p_2(x)=h_2^\prime(x)/\phi(x)$.

\State Generate $Z_1,Z_2,\dots,Z_K$ i.i.d. from standard normal. Derive the error term estimator:
\begin{align*}
 ER_B = & I_C(Z)\left[\sum_{1\leq i<j\leq K}p_{1}(Z_i)p_{1}(Z_j)+\sum_{1\leq i\leq K}p_{2}(Z_i)\right]
\end{align*}
Here, $C$ represents the set of $Z=(Z_1,\dots,Z_K)$ such that $-q\leq\sqrt{K(K-1)}\frac{\bar{Z}}{\sqrt{\sum_{i=1}^{K}\left(Z_i-\bar{Z}\right)^{2}}}\leq q$. 

\lastcon{An unbiased estimator of $c$  in \eqref{eq: c def} given by $ER_B$
}
\end{algorithmic}
\end{algorithm}

We run both Algorithms \ref{alg: error estimation} and \ref{alg: batching} $10^4$ times to estimate the theoretical coverage probabilities. The results are shown in Table \ref{tab: batching comparision}. When $K=2$, the estimated coverage probability of Algorithm \ref{alg: error estimation} is 0.495 which is much lower than the estimated coverage probability (0.788), while Algorithm \ref{alg: batching} gives a close approximation (0.774). The failure of Algorithm \ref{alg: error estimation} in this case verifies Proposition \ref{prop: batching small K}. When $K\geq 3$, Algorithms \ref{alg: error estimation} and \ref{alg: batching} have close coverages. Indeed, the differences between the estimated coverage probabilities of the two algorithms are smaller than the corresponding CI half widths. We also observe that when $K\geq 4$, Algorithm \ref{alg: error estimation} has a shorter CI.  Note that the CI half width is computed as  $ 1.96\frac{\text{empirical standard deviation}}{\sqrt{n_{rep}}}$ where $n_{rep}$ is the number of replications which is $10^4$ for this example. Therefore, the shorter CI implies that Algorithm \ref{alg: error estimation} has a smaller variance. This can be attributed to that the construction of Algorithm \ref{alg: error estimation} employs a conditioning argument which helps reduce the variance. Another observation is that when $K$ is larger, the differences between the actual coverage probabilities and the estimated coverage probabilities become larger. This is because when $K$ is large, the coverage error of batching is quite large so that the approximation of coverage error via expansion is not accurate in this regime.

\begin{table}[]
\caption{Estimated coverage probabilities for batching. ``Actual'' stands for the actual coverage estimated from experimental repetitions. ``Alg 1'' and ``Alg 2'' stand for the estimated coverage probabilities given by Algorithms \ref{alg: error estimation} and \ref{alg: batching} respectively. CI half widths 1 and 2 are the corresponding CI half widths.}
\centering

\begin{tabular}{|c|c|c|c|c|c|}
\hline
K  & Actual & Alg 1 & CI half width 1 & Alg 2 & CI half width 2 \\
\hline
2  & 0.788  & 0.495 & 0.024           & 0.774 & 0.007           \\
3  & 0.765  & 0.742 & 0.013           & 0.730 & 0.008           \\
4  & 0.742  & 0.710 & 0.006           & 0.715 & 0.009           \\
5  & 0.717  & 0.685 & 0.006           & 0.684 & 0.010           \\
6  & 0.695  & 0.649 & 0.005           & 0.644 & 0.011           \\
7  & 0.673  & 0.617 & 0.005           & 0.619 & 0.012           \\
8  & 0.652  & 0.585 & 0.005           & 0.590 & 0.013           \\
9  & 0.632  & 0.555 & 0.005           & 0.561 & 0.014           \\
10 & 0.612  & 0.524 & 0.005           & 0.523 & 0.015           \\
11 & 0.592  & 0.493 & 0.005           & 0.481 & 0.015           \\
12 & 0.573  & 0.468 & 0.006           & 0.450 & 0.016           \\
13 & 0.555  & 0.437 & 0.006           & 0.431 & 0.017           \\
14 & 0.537  & 0.406 & 0.006           & 0.404 & 0.018           \\
15 & 0.520  & 0.376 & 0.006           & 0.378 & 0.019           \\
16 & 0.503  & 0.348 & 0.007           & 0.333 & 0.019           \\
17 & 0.486  & 0.313 & 0.007           & 0.308 & 0.020           \\
18 & 0.471  & 0.289 & 0.007           & 0.281 & 0.021           \\
19 & 0.455  & 0.258 & 0.007           & 0.240 & 0.021           \\
20 & 0.439  & 0.227 & 0.008           & 0.233 & 0.022 \\
\hline
\end{tabular}
\label{tab: batching comparision}
\end{table}

\section{Proof of Theorem \ref{thm: validity}}\label{sec: proof_validity}
\begin{proof}
We consider sectioned jackknife for example. The analysis for other
methods is similar (indeed, easier). Moreover, in this proof we focus on the expansion for coverage of the symmetric CI, i.e., $P(-q\leq W_{SJ}\leq q)$. It will be clear that the proof could be easily adapted to the one-sided setting. We introduce some notations first.
Denote $m:=E_{P_{0}}X$ and $A_{i}=\sqrt{n}\left(\bar{X}_{i}-m\right)$
where $\bar{X}_{i}$ is the $i$-th batch average. Let $A_{0}=\frac{A_{1}+\dots+A_{K}}{K}$.
For $1\leq i\leq K$, let $B_{i}=A_{i}-A_{0}$. Denote $A_{0,j},1\leq j\leq d$
to be the $j$-th coordinate of $A_{0}$. Define $B_{i,j},1\leq i\leq K,1\leq j\leq d$
similarly. Denote $A$ as the $dK$-dimensional vector $(A_0,B_1,B_2,\dots,B_{K-1})$. Denote $A^{\prime}$ as the $\left(dK-1\right)$- dimensional
vector $(A_{0,2},\dots,A_{0,d},B_{1},\dots,B_{K-1})$. Let $\nabla f$ denote the gradient of $f$ at point $EX$.

First, we apply the Edgeworth expansion result in 
Theorem 20.1 of \citet{bhattacharya2010normal} to
approximate the distribution of batch averages (for the ease of reading, we included the statement of this theorem in Theorem \ref{thm: Edgeworth}), we have that there
exist polynomials $p_{j}(a),1\leq j\leq r$, such that (we include more details on this step in Section \ref{subsec: edgeworth details})
\begin{equation}
P(-q\leq W_{SJ}\leq q)=\tilde{P}(-q\leq W_{SJ}\leq q)+O(n^{-(r+1)/2})\label{eq: Edgeworth_approx}
\end{equation}
where under $\tilde{P}$, $A_{i},1\leq i\leq K$ has density $p(x)=\phi_{\Sigma}(x)\left(1+\sum_{i=1}^{r}n^{-i/2}p_{i}(x)\right)$ (Note that $p(x)$ can be negative when $n$ is small. In this case, we interpret $\tilde{P}$ as a signed measure.).
Here $\phi_{\Sigma}(x)$ is the limiting normal distribution of $A_{i}$.
In the sequel, unless otherwise mentioned, each statement is under
the approximated measure $\tilde{P}$.

Although we have the expansion for the distribution of $A_{i}$ under
$\tilde{P}$, this does not easily imply an expansion for the distribution
of $W_{SJ}$. This is because $W_{SJ}$ can not be expressed as a
fixed function of $(A_{1},A_{2},\dots,A_{K})$. Indeed, $\psi(\hat{P}_{(i)})=f\left(m+n^{-1/2}\left(\frac{KA_{0}-A_{i}}{K-1}\right)\right)$,
which is dependent on $n$ even when each $A_{i}$ is given. Our way
to handle this is to reformulate $-q\leq W_{SJ}\leq q$ as $F_{-}^{(n)}\leq A_{0,1}\leq F_{+}^{(n)}$.
In other words, we regard $W_{SJ}$ as a function of $\left(n^{-1/2},A_{0,1},A^{\prime}\right)$,
and then solve for the possible values of $A_{0,1}$ (given $A^{\prime}$
and $n^{-1/2}$) that makes $-q\leq W_{SJ}\leq q$. Essentially, this
is done by applying the implicit function theorem. Heuristically,
this is not difficult. But we will see that to get a valid expansion
for $P\left(-q\leq W_{SJ}\leq q\right)$, more care must be taken
to bound the moments of each term and the residual in the expansion. Let \[
E_{2}=\sum_{i=1}^{K}\left(\left(\nabla f\right)^{\top}B_{i}\right)^{2},F_{+}=\frac{q\sqrt{E_2}}{\sqrt{K(K-1)}}-\sum_{i=2}^d \nabla f_i A_{0,i}, F_{-}=\frac{-q\sqrt{E_2}}{\sqrt{K(K-1)}}-\sum_{i=2}^d \nabla f_i A_{0,i}.
\]
As one can check, when $n=\infty$, the expressions of $F_{+}^{n}$ and $F_{-}^{(n)}$ would be given by $F_{+}$ and $F_{-}$ respectively. The following lemma states that $F_{+}^{(n)}$ and $F_{-}^{n}$ have expansion around $F_{+}$ and $F_{-}$ and the terms in the expansion have bounded moments.
\begin{lem}
\label{lem: taylor_approximation}Suppose that the conditions of Theorem \ref{thm: validity} hold. There exist $F_{-}^{(n)}$ and $F_{+}^{(n)}$
(functions of $A^{\prime}$
and depend on $n$), such that 
\[
\tilde{P}\left(\left\{ -q\leq W_{SJ}\leq q\right\} \Delta\left\{ F_{-}^{(n)}\leq A_{0,1}\leq F_{+}^{(n)}\right\} \right)=O(n^{-(r+1)/2}).
\]
Moreover, the powers $\left(F_{+}^{(n)}-F_{+}\right)^{l}$ and
$\left(F_{-}^{(n)}-F_{-}\right)^{l}$, $l=1,2,\dots,r+1$ have expansion
\[
\left(F_{+}^{(n)}-F_{+}\right)^{l}=\sum_{m=l}^{r}n^{-m/2}\sqrt{E_{2}}\frac{{p}_{l,m}(A^{\prime})}{E_{2}^{m}}+n^{-(r+1)/2}R_{F,l}
\]
where each ${p}_{l,m}(A^{\prime})$ is a polynomial of $A^{\prime}$
and there exist $c>0$ such that $E\left[\left|\sqrt{E_{2}}\frac{p_{l,m}(A^{\prime})}{E_{2}^{m}}\right|^{1+c}\right]<\infty,E\left[\left|R_{F,l}\right|^{1+c}\right]<\infty$. Here, the expectation is taken under the limiting normal distribution of $A^{\prime}$.
\end{lem}

Based on this, we can write $-q\leq W_{SJ}\leq q$ as 
\[
F_{-}^{(n)}(A_{0,2},\dots,A_{0,d},B_{1},\dots,B_{K-1})\leq A_{0,1}\leq F_{+}^{(n)}(A_{0,2},\dots,A_{0,d},B_{1},\dots,B_{K-1}).
\]
Corresponding to $A=(A_{0},B_{1},B_{2},\dots,B_{K-1})$, let $a=(a_{0},b_{1},b_{2},\dots,b_{K-1})$
where each of $a_{0},b_{1},\dots,b_{K-1}$ has $d$ coordinates. Corresponding to $A^\prime$, let $a^{\prime}$
be $a$ with its first coordinate removed, i.e., $a^{\prime}=(a_{0,2},\dots,a_{0,d},b_{1},\dots,b_{K-1})$.
Then, following (\ref{eq: Edgeworth_approx}) and a linear transformation
argument, we have that there exist $p_{j,A}(a),1\leq j\leq r$, such
that 
\begin{align}
\tilde{P}(-q\leq W_{SJ}\leq q)& =\int\phi_{\tilde{\Sigma}}(a)\left(1+\sum_{j=1}^{r}n^{-j/2}p_{j,A}(a)\right)da\nonumber\\ &=\int\phi_{\Sigma^{\prime}}(a^{\prime})\left[\int_{F_{-}^{(n)}}^{F_{+}^{(n)}}\phi_{\tilde{\sigma}_{0}}(a_{0,1}-\tilde{\mu}_0)\left(1+\sum_{j=1}^{r}n^{-j/2}p_{j,A}(a)\right)da_{0,1}\right]da^{\prime}\label{eq: WSJ integration}
\end{align}
Here, $\Sigma$ is the variance of $A$, $\Sigma^{\prime}$ is the variance of $A^{\prime}$,
$\tilde{\sigma}_{0}$ and $\tilde{\mu}_0$ are the conditional variance and expectation of $A_{0,1}$ given $A^{\prime}$, when $A^{\prime}=a^{\prime}$ (under their joint limiting normal distribution), so that we have $\phi_{\Sigma^{\prime}}(a^{\prime})\phi_{\tilde{\sigma}_{0}}(a_{0,1}-\tilde{\mu}_0)=\phi_{\tilde{\Sigma}}(a)$where $\tilde{\Sigma}$ is covariance of $A=(A_0,B_1,\dots,B_{K-1})$.
For the inner integration, we have the following Taylor expansion
with residual
\begin{align}
 & \int_{F_{-}^{(n)}}^{F_{+}^{(n)}}\phi_{\tilde{\sigma}_{0}}(a_{0,1}-\tilde{\mu}_0)\left(1+\sum_{j=1}^{r}n^{-j/2}p_{j,A}(a)\right)da_{0,1}\nonumber \\
= & \int_{F_{-}}^{F_{+}}\phi_{\tilde{\sigma}_{0}}(a_{0,1}-\tilde{\mu}_0)\left(1+\sum_{j=1}^{r}n^{-j/2}p_{j,A}(a)\right)da_{0,1}\nonumber \\
 & +\sum_{m=1}^{r}\frac{1}{(m-1)!}\frac{\partial^{m-1}}{\partial a_{0,1}^{m-1}}\left[\phi_{\tilde{\sigma}_{0}}(a_{0,1}-\tilde{\mu}_0)\left(1+\sum_{j=1}^{r}n^{-j/2}p_{j,A}(a)\right)\right]|_{a_{0,1}=F_{+}}\left(F_{+}^{(n)}-F_{+}\right)^{m}\nonumber \\
 & -\sum_{m=1}^{r}\frac{1}{(m-1)!}\frac{\partial^{m-1}}{\partial a_{0,1}^{m-1}}\left[\phi_{\tilde{\sigma}_{0}}(a_{0,1}-\tilde{\mu}_0)\left(1+\sum_{j=1}^{r}n^{-j/2}p_{j,A}(a)\right)\right]|_{a_{0,1}=F_{-}}\left(F_{-}^{(n)}-F_{-}\right)^{m}\nonumber \\
 & +R_{7}\label{eq: integration_expansion}
\end{align}
Note that the derivatives of $\phi_{\tilde{\sigma}_{0}}(a_{0,1}-\tilde{\mu}_0)\left(1+\sum_{j=1}^{r}n^{-j/2}p_{j,A}(a)\right)$
have the form of $\phi_{\tilde{\sigma}_{0}}(a_{0,1}-\tilde{\mu}_0)$ multiplied by a polynomial of $a$. Therefore,
we have that the residual term can be bounded by 
\[
\left|R_{7}\right|\leq M(a^\prime)\left|F_{+}^{(n)}-F_{+}\right|^{r+1}+M(a^\prime)\left|F_{-}^{(n)}-F_{-}\right|^{r+1}
\]
where $M(a^\prime)$ is a polynomial of $a^\prime$.

Now, put the expansion in Lemma \ref{lem: taylor_approximation}
into (\ref{eq: integration_expansion}), and do the outer integration
for $A^{\prime}$ in \eqref{eq: WSJ integration}. (Lemma \ref{lem: taylor_approximation} indicates
that each coefficient in the expansion of $\left(F_{+}^{(n)}-F_{+}\right)^{l}$
has finite $1+c$ moment under the limiting normal distribution of
$A^{\prime}$. Observe that this still holds when we multiply the
coefficient with a polynomial of $A$, since $A$ has arbitrary order
of moments under $\hat{P}$), we get an expansion as a series of $n^{-1/2}$
as claimed.


The argument for the second part of the theorem, i.e., the error of a symmetric CI is of order $O(n^{-1})$ is given in Section \ref{subsec: symmetric CI}. 
\end{proof}

\subsection{The error of a symmetric CI is of order $n^{-1}$}\label{subsec: symmetric CI}


Consider sectioning first. Since we have shown the validity of the expansion as a series of $n^{-1/2}$ in the first part, to show that the error is of order $O(n^{-1})$ it suffices to show that the order is of order $o(n^{-1/2})$. 
With some Taylor expansion arguments as in Section \ref{sec: alg}, we can see that $W_{S}$ is given by
\begin{align}
W_S & =\frac{\nabla f^{T}\cdot A_{0}+\frac{1}{2\sqrt{nK}}A_{0}^{T}\left[\nabla^{2}f\right]A_{0}}{\sqrt{\frac{1}{K-1}\sum_{k=1}^{K}\left(\frac{\nabla f^{T}\cdot A_{k}}{\sqrt{K}}\right)^{2}}\left(1+\frac{1}{\sqrt{nK}}(A_0^Tc_b+c_a)\right)}+O_{p}\left(n^{-1}\right)\nonumber \\
 & =\frac{\nabla f^{T}\cdot A_{0}+\frac{1}{2\sqrt{nK}}A_{0}^{T}\left[\nabla^{2}f\right]A_{0}}{\sqrt{\frac{1}{K-1}\sum_{k=1}^{K}\left(\frac{\nabla f^{T}\cdot A_{k}}{\sqrt{K}}\right)^{2}}}\left(1-\frac{1}{2\sqrt{nK}}\left(A_{0}^{T}c_b+c_a\right)\right)+O_{p}\left(n^{-1}\right).\label{eq: 2nd_taylor}
\end{align}
where $c_a=\frac{\sum_{k=1}^{K}\frac{\nabla f^{T}\cdot B_{k}}{K}B_{k}^{T}\left[\nabla^{2}f\right]B_{k}}{\sum_{k=1}^{K}\left(\frac{\nabla f^{T}\cdot B_{k}}{\sqrt{K}}\right)^{2}},c_b=2\frac{\sum_{k=1}^{K}\frac{\nabla f^{T}\cdot B_{k}}{K}\left[\nabla^{2}f\right]B_{k}}{\sum_{k=1}^{K}\left(\frac{\nabla f^{T}\cdot B_{k}}{\sqrt{K}}\right)^{2}}$ (both of them are functions of $B_{1},\dots,B_{K}$).
Denote the above function as $g_{n}\left(A_{0},B_{1},\dots,B_{K-1}\right)$
(note that $B_{K}=-\left(B_{1}+\dots+B_{K-1}\right)$ so we do not need to include $B_K$ as an argument of $g_n$). Since $(A_0,B_1,\dots,B_{K-1})$ is a linear transformation of the batch averages, and the batch averages have valid Edgeworth expansions, we have that the joint distribution of $A_0,B_1,\dots,B_{K-1}$ admits a valid multivariate Edgeworth expansion: 
\[
P((A_0,B_1,\dots,B_{K-1})\in B) = \int_{B}\phi_{\tilde{\Sigma}}(\mathbf{z})(1+n^{-1/2}p(\mathbf{z}))d\mathbf{z} +o(n^{-1/2})
\]
for all Borel sets $B$. Here $\tilde{\Sigma}$ denotes the covariance of $(A_0,B_1,\dots,B_{K-1})$, and $p(\mathbf{z})$ is an odd polynomial. For the probability
that $-q\leq W_{S}\leq q$, we have that 
\begin{align*}
 & P(-q\leq W_{S}\leq q)\\
= & P(-q\leq g_{n}\left(B_{0},B_{1},\dots,B_{K-1}\right)\leq q)+o(n^{-1/2})\\
= & \int_{-q\leq g_{n}(\mathbf{z})\leq q}\left(\phi_{\tilde{\Sigma}}(\mathbf{z})+n^{-1/2}\phi_{\tilde{\Sigma}}(\mathbf{z})p\left(\mathbf{z}\right)\right)d\mathbf{z}+o(n^{-1/2})\\
= & \int_{-q\leq g_{n}(\mathbf{z})\leq q}\phi_{\tilde{\Sigma}}(\mathbf{z})d\mathbf{z}+n^{-1/2}\int_{-q\leq g_{\infty}(\mathbf{z})\leq q}\phi_{\tilde{\Sigma}}(\mathbf{z})p(\mathbf{z})+o(n^{-1/2}).
\end{align*}

Here $g_{\infty}(\mathbf{z})=\lim_{n\rightarrow\infty}g_{n}(\mathbf{z})=\frac{\nabla f^{T}\cdot z_{0}}{\sqrt{\frac{1}{K-1}\sum_{k=1}^{K}\left(\frac{\nabla f^{T}\cdot z_{k}}{\sqrt{K}}\right)^{2}}}$ and in the last equality above, we used that $g_{\infty}(\mathbf{z})-g(\mathbf{z})=O(n^{-1/2})$. Also note that $g_{\infty}$ satisfies
 $g_{\infty}(\mathbf{z})=g_{\infty}(-\mathbf{z})$ so $\int_{-q\leq g_{\infty}(\mathbf{z})\leq q}\phi_{\tilde{\Sigma}}(\mathbf{z})p(\mathbf{z})=0$.
Hence from the above displayed equality,
\begin{align}
 & P\left(-q\leq W_{S}\leq q\right)\nonumber \\
= & \int_{-q\leq g_{n}(\mathbf{z})\leq q}\phi_{\tilde{\Sigma}}(\mathbf{z})d\mathbf{z}+O(n^{-1})\nonumber \\
= & P\left(-q\leq t_{K-1}\leq q\right)+\int_{-q\leq g_{n}(\mathbf{z})\leq q}\phi_{\tilde{\Sigma}}(\mathbf{z})d\mathbf{z}-\int_{-q\leq g_{\infty}(\mathbf{z})\leq q}\phi_{\tilde{\Sigma}}(\mathbf{z})d\mathbf{z}+O(n^{-1}).\label{eq: W_B_decomposition}
\end{align}
Here we used $P\left(-q\leq t_{K-1}\leq q\right) = \int_{-q\leq g_{\infty}(\mathbf{z})\leq q}\phi_{\tilde{\Sigma}}(\mathbf{z})d\mathbf{z}$ since the limiting distribution of $W_S$ is $t_{K-1}$. Now it suffices to study the difference between $P(-q\leq g_{n}\left(Z_{0},Z_{1},\dots,Z_{K-1}\right)\leq q)$
and its counterpart as $n\rightarrow\infty$ and show that the difference is $o(n^{-1/2})$, where $Z_{0},Z_{1},\dots,Z_{K-1}$
follows from the limiting normal distribution of $(A_0,B_1,\dots,B_{K-1})$. In particular, we note that $Z_0$ is independent of $Z_{1},\dots,Z_{K-1}$ since $Cov(A_0,B_i)\rightarrow 0$ for each $i=1,2,\dots,K$.  For this probability, we can compute it as follows:
\begin{align*}
 & P(-q\leq g_{n}\left(Z_{0},Z_{1},\dots,Z_{K-1}\right)\leq q)\\
= & P\left(-q\leq\frac{\nabla f^{T}\cdot Z_{0}+\frac{1}{2\sqrt{nK}}Z_{0}^{T}\left[\nabla^{2}f\right]Z_{0}}{\sqrt{\frac{1}{K-1}\sum_{k=1}^{K}\left(\frac{\nabla f^{T}\cdot Z_{k}}{\sqrt{K}}\right)^{2}}}\left(1-\frac{1}{2\sqrt{nK}}\left(Z_{0}^{T}c_b(Z)+c_a(Z)\right)\right)\leq q\right)\\
= & P\left(-q\leq\frac{\nabla f^{T}\cdot Z_{0}+\frac{1}{2\sqrt{nK}}Z_{0}^{T}\left[\nabla^{2}f-c_b(Z)\nabla f^{T}\right]Z_{0}}{\sqrt{\frac{1}{K-1}\sum_{k=1}^{K}\left(\frac{\nabla f^{T}\cdot Z_{k}}{\sqrt{K}}\right)^{2}}}\left(1-\frac{c_a(Z)}{2\sqrt{nK}}\right)\leq q\right)\\
= & P\left(\begin{array}{cc}
     & -q\sqrt{\frac{1}{K-1}\sum_{k=1}^{K}\left(\frac{\nabla f^{T}\cdot Z_{k}}{\sqrt{K}}\right)^{2}}\left({1-\frac{c_a(Z)}{2\sqrt{nK}}}\right)^{-1} \\
    \leq & \nabla f^{T}\cdot Z_{0}+\frac{1}{2\sqrt{nK}}Z_{0}^{T}\left[\nabla^{2}f-c_b(Z)\nabla f^{T}\right]Z_{0} \\
    \leq & q\sqrt{\frac{1}{K-1}\sum_{k=1}^{K}\left(\frac{\nabla f^{T}\cdot Z_{k}}{\sqrt{K}}\right)^{2}}\left({1-\frac{c_a(Z)}{2\sqrt{nK}}}\right)^{-1}
\end{array}\right).
\end{align*}
Here corresponding to $c_a$ and $c_b$, we let $c_a(Z)=\frac{\sum_{k=1}^{K}\frac{\nabla f^{T}\cdot Z_{k}}{K}Z_{k}^{T}\left[\nabla^{2}f\right]Z_{k}}{\sum_{k=1}^{K}\left(\frac{\nabla f^{T}\cdot Z_{k}}{\sqrt{K}}\right)^{2}},c_b=2\frac{\sum_{k=1}^{K}\frac{\nabla f^{T}\cdot Z_{k}}{K}\left[\nabla^{2}f\right]Z_{k}}{\sum_{k=1}^{K}\left(\frac{\nabla f^{T}\cdot Z_{k}}{\sqrt{K}}\right)^{2}}.$ Conditional on $Z_{1},\dots,Z_{K-1}$ (note that the value of $c_a,c_b$ is determined by $Z_{1},\dots,Z_{K-1}$), by normality of the distribution of $Z_0$ and Edgeworth expansion, we
know that the conditional distribution function for $\nabla f^{T}\cdot Z_{0}+\frac{1}{2\sqrt{nK}}Z_{0}^{T}\left[\nabla^{2}f-b\nabla f^{T}\right]Z_{0}$ can be expanded as
$\Phi(q)+n^{-1/2}\tilde{p}_{1}(q)\phi(q)+O(n^{-1})$
 where $\tilde{p}_{1}$ is even (this evenness claim follows since $\nabla f^{T}\cdot Z_{0}+\frac{1}{2\sqrt{nK}}Z_{0}^{T}\left[\nabla^{2}f-b\nabla f^{T}\right]Z_{0}$ is a polynomial with the same form as Theorem 2.1 of \cite{Hall1992}). Thus, conditional on $Z_{1},\dots,Z_{K-1}$,
the above probability is given as (denote $q^{\prime}=q\sqrt{\frac{1}{K-1}\sum_{k=1}^{K}\left(\frac{\nabla f^{T}\cdot Z_{k}}{\sqrt{K}}\right)^{2}}$)
\[
\Phi\left(q^{\prime}\left(1+\frac{c_a(Z)}{2\sqrt{nK}}\right)\right)+n^{-1/2}\tilde{p}_{1}(q^{\prime})\phi(q^{\prime})-\Phi\left(-q^{\prime}\left(1+\frac{c_a(Z)}{2\sqrt{nK}}\right)\right)-n^{-1/2}\tilde{p}_{1}(q^{\prime})\phi(q^{\prime})+O(n^{-1})
\]
which is (by the evenness of $\tilde{p}_{1}$)
\[
\Phi\left(q^{\prime}\right)-\Phi\left(-q^{\prime}\right)+\phi(q^{\prime})q^{\prime}\frac{c_a(Z)}{\sqrt{nK}}+O(n^{-1}).
\]
Taking expectation w.r.t. $(Z_1,\dots,Z_{K-1})$ and note that $E[\Phi\left(q^{\prime}\right)-\Phi\left(-q^{\prime}\right)]=P(-q\leq g_{\infty}(Z_0,Z_1,\dots,Z_{K-1})\leq q)$, we conclude that 
\begin{align*}
 & P(-q\leq g_{n}\left(Z_{0},Z_{1},\dots,Z_{K-1}\right)\leq q)-P(-q\leq g_{\infty}\left(Z_{0},Z_{1},\dots,Z_{K-1}\right)\leq q)\\
= & E_{Z_{1},\dots Z_{K-1}}\left[\phi(q^{\prime})q^{\prime}\frac{a}{\sqrt{nK}}\right]+O(n^{-1})
\end{align*}
Noting that $a$ is odd (and $q^{\prime} $ is even) in $Z_{1},\dots,Z_{k-1}$, we have that
the expectation is 0, so the above difference is indeed $O(n^{-1})$.

To see that $W_{SB}$ has the same property, notice that
\[
nS_{\text{sec}}^2 = nS_{\text{batch}}^2 + \frac{K}{K-1}\left(\sqrt{n}\psi(\hat{P})-\frac{1}{K}\sum_{i=1}^{K}\sqrt{n}\psi(\hat{P}_i) \right)^2 =  nS_{\text{batch}}^2 + O_p(1)
\]
Here the second equality holds since $\sqrt{n}\psi(\hat{P})-\frac{1}{K}\sum_{i=1}^{K}\sqrt{n}\psi(\hat{P}_i)=O(n^{-1/2})$, which can be seen by plugging in expansions \begin{equation*}
\sqrt{nK}(\psi(\hat{P})-\psi(P_0))=\nabla f^{T}\cdot A_{0}+\frac{1}{2\sqrt{nK}}A_{0}^{T}\left[\nabla^{2}f\right]A_{0}+O_p\left(n^{-1}\right)
\end{equation*}
\begin{equation*}
\sqrt{n}(\psi(\hat{P}_{k})-\psi(P_0))=\frac{1}{\sqrt{K}}\nabla f^{T}\left(A_{0}+B_{k}\right)+\frac{n^{-1/2}}{2\sqrt{K}}\left(A_{0}+B_{k}\right)^{T}\left[\nabla^{2}f\right]\left(A_{0}+B_{K}\right)+O_p\left(n^{-1}\right),k=1,2,\dots,K
\end{equation*} and using the equation $\sum_{i=1}^K B_i=0$. This implies $W_{SB}=W_S+O_p(n^{-1})$. Since we have shown $P(-q\leq W_S \leq q) = O(n^{-1})$, the same holds for $W_{SB}$. 

Next, we consider sectioned jackknife. By examining the expansion for $W_{SJ}$ in Section \ref{sec: expansion other schemes}, we have that similar to \eqref{eq: 2nd_taylor},
\begin{align}
 W_{SJ}& =\frac{\nabla f^{T}\cdot A_{0}+\frac{1}{2\sqrt{nK}}A_{0}^{T}\left[\nabla^{2}f\right]A_{0}+n^{-1/2}\tilde{c}_d}{\sqrt{\frac{1}{K-1}\sum_{k=1}^{K}\left(\frac{\nabla f^{T}\cdot A_{k}}{\sqrt{K}}\right)^{2}}}\left(1-\frac{1}{2\sqrt{nK}}\left(A_{0}^{T}\tilde{c}_b+\tilde{c}_a\right)\right)+o_{p}\left(n^{-1/2}\right).
\end{align}
where $\tilde{c}_a$ is odd in $(B_1,\dots,B_K)$ and $\tilde{c}_b,\tilde{c}_d$ are even in $(B_1,\dots,B_K)$. Compared to sectioning, the difference is that we have a new term $\tilde{c}_d$. Following the argument for sectioning, it suffices to show that the difference between
\[
P\left(-\frac{q^\prime}{1-\frac{\tilde{c}_a(Z)}{2\sqrt{nK}}}-n^{-1/2}\tilde{c}_d(Z)\leq\nabla f^{T}\cdot Z_{0}+\frac{1}{2\sqrt{nK}}Z_{0}^{T}\left[\nabla^{2}f-\tilde{c}_b(Z)\nabla f^{T}\right]Z_{0}\leq\frac{q^\prime}{1-\frac{\tilde{c}_a(Z)}{2\sqrt{nK}}}-n^{-1/2}\tilde{c}_d(Z)\right)
\]
and $P(-q\leq g_{\infty}(Z_0,Z_1,\dots,Z_{K-1})\leq q)$ is $o(n^{-1/2})$. But note that the above probability is 
\begin{align*}
    & \Phi\left(q^{\prime}\left(1+\frac{\tilde{c}_a(Z)}{2\sqrt{nK}}\right)-n^{-1/2}\tilde{c}_d(Z)\right)+n^{-1/2}\tilde{p}_{1}(q^{\prime})\phi(q^{\prime}) \\
    & -\Phi\left(-q^{\prime}\left(1+\frac{\tilde{c}_a(Z)}{2\sqrt{nK}}\right)-n^{-1/2}\tilde{c}_d(Z)\right)-n^{-1/2}\tilde{p}_{1}(q^{\prime})\phi(q^{\prime})+O(n^{-1})
\end{align*}
From this, we can see that the contribution of $\tilde{c}_d$ to the $n^{-1/2}$ error is cancelled. Therefore, proceeding as in the proof for sectioning to handle $\tilde{c}_a$, we get the desired result.

\subsection{Proof of Lemma \ref{lem: taylor_approximation}}

First, we give a lemma that helps bound the moments for the inverse
of $E_{2}$. The proof is deferred to Section \ref{subsec: proof_validity_smaller}.
\begin{lem}
\label{lem: finite_expectation}Under our assumptions in Theorem \ref{thm: validity}, we have that
for any $\epsilon<1$, there exists $c>0$ such that $\tilde{E}\left[\left|E_{2}^{-(r+1+\epsilon)/2}\right|^{1+c}\right]<\infty$.
\end{lem}

Then, we give the proof of Lemma \ref{lem: taylor_approximation}:
\begin{proof}
Consider $F_{+}^{(n)}$ for example. Note that when $W_{SJ}=q$, we
have that $\sqrt{nK}\left(\bar{J}-\psi_{0}\right)=q\sqrt{\frac{1}{K-1}\sum_{i=1}^{K}(\sqrt{n}J_{i}-\sqrt{n}\bar{J})^{2}}.$
We can do a Taylor expansion for each $\psi(\hat{P}_{(i)})=f\left(m+n^{-1/2}\left(\frac{KA_{0}-A_{i}}{K-1}\right)\right)$
and $\psi(\hat{P})$:
\begin{align*}
\psi(\hat{P}_{(i)}) & =\psi_{0}+\sum_{m=1}^{r}n^{-m/2}\sum_{1\leq i_{1}\dots i_{m}\leq d}\nabla f_{i_{1}i_{2}\dots i_{m}}(m)\left(\frac{KA_{0}-A_{i}}{K-1}\right)_{i_{1}}\dots\left(\frac{KA_{0}-A_{i}}{K-1}\right)_{i_{m}}\\
 & +n^{-(r+1)/2}\sum_{1\leq i_{1}\dots i_{m}\leq d}\nabla f_{i_{1}i_{2}\dots i_{r+1}}(\xi)\cdot\left(\frac{KA_{0}-A_{i}}{K-1}\right)_{i_{1}}\dots\left(\frac{KA_{0}-A_{i}}{K-1}\right)_{i_{r+1}}
\end{align*}
for some $\xi$ on the line segment between $m$ and $m+n^{-1/2}\left(\frac{KA_{0}-A_{i}}{K-1}\right)$.
Moreover, the residual $R$ can be bounded by $C\sum_{i=0}^{K}\left\Vert A_{i}\right\Vert ^{r+1}$
where $C$ and $C_{i}$ in the sequel are constants independent of
$n$. By the continuity of the $(r+1)$-th derivative of $f$, we
may pick some $\delta>0$ such that $\sup_{\left\Vert x-m\right\Vert \leq\delta}\left\Vert \nabla^{r+1}f\right\Vert $
is bounded. Note that when $\left\Vert n^{-1/2}\left(\frac{KA_{0}-A_{i}}{K-1}\right)-m\right\Vert \leq\delta$,
the last term in the preceding displayed equation can be bounded by
$Cn^{-(r+1)/2}\sup_{\left\Vert x-m\right\Vert \leq\delta}\left\Vert \nabla^{r+1}f\right\Vert \sum_{i=0}^{K}\left\Vert A_{i}\right\Vert ^{r+1}$.
Since we assumed that $X$ has finite $r+2$ moments, (by Markov's
inequality) we can show that the probability that this $\left\Vert n^{-1/2}\left(\frac{KA_{0}-A_{i}}{K-1}\right)-m\right\Vert >\delta$
is of order $O(n^{-(r+1)/2})$. Therefore, we conclude that with probability
$1-O(n^{-(r+1)/2})$, the residual of the Taylor expansion is bounded
by $Cn^{-(r+1)/2}\sum_{i=0}^{K}\left\Vert A_{i}\right\Vert ^{r+1}$
for some $C$ independent of $n$. Similar expansion holds for $\psi(\hat{P})$.
With these expansions, notice that each of $J_{i}$ and $\bar{J}$
can be written as linear combinations of $\psi(\hat{P}_{(i)})$ and
$\psi(\hat{P})$, we also have an expansion for them:
\[
J_{i}=\psi_{0}+n^{-1/2}\left(\nabla f\right)^{T}A_{i}+\sum_{j=2}^{r}n^{-j/2}p_{i,j}(A)+n^{-(r+1)/2}R_{i}
\]
\[
\bar{J}=\psi_{0}+n^{-1/2}\left(\nabla f\right)^{T}A_{0}+\sum_{j=2}^{r}n^{-j/2}\bar{p}_{j}(A)+n^{-(r+1)/2}\bar{R}
\]
where each $p_{i,j}(A)$ is a polynomial in $A$
of order $j$, $R_{i}$ can be bounded by $C_{1}\sum_{i=0}^{K}\left\Vert A_{i}\right\Vert ^{r+1}$
(with probability $1-O(n^{-(r+1)/2})$), $\bar{p}_{j}$ and $\bar{R}$
are the average of $p_{i,j},R_{i},1\leq i\leq K$ respectively. Therefore,
when $W_{SJ}=q$ or equivalently $\sqrt{nK}\left(\bar{J}-\psi_{0}\right)=q\sqrt{\frac{1}{K-1}\sum_{i=1}^{K}(\sqrt{n}J_{i}-\sqrt{n}\bar{J})^{2}}$,
we have (using the expansion for $J_{i}$ and $\bar{J}$ above)
\begin{equation}
\left|\sqrt{K}\left(\nabla f\cdot A_{0}+\sum_{j=1}^{r}n^{-\frac{j}{2}}\bar{p}_{j+1}(A)\right)-q\sqrt{\frac{E_{2}+\sum_{j=1}^{r}n^{-\frac{j}{2}}p_{j+1}^{(1)}(A)+n^{-\frac{r+1}{2}}R_{2}}{K-1}}\right|\leq n^{-\frac{r+1}{2}}R_{3}\label{eq: Taylor_p}
\end{equation}
where $E_{2}=\text{\ensuremath{\sum_{i=1}^{K}\left(\left(\nabla f\right)^{T}\left(A_{i}-A_{0}\right)\right)^{2}}},p_{2}^{(1)}(A)=\sum_{i=1}^{K}\left(\nabla f\right)^{T}\left(A_{i}-A_{0}\right)q_{i}(A)$,
$q_{i}(A)$ is a polynomial of $A$. For for each $j=3,4,\dots,r+1$,
$p_{j}^{(1)}(A)$ is a polynomial of $A$, and $\left|R_{2}\right|,\left|R_{3}\right|$
can be bounded by $C_{2}\sum_{i=0}^{K}\left\Vert A_{i}\right\Vert ^{r+1}$
(with probability $1-O(n^{-(r+1)/2})$).

Observe that 
\[
\sqrt{E_{2}+\sum_{j=1}^{r}n^{-\frac{j}{2}}p_{j+1}^{(1)}(A)+n^{-\frac{r+1}{2}}R_{2}}=\sqrt{E_{2}}\sqrt{1+\frac{\sum_{j=1}^{r}n^{-\frac{j}{2}}p_{j+1}^{(1)}(A)}{E_{2}}+\frac{n^{-\frac{r+1}{2}}R_{2}}{E_{2}}}
\]
By a Taylor expansion for the composition of $\sqrt{1+x}$ with $1+\frac{\sum_{j=1}^{r}n^{-\frac{j}{2}}p_{j+1}^{(1)}(A)}{E_{2}}+\frac{n^{-\frac{r+1}{2}}R_{2}}{E_{2}}$,
we have that the above can be expanded as follows
\begin{equation}
\sqrt{E_{2}}\left(1+\sum_{j=1}^{r}n^{-\frac{j}{2}}\frac{p_{j+1}^{(2)}(A)}{E_{2}^{j}}+n^{-\frac{r+1}{2}}\frac{R_{4}}{E_{2}^{r+1}}\right).\label{eq: Taylor_denom}
\end{equation}
Here, $p_{j+1}^{(2)}(A)$ can be bounded by $E_{2}^{j/2}$
times a polynomial of $A$ and $R_{4}$ can be bounded by $E_{2}^{(r+1)/2}$
times a polynomial of $A$. (The proof is given in Section \ref{subsec: proof_validity_smaller}).
Therefore, by Lemma \ref{lem: finite_expectation}, each of $\frac{p_{j+1}^{(2)}(A)}{E_{2}^{j}},1\leq j\leq r$
and $\frac{R}{E_{2}^{r+1}}$ has finite $1+c$ moment for some $c>0$. Putting
(\ref{eq: Taylor_denom}) into (\ref{eq: Taylor_p}), we get that
\[
\left|\sqrt{K}\left(\nabla f\cdot A_{0}+\sum_{j=1}^{r}n^{-\frac{j}{2}}\bar{p}_{j+1}(A)\right)-q\sqrt{E_{2}}\left(1+\sum_{j=1}^{r}n^{-\frac{j}{2}}\frac{p_{j+1}^{(2)}(A)}{E_{2}^{j}}\right)\right|\leq n^{-\frac{r+1}{2}}R_{3}.
\]
Therefore, with an error of $(\frac{1}{\sqrt{K}\nabla_1f}+o(1)(n^{-\frac{r+1}{2}}R_{3}$ (note that $E_2$ does not depend on $A_{0,1}$ so the leading term for the derivative of $A_{0,1}$ is $\sqrt{K}\nabla_1f$), we can find
the critical value for $A_{0,1}$ that makes $W_{SJ}=q$ (i.e., $F_{+}^{(n)}$) by solving
\begin{equation}
\sqrt{K}\left(\nabla f\cdot A_{0}+\sum_{j=1}^{r}n^{-\frac{j}{2}}\bar{p}_{j+1}(A)\right)-q\sqrt{E_{2}}\left(1+\sum_{j=1}^{r}n^{-\frac{j}{2}}\frac{p_{j+1}^{(2)}(A)}{E_{2}^{j}}\right)=0\label{eq: implicit function}
\end{equation}
Denote the LHS above as $F(A_{0,1},n^{-1/2})$. Observe that $F(A_{0,1},n^{-1/2})$
is a polynomial whose coefficients are determined by $A^{\prime}.$
We are interested in solving $A_{0,1}$ as a function of $n^{-1/2}$
(denote by $y(x)$) and get the derivatives of $y(\cdot)$ via implicit
function theorem. We have that
\[
F_{x}+F_{y}y_{x}=0\Rightarrow y_{x}=-F_{x}/F_{y},
\]
\[
F_{xx}+2F_{xy}y_{x}+F_{y}y_{xx}=0\Rightarrow y_{xx}=-\left(F_{xx}+F_{xy}y_{x}\right)/F_{y}.
\]
Here $F_{x},F_{y}$ stands for the derivative of $F$ with its first
and second argument, respectively. From (\ref{eq: implicit function}),
we have that $\frac{\partial^{m+l}}{\partial x^{m}\partial y^{l}}F|_{x=0}$
has form $\sqrt{E_{2}}\frac{{p}_{m}^{(3)}(A^{\prime})}{E_{2}^{m}}$
where $\left|{p}_{m}^{(3)}(A^{\prime})\right|$ can be bounded by
$E_{2}^{m/2}$ times a polynomial of $A^{\prime}$ (which follows
from \eqref{eq: Taylor_denom}). Also note that $F_{y}=\sqrt{K}\left(\nabla f\right)_{1}\neq0$.
With this observation, by induction similar to the above displayed
derivations, we can show that $\frac{d^{m}}{dx^{m}}y(x)|_{x=0},1\leq m\leq r+1$
has form $\sqrt{E_{2}}\frac{{p}_{m}^{(4)}(A^{\prime})}{E_{2}^{m}}$
where $\left|{p}_{m}^{(4)}(A^{\prime})\right|$ can also
be bounded by $E_{2}^{m/2}$ times a polynomial of $A^{\prime}$.
Therefore, by Taylor expansion on $y(\cdot)$ around 0, we have 
\[
F_{+}^{(n)}=y(n^{-1/2})=F_{+}+\sum_{m=1}^{r}n^{-m/2}\frac{1}{m!}\sqrt{E_{2}}\frac{\tilde{p}_{m}^{\prime}(A^{\prime})}{E_{2}^{m}}+n^{-(r+1)/2}R_{5}
\]
Moreover, the residual $R_{5}$ has form $\sqrt{E_{2}}\frac{\tilde{p}_{r+1}^{\prime}(A^{\prime})}{E_{2}^{r+1}}$
where $\left|\tilde{p}_{r+1}^{\prime}(A^{\prime})\right|$ can be
bounded by $E_{2}^{(r+1)/2}$ times a polynomial of $A^{\prime}$ (similar to the argument for the form of the residual term in the proof of \eqref{eq: Taylor_denom}).
which has finite $1+c$ moment by Lemma \ref{lem: finite_expectation}.

More generally, for $\left(F_{+}^{(n)}-F_{+}\right)^{l}$, since it
can be written as $\left(y(n^{-1/2})-y(0)\right)^{l}$, with the derivatives
for $y$, we also have its expansion of form (heuristically, we can
see the following expansion by taking the $l$-th power in the preceding
expansion, but that would induce a power of $R_{4}$ which may not
have finite expectation)
\[
\left(F_{+}^{(n)}-F_{+}\right)^{l}=\sum_{m=l}^{r}n^{-m/2}\sqrt{E_{2}}\frac{{p}_{l,m}(A^{\prime})}{E_{2}^{m}}+n^{-(r+1)/2}R_{F,l}.
\]
where $\left|{p}_{l,m}(A^{\prime})\right|$ can
be bounded by $E_{2}^{m/2}$ times a polynomial of $A^{\prime}$ and
the residual $R_{F,l}$ has form $\sqrt{E_{2}}\frac{{p}_{l,r+1}(A^{\prime})}{E_{2}^{r+1}}$
where $\left|{p}_{l,r+1}(A^{\prime})\right|$ can be
bounded by $E_{2}^{(r+1)/2}$ times a polynomial of $A^{\prime}$.
This gives the form of expansion as claimed. The finiteness of $1+c$
moment follows from Lemma (\ref{lem: finite_expectation}) (note that
the existence of $1+c$ moment preserves when we multiply a polynomial
of $A^{\prime}$).
\end{proof}

\subsubsection{Proof of Lemma \ref{lem: finite_expectation}\label{subsec: proof_validity_smaller}}

\begin{proof}
It suffices to show that $E\left[E_{2}^{-(r+1+\epsilon)/2}\right]<\infty$
for any $\epsilon<1$ (once we can show this, the claim holds for
any $c$ such that $(r+1+\epsilon)(1+c)/2<r+2$). Note that the limiting
distribution of $E_{2}^{-1}$ follows a inverse-chi-squared distribution
with $K-1$ degrees of freedom, whose density is proportional to $x^{-K/2-1}e^{-1/2x}$.
When $K\geq r+3$, 
\[
\int_{1}^{\infty}x^{-K/2-1}e^{-1/2x}x^{(r+1+\epsilon)/2}dx\leq\int_{1}^{\infty}x^{-(3-\epsilon)/2}e^{-1/2x}dx<\infty
\]
also, because of the $e^{-1/2x}$ term in the expression of the density,
it is clear that $\int_{0}^{1}x^{-K/2-1}e^{-1/2x}x^{(r+1+\epsilon)/2}dx<\infty$.
Therefore, $\int_{0}^{\infty}x^{-K/2-1}e^{-1/2x}x^{(r+1)/2}dx<\infty$,
which means that under the limiting distribution, $E_{2}^{-(r+1+\epsilon)/2}$
has finite expectation. This means
\[
\int E_{2}\left(a\right)^{-(r+1+\epsilon)/2}\phi_{\Sigma}(a_{1})\dots\phi_{\Sigma}(a_{K})da_{1}\dots da_{K}<\infty.
\]
Here $E_{2}(a_1,\dots,a_K):=\sum_{i=1}^{K}\left(\nabla f(EX)^{\top}(a_i-\bar{a})\right)^2$ is the function such that $E_{2}(A_1,\dots,A_K)=E_{2}$.
Moreover, for any $c$ such that $(r+1+\epsilon)(1+c)/2<r+2$ we still
have
\[
\int\left[E_{2}\left(a\right)^{-(r+1+\epsilon)/2}\right]^{1+c}\phi_{\Sigma}(a_{1})\dots\phi_{\Sigma}(a_{K})da_{1}\dots da_{K}<\infty.
\]
Therefore, for any $p(a)$ that is polynomial in $a$, by Holder's
inequality and the fact that $p(A)$ has an arbitrary order of moments
under the limiting distribution, we have
\begin{equation}
\int E_{2}\left(a\right)^{-(r+1+\epsilon)/2}p(a)\phi_{\Sigma}(a_{1})\dots\phi_{\Sigma}(a_{K})da_{1}\dots da_{K}<\infty.\label{eq: inv_chi_poly}
\end{equation}
Observing that $\tilde{E}\left[E_{2}^{-(r+1+\epsilon)/2}\right]$
can be written as 
\[
\int E_{2}\left(a\right)^{-(r+1+\epsilon)/2}\left(1+\sum_{i}n^{-i/2}p_{i}^{(5)}(a)\right)\phi_{\Sigma}(a_{1})\dots\phi_{\Sigma}(a_{K})da_{1}\dots da_{K},
\]
from (\ref{eq: inv_chi_poly}) we conclude the desired result.
\end{proof}

\begin{proof}[Proof of \eqref{eq: Taylor_denom}]
Let $z(h):=1+\frac{\sum_{j=1}^{r}h^{j}p_{j+1}^{(1)}(A)}{E_{2}}+\frac{h^{r+1}R_{2}}{E_{2}}$
and $r(x):=\sqrt{1+x}$. By Taylor expansion, we have that
\begin{equation}\label{eq: rz expansion}
    r(z(n^{-1/2})) = r(z(0)) + \sum_{k=1}^r\frac{d^k}{dh^k}(r\circ z)(h)|_{h=0}n^{-k/2} + \frac{d^{r+1}}{dh^{r+1}}(r\circ z)(h)|_{h=h^{\prime}}n^{-(r+1)/2}
\end{equation}
for some $0<h^{\prime}<n^{-1/2}$. Therefore, to get \eqref{eq: Taylor_denom}, it suffices to argue that the derivatives of $r\circ z$ have the claimed forms.

By the chain rule, we have
\[
\frac{d}{dh}\left(r\circ z\right)\left(h\right)=r^{\prime}(z(h))z^{\prime}(h)
\]
\[
\frac{d^2}{dh^{2}}\left(r\circ z\right)\left(h\right)=r^{\prime\prime}(z(h))\left(z^{\prime}(h)\right)^{2}+r^{\prime}(z(h))z^{\prime\prime}(h)
\]
In general, by induction, it is not hard to show that $\frac{d}{dh^{k}}\left(r\circ q\right)\left(h\right)$
can be written as 
\begin{equation}
\frac{d^k}{dh^{k}}\left(r\circ q\right)\left(h\right)=\sum_{
    \mathbf{i}=(i_{1},\dots,i_{l}),i_{1}+\dots+i_{l}=k,i_1,\dots,i_l\in\{1,2,\dots,k\}  
}C_{\mathbf{i}}r^{[i_{0}]}(z(h))z^{[i_{1}]}(h)\dots z^{[i_{l}]}(h).\label{eq: derivative_compo}
\end{equation}
Here $r^{[i_0]}$ denote the $i_0$-th derivative of $r$ and $z^{[i_1]},\dots,z^{[i_l]}$ are similarly defined. Note that $\left|p_{2}^{(1)}(A)\right|=\left|\sum_{i=1}^{K}\left(\nabla f\right)^{T}\left(A_{i}-A_{0}\right)q_{i}(A)\right|\leq\sqrt{E_{2}}\left|\sum_{i=1}^{K}\left\Vert \nabla f\right\Vert _{2}q_{i}(A)\right|$ (see the expression for $p_2^{(1)}(A)$ after \eqref{eq: Taylor_p}),
we have $z^{\prime}(0)=\frac{p_{2}^{(1)}(A)}{E_{2}}$ where $p_{2}^{(1)}(A)$
can be bounded by $\sqrt{E_{2}}$ times a polynomial of $A$. For
any $k\geq2$, from the expression of $z(h)$ we also have that $z^{[k]}(0)=\frac{p_{k+1}^{(5)}(A)}{E_{2}^{k}}$
where $p_{k+1}^{(5)}(A)$ can be bounded by $E_2^{k-1}$ times a polynomial of $A$. Note that when $k\geq 2$ we have $k-1\geq k/2$ and $E_2$ itself is bounded by a polynomial of $A$, we also have that $p_{k+1}^{(2)}(A)$ can be bounded by $E_2^{k/2}$ times a polynomial of $A$.
Therefore, for any $i_{1}\dots i_{l}$ such that $i_{1}+\dots+i_{l}=k$,
$z^{[i_{1}]}(0)\dots z^{[i_{l}]}(0)=\frac{p_{\mathbf{i}}^{(6)}(A)}{E_{2}^{k}}$
where $p_{\mathbf{i}}^{(3)}(A)$ can be bounded by $E_{2}^{k/2}$
times a polynomial of A. Also note that $r^{(i_{0})}(z(0))$ is a constant that does not depend on $A$. So
by \eqref{eq: derivative_compo} we have that $\frac{d^k}{dh^{k}}\left(r\circ q\right)\left(0\right)=\frac{p_{k+1}^{(7)}(A)}{E_{2}^{k}}$
where $p_{k+1}^{(7)}(A)$ can be bounded by $E_{2}^{k/2}$ times a
polynomial of A. From this, by \eqref{eq: rz expansion}, we get \eqref{eq: Taylor_denom} except for the residual. 

The additional difficulty for handling the residual is that the derivatives are evaluated at a point that is not 0. But as we will show, the knowledge that $0<h^{\prime}<n^{-1/2}$ suffices to control the residual.
When $k\geq2$, we note that
\[
z^{[k]}(h)=\frac{\sum_{j=k}^{r}{k \choose j}h^{j-k}p_{j+1}^{(1)}(A)}{E_{2}}+{k \choose r+1}\frac{h^{r+1-k}R_{2}}{E_{2}}
\]
When $h<1/2$, we have that
\begin{equation}
\left|z^{[k]}(h)\right|\leq\sum_{j=k}^{r}\frac{{k \choose j}\left|p_{j+1}^{(1)}(A)\right|}{E_{2}}+{k \choose r+1}\frac{R_{2}}{E_{2}}.\label{eq: dz_h}
\end{equation}
The RHS can be bounded by $\frac{p_{k+1}^{(8)}(A)}{E_{2}^{k}}$ where
$p_{k+1}^{(8)}(A)$ can be bounded by $E_{2}^{k/2}$ times a polynomial
of A, as in the $h=0$ case discussed above. When $k=1$, we have
that 
\[
z^{\prime}(h)=z^{\prime}(0)+h\left[\frac{\sum_{j=2}^{r}jh^{j-2}p_{j+1}^{(1)}(A)}{E_{2}}+(r+1)\frac{h^{r-1}R_{2}}{E_{2}}\right]
\]
It follows from Lemma \ref{lem: finite_expectation} that $P(E_{2}^{-1/2}>n^{1/2})=O(n^{-(r+1)/2})$.
Therefore, with probability $1-O(n^{-(r+1)/2})$ we can assume that
$n^{-1/2}<\sqrt{E_{2}}$. Therefore, for any $h<n^{-1/2}$, 
\begin{equation}
\left|z^{\prime}(h)\right|\leq\left|z^{\prime}(0)\right|+\sum_{j=k}^{r}\frac{{k \choose j}\left|p_{j+1}^{(1)}(A)\right|}{E_{2}^{1/2}}+{k \choose r+1}\frac{R_{2}}{E_{2}^{1/2}}.\label{eq: dz_h_1}
\end{equation}
 Since we have shown that $z^{\prime}(0)=\frac{p_{2}^{(1)}(A)}{E_{2}}$
where $p_{2}^{(1)}(A)$ can be bounded by $\sqrt{E_{2}}$ times a
polynomial of $A$, from the above we have that $\left|z^{\prime}(h)\right|\leq\frac{p_{2}^{(8)}(A)}{E_{2}}$
where $p_{2}^{(8)}(A)$ can be bounded by $\sqrt{E_{2}}$ times a
polynomial of $A$. Moreover, note that for any $h<n^{-1/2}$,
\begin{align*}
\left|z(h)-1\right| & =\left|\frac{\sum_{j=1}^{r}h^{j}p_{j+1}^{(1)}(A)}{E_{2}}+\frac{h^{r+1}R_{2}}{E_{2}}\right|\leq\sum_{j=1}^{r}h^{j}\left|\frac{p_{j+1}^{(1)}(A)}{E_{2}}\right|+h^{r+1}\frac{R_{2}}{E_{2}}\\
 & \leq h\frac{\left|\sum_{i=1}^{K}\left\Vert \nabla f\right\Vert _{2}q_{i}(A)\right|}{\sqrt{E_{2}}}+\sum_{j=2}^{r}h^{j}\left|\frac{p_{j+1}^{(1)}(A)}{E_{2}}\right|+h^{r+1}\frac{R_{2}}{E_{2}}\\
 & \leq n^{-\left(1-\frac{r+1}{r+1+\epsilon}\right)/2}\left|\sum_{i=1}^{K}\left\Vert \nabla f\right\Vert _{2}q_{i}(A)\right|+n^{-\left(j-\frac{2(r+1)}{r+1+\epsilon}\right)/2}\sum_{j=2}^{r}\left|p_{j+1}^{(1)}(A)\right|+n^{-\left(r+1-\frac{2(r+1)}{r+1+\epsilon}\right)/2}R_{2}\\
 & (\text{when } n^{-1/2}<E_{2}^{\frac{r+1+\epsilon}{2(r+1)}})
\end{align*}
By Lemma \ref{lem: finite_expectation} we have that $P(n^{-1/2}<E_{2}^{\frac{r+1+\epsilon}{2(r+1)}})=1-O(n^{-(r+1)/2})$ when $\epsilon>0$ is sufficiently small.
Note that the RHS above is smaller than $1/2$ with high probability
(i.e., with probability $1-O(e^{-n^{\alpha}})$ for some $\alpha>0$
since it has the form $n^{-\alpha}$ times polynomials of $A$, which
has moments of arbitrary order under $\tilde{P}$). Therefore, with
probability $1-O(n^{-(r+1)/2})$, we have that $\left|z(h)-1\right|<1/2$.
This gives an absolute bound for $r^{[i_{0}]}(z(h))$. With this and
the bounds for $z^{\prime}(h)$ and $z^{[k]}(h)$, $k\geq2$ that
we derived, by (\ref{eq: derivative_compo}) we have that the residual of \eqref{eq: rz expansion} gives the claimed form of residual in \eqref{eq: Taylor_denom}.
\end{proof}

\subsection{Proof of \eqref{eq: Edgeworth_approx}}\label{subsec: edgeworth details}
Consider the i.i.d. sequence \[\mathrm{X}_i:=(X_{i},X_{n+i},\dots,X_{(K-1)n+i}),i=1,2,\dots,n\] Then, the vector of all batch averages can be seen as the average of $\mathrm{X}_i,i=1,2,\dots,n$. With the imposed conditions on $X_i$, it is not hard to verify that the distribution of $\mathrm{X}_i$ satisfies the conditions in Theorem \ref{thm: Edgeworth}. We regard the indicator $I_{SJ}^{(n)}({\bar{X}}_1,\bar{X}_2,\dots,\bar{X}_K):=I(-q\leq W_{SJ}\leq q)$ as a function of all the batch averages, which is in turn a function of $\bar{\mathrm{X}}$.  Therefore, we may let $f$ in Theorem \ref{thm: Edgeworth} be this function. Now we are in the setting of Theorem \ref{thm: Edgeworth}, so it suffices to show that the residual guaranteed by Theorem \ref{thm: Edgeworth} is of order $O(n^{-(r+1)/2})$. We choose $s=r+3$ in Theorem \ref{thm: Edgeworth}. Then, since $f$ is bounded by 1, we have that $M_{s^{\prime}}(f)\leq 1$ which implies that $M_{s^{\prime}}(f)\delta_1(n)=o(n^{-(r+1)/2})$. Next, we show that the other residual term given in Theorem \ref{thm: Edgeworth}, i.e., $\bar{\omega}_{f}\left(2 e^{-d n}: \Phi_{0, V}\right)=O(n^{-(r+1)/2})$, is bounded by $O(n^{-(r+1)/2})$. Here, $V$ is the product distribution of $K$ r.v. with covariance equal to the covariance of $X_1$. This is heuristically correct since this term represents the oscillation of $f$ in an exponentially small neighborhood.

Consider the function of batch averages induced by the SJ statistic
\[
W_{SJ}(x):=\frac{\sqrt{K}\left(\bar{J}(x)-f(EX)\right)}{\sqrt{\frac{1}{K-1}\sum_{i=1}^{K}(J_{i}(x)-\bar{J}(x))^{2}}}
\]
where $J_i(x)=Kf(\bar{x})-(K-1)f(x_{(i)})$
For its gradient, we have that 
\begin{equation}
\nabla W_{SJ}(x)=\frac{\sqrt{K(K-1)}\nabla\bar{J}}{\sqrt{\sum_{i=1}^{K}(J_{i}(x)-\bar{J}(x))^{2}}}-\frac{\sqrt{K(K-1)}\left(\bar{J}(x)-f(EX)\right)\sum_{i=1}^{K}(J_{i}(x)-\bar{J}(x))\left(\nabla J_{i}-\nabla\bar{J}\right)}{\sum_{i=1}^{K}(J_{i}(x)-\bar{J}(x))^{2}\sqrt{\sum_{i=1}^{K}(J_{i}(x)-\bar{J}(x))^{2}}}\label{eq: dWSJ}
\end{equation}

Since our target is to bound $\omega_{I_{SJ}}(2e^{-dn}:\Phi_{0,V})$ which is the oscillation under the limiting normal distribution, in what follows, we suppose all $X_i$ are normal with variance $\Sigma$. By Gaussian concentration, with probability $1-O(n^{-(r+1)/2})$ we have that
$\left\Vert \bar{X}_{i}-EX\right\Vert \leq\delta$ for each $i=1,2,\dots,K$ where
$\delta$ is a value such that $\sup_{\left\Vert x-EX\right\Vert \leq\delta}\nabla f(x)\leq C$.
Moreover, from the assumption that $Var_{P}X$ is nonsingular and $\nabla f\neq0$,
it is not hard to show that there exists $\lambda>0$ such that $P\left(\sup_{\left\Vert x-X\right\Vert \leq e^{-\gamma n}}\sum_{i=1}^{K}(J_{i}(x)-\bar{J}(x))^{2}<n^{-\lambda}\right)=O(n^{-(r+1)/2})$.
Therefore, from the preceding displayed equation, we conclude that
there exist $\lambda>0$ such that $P\left(\sup_{\left\Vert x-X\right\Vert \leq e^{-\gamma n}}\left\Vert \nabla W_{SJ}(x)\right\Vert \leq n^{\lambda}C\right)=1-O(n^{-(r+1)/2})$.
Therefore, by Lagrange mean value theorem, we have that 
\begin{equation}
P\left(\sup_{\left\Vert x-X\right\Vert \leq e^{-\gamma n}}\left\Vert W_{SJ}(x)-W_{SJ}(X)\right\Vert \leq n^{\lambda}e^{-\gamma n}C\right)=1-O(n^{-(r+1)/2}).\label{eq: WSJ osc bound}
\end{equation}
Pick any $0<\gamma^{\prime}<\gamma$. Note that $P\left(q-e^{-\gamma^{\prime}n}\leq W_{SJ}(X)\leq q+e^{-\gamma^{\prime}n}\right)=O(n^{-(r+1)/2})$
since otherwise $W_{SJ}$ would have unbounded density around $q$.
Similarly $P\left(-q-e^{-\gamma^{\prime}n}\leq W_{SJ}(X)\leq-q+e^{-\gamma^{\prime}n}\right)=O(n^{-(r+1)/2})$.
Note also that $n^{\lambda}e^{-\gamma n}C$ can be bounded by $Ce^{-\gamma^{\prime}n}$.
Therefore, from (\ref{eq: WSJ osc bound}) we get 
\[
P\left(\sup_{\left\Vert x-X\right\Vert \leq e^{-\gamma n}}\left\Vert I(-q\leq W_{SJ}(x)\le q)-I(-q\leq W_{SJ}(X)\leq q)\right\Vert =1\right)=O(n^{-(r+1)/2})
\]
which implies that $\bar{\omega}_{I_{SJ}}(e^{-\gamma n},\Phi_{0,V})=O(n^{-(r+1)/2}).$
\section{Proofs of Other Theorems}\label{sec: proof other}

\begin{proof}[Proof of \eqref{eq: WB integration}]
Denote $Z_{i}=\sqrt{n}\left(\psi(\hat{P}_{i})-\psi\right)/\sigma,i=1,2,\dots,K$.
Let $\tilde{P}$ be the measure (possibly signed when $n$ is
small) such that 
\[
\tilde{P}\left((-\infty,q]\right)=\Phi(q)+\sum_{j=1}^{r}n^{-j/2}p_{j}(q)\phi(q)
\]
Let $P_{0}$ be the measure induced by the distribution of $Z_{i}$,
i.e., $P_{0}\left((-\infty,q]\right)=P(Z_{i}\leq q)$. It suffices
to show that 
\begin{equation}
\left|P_{0}^{\otimes n}\left(\{z_{1},\dots,z_{K}:f(z)\leq q\}\right)-\tilde{P}^{\otimes n}\left(\{z_{1},\dots,z_{K}:f(z)\leq q\}\right)\right|=O(n^{-(r+1)/2})\label{eq: thm1_toshow}
\end{equation}
Here $P^{\otimes n}$ stands for the product measure of $n$ copies
of $P$. Then, note that $f(z_{1},\dots,z_{K})\leq q\iff(K-1)/K\left(z_{1}+\dots+z_{K}\right)^{2}\leq q\sum_{i=1}^{K}\left(z_{i}-\frac{1}{K}\sum_{j}z_{j}\right)^{2}$
which can be formulated as 
\[
z_{1}\in\left[z^{-}(z_{2},\dots,z_{K}),z^{+}(z_{2},\dots,z_{K})\right]
\]
or 
\[
z_{1}\in(-\infty,z^{-}(z_{2},\dots,z_{K})]\cup[z^{+}(z_{2},\dots,z_{K}),\infty)
\]
where $z^{+},z^{-}$ are the two roots of $z_{1}$ that solve $(K-1)/K\left(z_{1}+\dots+z_{K}\right)^{2}=q\sum_{i=1}^{K}\left(z_{i}-\frac{1}{K}\sum_{j}z_{j}\right)^{2}$.
Therefore, applying \eqref{eq: Edgeworth assumption}  at $z^{+}$ and $z^{-}$, by the uniformity
assumption we have that there exists a deterministic $C$ such that
\[
\left|P_{0}(\{z_{1}:f(z_{1},z_{2},\dots,z_{K})\leq q\})-\tilde{P}(\{z_{1}:f(z_{1},z_{2},\dots,z_{K})\leq q\})\right|\leq Cn^{-(r+1)/2},\forall z_{2},\dots,z_{K}.
\]
Therefore, by Fubini's theorem, 
\begin{align*}
 & \left|P_{0}^{\otimes n}\left(\{z_{1},\dots,z_{K}:f(z)\leq q\}\right)-\tilde{P}\times P_{0}^{\otimes(n-1)}\left(\{z_{1},\dots,z_{K}:f(z)\leq q\}\right)\right|\\
= & \left|P_{0}^{\otimes(n-1)}\left(P_{0}(\{z_{1}:f(z_{1},z_{2},\dots,z_{K})\leq q\})-\tilde{P}(\{z_{1}:f(z_{1},z_{2},\dots,z_{K})\leq q\})\right)\right|\\
\leq & P_{0}^{\otimes(n-1)}Cn^{-(r+1)/2}\\
= & Cn^{-(r+1)/2}
\end{align*}
Similarly, for any $i=0,1,\dots,K-1$, we can show that 
\begin{align*}
 & \left|\tilde{P}^{\otimes i}\times P_{0}^{\otimes(n-i)}\left(\{z_{1},\dots,z_{K}:f(z)\leq q\}\right)-\tilde{P}^{\otimes(i+1)}\times P_{0}^{\otimes(n-i-1)}\left(\{z_{1},\dots,z_{K}:f(z)\leq q\}\right)\right|\\
= & \left|\tilde{P}^{\otimes i}\times P_{0}^{\otimes(n-i-1)}\left(P_{0}(\{z_{i+1}:f(z_{1},z_{2},\dots,z_{K})\leq q\})-\tilde{P}(\{z_{i+1}:f(z_{1},z_{2},\dots,z_{K})\leq q\})\right)\right|\\
\leq & \left|\tilde{P}^{\otimes i}\times P_{0}^{\otimes(n-i-1)}\right|Cn^{-(r+1)/2}\\
\leq & Cn^{-(r+1)/2}\left|\tilde{P}\right|^{i}\left(\mathbb{R}\right)
\end{align*}
Here for a signed measure $P$, $|P|$ stands for $P^{+}+P^{-}$ where
$P^{+},P^{-}$ are the positive and negative parts of $P$. By the
definition of $\tilde{P}$, we have $\left|\tilde{P}\right|\left(\mathbb{R}\right)=\int_{-\infty}^{\infty}\left|\frac{d}{dq}\left(\Phi(q)+\sum_{j=1}^{r}n^{-j/2}p_{j}(q)\phi(q)\right)\right|dq$
which can be uniformly bounded over all $n$. Hence $\left|\tilde{P}\right|\left(\mathbb{R}\right)\leq C_{1}<\infty$
where $C_{1}$ does not depend on $n$. Therefore, from the above
we have
\[
\left|\tilde{P}^{\otimes i}\times P_{0}^{\otimes(n-i)}\left(\{z_{1},\dots,z_{K}:f(z)\leq q\}\right)-\tilde{P}^{\otimes(i+1)}\times P_{0}^{\otimes(n-i-1)}\left(\{z_{1},\dots,z_{K}:f(z)\leq q\}\right)\right|\leq CC_{1}^{i}n^{-(r+1)/2}
\]
Summing the above over $i=0,1,\dots K-1$ and applying the triangle inequality,
we get that
\[
\left|P_{0}^{\otimes n}\left(\{z_{1},\dots,z_{K}:f(z)\leq q\}\right)-\tilde{P}^{\otimes n}\left(\{z_{1},\dots,z_{K}:f(z)\leq q\}\right)\right|\leq C\sum_{i=0}^{K-1}C_{1}^{i}n^{-(r+1)/2}
\]
which gives \eqref{eq: thm1_toshow}.

\end{proof}

\begin{proof}[Proof of Proposition \ref{prop: errorK2}]
The notations essentially follows from the notations in the proof
of Theorem 2. But to make things concrete we will introduce them again
for this example. Suppose that the batch averages multiplied by $\sqrt{n}$
are given by $(X_{1},Y_{1})$ and $(X_{2},Y_{2})$ (both of them are
2-d standard normal) Let $A_{0,1}=\frac{X_{1}+X_{2}}{2},A_{0,2}=\frac{Y_{1}+Y_{2}}{2},B_{1,1}=-B_{2,1}=\frac{X_{1}-X_{2}}{2},B_{1,2}=-B_{2,2}=\frac{Y_{1}-Y_{2}}{2}$
(note that $A_{0,1},A_{0,2},B_{1,1},B_{1,2}$ are standard normal
and independent). We have that 
\begin{align*}
W_{S} & =\sqrt{2}\frac{A_{0,1}+n^{-1/2}\left(A_{0,2}\right)^{2}}{\sqrt{\left(\left(B_{1,1}+n^{-1/2}Y_{1}^{2}-n^{-1/2}A_{0,2}^{2}\right)^{2}+\left(B_{2,1}+n^{-1/2}Y_{2}^{2}-n^{-1/2}A_{0,2}^{2}\right)^{2}\right)}}\\
 & =\frac{\sqrt{2}\left[A_{0,1}+n^{-1/2}\left(A_{0,2}\right)^{2}\right]}{\sqrt{\left(2B_{1,1}^{2}+n^{-1/2}A_{0,2}B_{1,1}B_{1,2}+n^{-1}\left(\left((B_{1,2}+A_{0,2})^{2}-A_{0,2}^{2}\right)^{2}+\left((B_{2,2}+A_{0,2})^{2}-A_{0,2}^{2}\right)^{2}\right)\right)}}\\
 & =\sqrt{2}\frac{A_{0,1}+n^{-1/2}\left(A_{0,2}\right)^{2}}{\sqrt{\left(2B_{1,1}^{2}+n^{-1/2}A_{0,2}B_{1,1}B_{1,2}+n^{-1}\left(B_{1,2}^{4}+8B_{1,2}^{2}A_{0,2}^{2}\right)\right)}}
\end{align*}
Therefore, $W_{S}\leq q$ can be written as 
\[
A_{0,1}\leq F_{+}^{(n)}:=\frac{q}{\sqrt{2}}\sqrt{\left(2B_{1,1}^{2}+n^{-1/2}A_{0,2}B_{1,1}B_{1,2}+n^{-1}\left(B_{1,2}^{4}+8B_{1,2}^{2}A_{0,2}^{2}\right)\right)}-n^{-1/2}A_{0,2}^{2}.
\]
Similarly, $W_{S}\geq-q$ is equivalent to 
\[
A_{0,1}\geq F_{-}^{(n)}:=-\frac{q}{\sqrt{2}}\sqrt{\left(2B_{1,1}^{2}+n^{-1/2}A_{0,2}B_{1,1}B_{1,2}+n^{-1}\left(B_{1,2}^{4}+8B_{1,2}^{2}A_{0,2}^{2}\right)\right)}-n^{-1/2}A_{0,2}^{2}.
\]
So we have $P(-q\leq W_{S}\leq q)=P(F_{-}^{(n)}\leq A_{0,1}\leq F_{+}^{(n)})=E\left[\Phi(F_{+}^{(n)})-\Phi(F_{-}^{(n)})\right]$
where $\Phi$ is the c.d.f. of standard normal. Let $F_{+}=\frac{q}{\sqrt{2}}\sqrt{2B_{1,1}^{2}}$
and $F_{-}=-\frac{q}{\sqrt{2}}\sqrt{2B_{1,1}^{2}}$. By the definition
of $t_{1}$ distribution, we have that $P(-q\leq t_{1}\leq q)=E\left[\Phi(F_{+})-\Phi(F_{-})\right]$.
Therefore, the coverage error, i.e., $\left|P(-q\leq W_{S}\leq q)-P(-q\leq t_{1}\leq q)\right|$,
can be expressed as 
\[
\left|E\left[\left(\Phi(F_{+}^{(n)})-\Phi(F_{+})\right)-\left(\Phi(F_{-}^{(n)})-\Phi(F_{-})\right)\right]\right|
\]
which by Taylor expansion can be expressed as
\[
\left|E\left[\phi(F_{+})(F_{+}^{(n)}-F_{+})+\phi^{\prime}(\xi_{+})(F_{+}^{(n)}-F_{+})^{2}-\phi(F_{-})(F_{-}^{(n)}-F_{-})-\phi^{\prime}(\xi_{-})(F_{+}^{(n)}-F_{+})^{2}\right]\right|
\]
for some $\xi_{+},\xi_{-}\in\mathbb{R}$. Note that $F_{+}=-F_{-}$
and $\phi(F_{+})=\phi(F_{-})$. So from the preceding we get that
the coverage error is
\begin{equation}
\left|E\left[\phi(F_{+})(F_{+}^{(n)}-F_{-}^{(n)}-2F_{+})+\phi^{\prime}(\xi_{+})(F_{+}^{(n)}-F_{+})^{2}-\phi^{\prime}(\xi_{-})(F_{-}^{(n)}-F_{-})^{2}\right]\right|.\label{eq: sec K2 er}
\end{equation}

We study $F_{+}^{(n)}-F_{+}$. We observe that 
\begin{align*}
 & \sqrt{2B_{1,1}^{2}+n^{-1/2}A_{0,2}B_{1,1}B_{1,2}+n^{-1}\left(B_{1,2}^{4}+8B_{1,2}^{2}A_{0,2}^{2}\right)}-\sqrt{2B_{1,1}^{2}}\\
= & \frac{n^{-1/2}A_{0,2}B_{1,1}B_{1,2}+n^{-1}\left(B_{1,2}^{4}+8B_{1,2}^{2}A_{0,2}^{2}\right)}{\sqrt{2B_{1,1}^{2}+n^{-1/2}A_{0,2}B_{1,1}B_{1,2}+n^{-1}\left(B_{1,2}^{4}+8B_{1,2}^{2}A_{0,2}^{2}\right)}+\sqrt{2B_{1,1}^{2}}}
\end{align*}
is dominated by (note that $2B_{1,1}^{2}+n^{-1/2}A_{0,2}B_{1,1}B_{1,2}+n^{-1}\left(B_{1,2}^{4}+8B_{1,2}^{2}A_{0,2}^{2}\right)\geq n^{-1}\left(B_{1,2}^{4}+8B_{1,2}^{2}A_{0,2}^{2}\right)-\frac{n^{-1}A_{0,2}^{2}B_{1,2}^{2}}{8}\geq\frac{1}{4}n^{-1}\left(B_{1,2}^{4}+8B_{1,2}^{2}A_{0,2}^{2}\right)$)
\[
\frac{n^{-1/2}\left|A_{0,2}B_{1,1}B_{1,2}\right|}{\sqrt{2B_{1,1}^{2}}}+n^{-1/2}\frac{B_{1,2}^{4}+8B_{1,2}^{2}A_{0,2}^{2}}{\frac{1}{2}\sqrt{B_{1,2}^{4}+8B_{1,2}^{2}A_{0,2}^{2}}}=\frac{n^{-1/2}\left|A_{0,2}B_{1,2}\right|}{\sqrt{2}}+2n^{-1/2}\sqrt{B_{1,2}^{4}+8B_{1,2}^{2}A_{0,2}^{2}}
\]
Therefore
\begin{equation}
\left|F_{+}^{(n)}-F_{+}\right|\leq n^{-1/2}\left(\frac{\left|A_{0,2}B_{1,2}\right|}{\sqrt{2}}+2\sqrt{B_{1,2}^{4}+8B_{1,2}^{2}A_{0,2}^{2}}+A_{0,2}^{2}\right).\label{eq: domi func}
\end{equation}
We note that the RHS is bounded by $n^{-1/2}$ times a polynomial
of normal which have arbitrary order of moments. Also note that the
derivative of normal density is uniformly bounded, we have that $E\left[\phi^{\prime}(\xi_{+})(F_{+}^{(n)}-F_{+})^{2}\right]=O(n^{-1})$.
Similarly we have $E\left[\phi^{\prime}(\xi_{-})(F_{-}^{(n)}-F_{-})^{2}\right]=O(n^{-1})$.
Therefore, by \eqref{eq: sec K2 er} we have that the coverage error
is
\begin{equation}
\left|E\left[\phi(F_{+})(F_{+}^{(n)}-F_{-}^{(n)}-2F_{+})\right]\right|+O(n^{-1}).\label{eq: sec K2 er2}
\end{equation}

Note that conditional on $B_{1,1}$, the expectation $E\left[\phi(F_{+})(F_{+}^{(n)}-F_{-}^{(n)}-2F_{+})|B_{1,1}\right]$
is always positve (which is implied by the positiveness of the second
order derivative in \eqref{eq: sec derivative}). Also note that $\phi(F_{+})$
is uniformly bounded away from 0 when $B_{1,1}$ belongs to a bounded
set. Therefore, to show that the above is $\omega(n^{-1})$, it suffices
to show that 
\[
E\left[F_{+}^{(n)}-F_{-}^{(n)}-2F_{+};\left|B_{1,1}\right|\leq M\right]=\omega(n^{-1}).
\]
Plugging in the expressions for $F_{+}^{(n)},F_{-}^{(n)},F_{+}$,
the above can be restated as: for independent $X,Y,Z\sim N(0,1)$,
we want to show that 
\begin{equation}
\lim_{n\rightarrow\infty}n\mathbb{E}\left[\sqrt{2X^{2}+n^{-1/2}XYZ+n^{-1}\left(Y^{4}+8Y^{2}Z^{2}\right)}-\sqrt{2X^{2}};|X|\leq M\right]=\infty\label{eq: batch K2 lim}
\end{equation}
To show this, we let $f(h)=\sqrt{2X^{2}+hXYZ+h^{2}\left(Y^{4}+8Y^{2}Z^{2}\right)}$.
Then we have that $Ef^{\prime}(0)=E\frac{XYZ}{2\sqrt{2}\left|X\right|}=0$,
and for any $\theta\in[0,h]$ we have that 
\begin{equation}
f^{\prime\prime}(\theta)=\frac{8X^{2}Y^{4}+63X^{2}Y^{2}Z^{2}}{4\left(2X^{2}+\theta XYZ+\theta^{2}\left(Y^{4}+8Y^{2}Z^{2}\right)\right)^{3/2}}\geq\frac{8X^{2}Y^{4}+63X^{2}Y^{2}Z^{2}}{4\left(2X^{2}+h\left|XYZ\right|+h^{2}\left(Y^{4}+8Y^{2}Z^{2}\right)\right)^{3/2}}.\label{eq: sec derivative}
\end{equation}
By Taylor expansion, we have that $f(h)=f(0)+f^{\prime}(0)h+\frac{1}{2}f^{\prime\prime}(\theta_{h})h^{2}$.
Therefore, 
\begin{align*}
E\left[\frac{f(h)-f(0)}{h^{2}};|X|\leq M\right] & =\frac{1}{2}E\left[f^{\prime\prime}(\theta_{h});|X|\leq M\right]\\
 & \geq E\left[\frac{8X^{2}Y^{4}+63X^{2}Y^{2}Z^{2}}{4\left(2X^{2}+h\left|XYZ\right|+h^{2}\left(Y^{4}+8Y^{2}Z^{2}\right)\right)^{3/2}};|X|\leq M\right]
\end{align*}
By monotone convergence, the RHS goes to $\infty$ as $h\rightarrow0$
(note that $E\left[\left|X\right|^{-1};\left|X\right|\leq M\right]=\infty$).
Let $h=n^{-1/2}$, we get \eqref{eq: batch K2 lim}.

To show that the coverage error of sectioning is $o(n^{-1/2})$, from
\eqref{eq: sec K2 er2} it suffices to show that 
\[
\left|E\left[\phi(F_{+})(F_{+}^{(n)}-F_{-}^{(n)}-2F_{+})\right]\right|=o(n^{-1/2}).
\]
From the boundness of $\phi(F_{+})$ and the positiveness of the conditional
distribution on $B_{1,1}$, it suffices to show that $E\left[F_{+}^{(n)}-F_{-}^{(n)}-2F_{+}\right]=o(n^{-1/2})$,
which can be restated as: for independent $X,Y,Z\sim N(0,1)$, we
want to show that
\[
\lim_{n\rightarrow\infty}\sqrt{n}\mathbb{E}\left[\sqrt{2X^{2}+n^{-1/2}XYZ+n^{-1}\left(Y^{4}+8Y^{2}Z^{2}\right)}-\sqrt{2X^{2}}\right]=0
\]
Or equivalently
\[
\lim_{n\rightarrow\infty}E\left[\frac{XYZ+n^{-1/2}\left(Y^{4}+8Y^{2}Z^{2}\right)}{\sqrt{2X^{2}+n^{-1/2}XYZ+n^{-1}\left(Y^{4}+8Y^{2}Z^{2}\right)}+\sqrt{2X^{2}}}\right]=0
\]
This can be shown by dominated convergence with dominating function
given by $\frac{\left|YZ\right|}{\sqrt{2}}+2\sqrt{Y^{4}+8Y^{2}Z^{2}}$
(see the derivation of \eqref{eq: domi func}).

\end{proof}

\begin{proof}[Proof of Theorem \ref{thm: validity dependent}]
Let $\bar{g}_k$ denote the average of $g(X_i)$ for all $X_i$ in the $k$-th section. From the mixing condition, since the gap between sections is of order $n^{\delta}$, we have that for each $k_1\neq k_2$ and measurable sets $A_{k_1},A_{k_2}$, $\left|P(\bar{g}_{k_1}\in A_{k_1},\bar{g}_{k_2}\in A_{k_2})-P(\bar{g}_{k_1}\in A_{k_1})P(\bar{g}_{k_2}\in A_{k_2})\right|\leq \alpha(n^{\delta})=o(n^{-(r+1)/2})$. It is not hard to see that by applying this for $K-1$ times, we will get that for any $A_1,A_2,\dots,A_K$, we have that 
\begin{equation}\label{eq: K weak dependence}
\left|P(\bar{g}_1\in A_1,\bar{g}_2\in A_2,\dots,\bar{g}_K\in A_K)-P(\bar{g}_1\in A_1)\dots P(\bar{g}_K\in A_K)\right|\leq K\alpha(n^\delta)=o(n^{-(r+1)/2}).
\end{equation}
Therefore, to study the joint distribution of $\bar{g}_1,\dots,\bar{g}_K$, it suffices to look at the marginal distributions.

For these marginal distributions, by Theorem 1 of \cite{jensen1989asymtotic}, we have that 
\begin{equation}\label{eq: Edgeworth dependent}
P(\sqrt{n}(\bar{g}_1-Eg(X_1))\in B) = \int_B \phi_{\tilde{\Sigma}}(x)\sum_{j=1}^{r}n^{j/2}q_j(x)dx + O(n^{-(r+1)/2}) 
\end{equation}
uniformly over $B$ satisfying $\phi_{\tilde{\Sigma}}((\delta B)^\epsilon) <c\epsilon$ where $(\delta B)^{\epsilon}:=\{x: d(x,B)<\epsilon,x\notin B\}$ (heuristically, this means the area of the boundary of $B$ is bounded) and $q_j,j=1,2,\dots,r$ are polynomials. Combining this with \eqref{eq: K weak dependence}, and doing integration, we again get \eqref{eq: Edgeworth_approx} with $A_i$ replaced by $\sqrt{n}(\bar{g}_i-Eg(X_1))$. Therefore, following the proof of Theorem \ref{thm: validity}, we get the desired result.
\end{proof}

\begin{proof}[Proof of Theorem \ref{thm: validity dependent 2}]
The proof for the first part is almost identical to the proof of Theorem \ref{thm: validity dependent}. The only difference is that we use the Edgeworth expansion result in \cite{malinovskii1987limit} to conclude that (here all probabilities are taken under the stationary measure)
\begin{equation*}
    \sup_{x\in R}\left|P\left(\frac{\sqrt{a}}{\sigma_F\Sigma_{f,n}}\leq x\right)-\Phi(x)-\sum_{j=1}^rn^{-j/2}\tilde{q}_j(x)\phi(x)\right|=O(n^{-(r+1)/2}).
\end{equation*}
Using the above expansion to take the place of \eqref{eq: Edgeworth dependent}, and proceed as in the proof of Theorem \ref{thm: validity dependent}, we get the desired result. 

For the second part, we note that the coeffcient for the $n^{-1/2}$ term in the Edgeworth expansion given in \cite{malinovskii1987limit} (more precisely, the polynomial $q_1(x)$ in its Theorem 1) is an even polynomial. Therefore, with the same oddness and evenness argument as in the proof of Theorem \ref{thm: validity}, we get the claimed result.
\end{proof}

\begin{proof}[Proof of Theorem \ref{prop: alg correctness}]
Based on the approximation $W=\sqrt{K(K-1)}\frac{[u,A_{0}]+n^{-1/2}b_{1}+n^{-1}b_{2}}{\sqrt{E_{2}+n^{-1/2}\lambda+n^{-1}e}}+O_p(n^{-3/2})$, we may further do Taylor expansion and get that \begin{align*}
W & =\sqrt{K(K-1)}\frac{[u,A_{0}]+n^{-1/2}b_{1}+n^{-1}b_{2}}{\sqrt{E_{2}+n^{-1/2}\lambda+n^{-1}e}}\\
 & =\sqrt{K(K-1)}\frac{[u,A_{0}]+n^{-1/2}b_{1}+n^{-1}b_{2}}{\sqrt{E_{2}}}\left(1+n^{-1/2}\frac{\lambda}{E_{2}}+n^{-1}\frac{e}{E_{2}}\right)^{-1/2}\\
 & =\sqrt{K(K-1)}\frac{[u,A_{0}]+n^{-1/2}b_{1}+n^{-1}b_{2}}{\sqrt{E_{2}}}\left(1-\frac{1}{2}n^{-1/2}\frac{\lambda}{E_{2}}-\frac{1}{2}n^{-1}\frac{e}{E_{2}}+\frac{3}{8}n^{-1}\frac{\lambda^{2}}{E_{2}^{2}}\right)\\
 & =\frac{\sqrt{K(K-1)}}{\sqrt{E_{2}}}\left[a\left(1-\frac{1}{2}n^{-1/2}\frac{\lambda}{E_{2}}-\frac{1}{2}n^{-1}\frac{e}{E_{2}}+\frac{3}{8}n^{-1}\frac{\lambda^{2}}{E_{2}^{2}}\right)+b_{1}n^{-1/2}-b_{1}n^{-1}\frac{1}{2}\frac{\lambda}{c}+b_{2}n^{-1}\right]
\end{align*}

Let $F=\frac{W}{\sqrt{K(K-1)}}$. Let $F_{+}$ and $F_{-}$ denote the value of $A_{0,1}$ such that the value of $W$ is equal to $q$ and $-q$ when $n=\infty$, respectively. Then indeed we have that $F_{+}=\frac{q\sqrt{E_2}}{\sqrt{K(K-1)}}-\sum_{i=2}^d u_i A_{0,i}$ and $F_{-}=\frac{-q\sqrt{E_2}}{\sqrt{K(K-1)}}-\sum_{i=2}^d u_i A_{0,i}$. Denote $F_x,F_y$ as the derivative of $F$ w.r.t. $n^{-1/2}$ and $A_{0,1}$ respectively when $n^{-1/2}=0,A_{0,1}=F_{+}$. Let $F_{xx},F_{xy},F_{yy}$ be the corresponding second-order derivatives (note that we actually have $F_{yy}=0$ since when $n^{-1/2}=0$, the only term is $\sqrt{K(K-1)}\frac{[u,A_0]}{\sqrt{E_2}}$ which is linear in $A_{0,1}$). 
Then $F_{x}=\frac{b_{1}}{\sqrt{E_{2}}}-\frac{1}{2}\frac{\lambda a}{E_{2}\sqrt{E_{2}}},F_{xx}=\frac{1}{\sqrt{E_{2}}}\left(a\left(-\frac{e}{E_{2}}+\frac{3}{4}\frac{\lambda^{2}}{E_{2}^{2}}\right)-b_{1}\frac{\lambda}{c}+2b_{2}\right).$
\[
F_{y}=\frac{a^{\prime}}{\sqrt{E_{2}}}=\frac{u_{1}}{\sqrt{E_{2}}}.
\]
\[
F_{xy}=\frac{b_{1}^{\prime}}{\sqrt{E_{2}}}-\frac{1}{2}\frac{\left(\lambda a\right)^{\prime}}{E_{2}\sqrt{E_{2}}}.
\]

Therefore, 
\[
F_{x}+F_{y}y_{x}=0\Rightarrow y_{x}=-F_{x}/F_{y}
\]
\[
F_{xx}+2F_{xy}y_{x}+F_{yy}y_{x}^{2}+F_{y}y_{xx}=0\Rightarrow y_{xx}=-\left(F_{xx}+2F_{xy}y_{x}\right)/F_{y}
\]
By a second-order Taylor expansion,
we have that when $W = q$,
\[
A_{0,1} = F_{+}^{(n)}=F_{+}+n^{-1/2}y_{x}+\frac{1}{2}n^{-1}y_{xx}+O_p(n^{-3/2})
\]
and similarly when $W = -q$,
\[
A_{0,1} = F_{-}^{(n)}=F_{-}+n^{-1/2}y_{x}^{(-)}+\frac{1}{2}n^{-1}y_{xx}^{(-)}+O_p(n^{-3/2})
\]
where $y_x^{-}$ and $y_{xx}^{(-)}$ are the counterpart of the derivatives at $(n^{-1/2},A_{0,1})=(0,F_{-})$.
Then, we have that
\begin{align*}
P(-q & \leq W\leq q)=\int\int_{F_{-}^{(n)}}^{F_{+}^{(n)}}\phi_{\tilde{\Sigma}}(a)\left(1+\sum_{j=1}^{2}n^{-j/2}p_{j,A}(a)\right)da_{0,1}da^{\prime}+O(n^{-3/2})\\
 & =\int\int_{F_{-}}^{F_{+}}\phi_{\tilde{\Sigma}}(a)\left(1+\sum_{j=1}^{2}n^{-j/2}p_{j,A}(a)\right)da_{0,1}da^{\prime}\\
 & +\int\left(\begin{array}{c}
\left[\phi_{\tilde{\Sigma}}(a)\left(1+n^{-1/2}p_{1,A}(a)\right)\right]|_{a_{0,1}=F_{+}}\left(n^{-1/2}y_{x}+\frac{1}{2}n^{-1}y_{xx}\right)\\
-\left[\phi_{\tilde{\Sigma}}(a)\left(1+n^{-1/2}p_{1,A}(a)\right)\right]|_{a_{0,1}=F_{-}}\left(n^{-1/2}y_{x}^{(-)}+\frac{1}{2}n^{-1}y_{xx}^{(-)}\right)
\end{array}\right)da^{\prime}\\
 & +\frac{1}{2}n^{-1}\int\left(\begin{array}{c}
\frac{\partial}{\partial a_{0,1}}\left[\phi_{\tilde{\Sigma}}(a)\right]|_{a_{0,1}=F_{+}}y_{x}^{2}\\
-\frac{\partial}{\partial a_{0,1}}\left[\phi_{\tilde{\Sigma}}(a)\right]|_{a_{0,1}=F_{-}}y_{x}^{(-)2}
\end{array}\right)da^{\prime}+O(n^{-3/2}).
\end{align*}
Hence, the $n^{-1}$ error term is given by
\begin{align}
 & \int\int_{F_{-}}^{F_{+}}\phi_{\tilde{\Sigma}}(a)\left(p_{2,A}(a)\right)da_{0,1}da^{\prime}+\int\left(\begin{array}{c}
\left[\phi_{\tilde{\Sigma}}(a)\left(p_{1,A}(a)\right)\right]|_{a_{0,1}=F_{+}}\left(y_{x}\right)\\
-\left[\phi_{\tilde{\Sigma}}(a)\left(p_{1,A}(a)\right)\right]|_{a_{0,1}=F_{-}}\left(y_{x}^{(-)}\right)
\end{array}\right)da^{\prime}\nonumber \\
 & +\int\left(\begin{array}{c}
\left[\phi_{\tilde{\Sigma}}(a)\right]|_{a_{0,1}=F_{+}}\left(\frac{1}{2}y_{xx}\right)\\
-\left[\phi_{\tilde{\Sigma}}(a)\right]|_{a_{0,1}=F_{-}}\left(\frac{1}{2}y_{xx}^{(-)}\right)
\end{array}\right)da^{\prime}+\frac{1}{2}\int\left(\begin{array}{c}
\frac{\partial}{\partial a_{0,1}}\left[\phi_{\tilde{\Sigma}}(a)\right]|_{a_{0,1}=F_{+}}y_{x}^{2}\\
-\frac{\partial}{\partial a_{0,1}}\left[\phi_{\tilde{\Sigma}}(a)\right]|_{a_{0,1}=F_{-}}y_{x}^{(-)2}
\end{array}\right)da^{\prime}\label{eq: error4sum}
\end{align}

Note that the joint density of $\left(\sqrt{n}\left(X_{1}-m\right),\dots,\sqrt{n}\left(X_{K}-m\right)\right)=(A_{0}+B_{1},A_{0}+B_{2},\dots,A_{0}+B_{K-1},A_{0}-B_{1}-B_{2}-\dots-B_{K-1})$
is given by 
\[
p(x)=\phi_{\sigma}(x_{1})\dots\phi_{\sigma}(x_{K})\prod_{i}(1+n^{-1/2}p_{1}(x_{i})+n^{-1}p_{2}(x_{i}))
\]
The formulas for $p_{1}(x)$ and $p_{2}(x)$, given in \eqref{eq: Edgeworth formula}, can be found in \cite{skovgaard1986multivariate}, p172.
Therefore, after a linear transformation,
we get that the density of $A=(A_{0},B_{1},\dots,B_{K-1})$ is
\[
p_{A}(a)=p(x(a))\left|det\left(\frac{Dx}{Da}\right)\right|=Kp(x(a))=\phi_{\Sigma}(a)\prod_{i}(1+n^{-1/2}p_{1}(x_{i}(a))+n^{-1}p_{2}(x_{i}(a))) + O(n^{-3/2}).
\]
(Note that, since $\frac{Dx}{Da}=\left(\begin{array}{ccccc}
1 & 1\\
1 &  & 1\\
\dots &  &  & \dots\\
1 &  &  &  & 1\\
1 & -1 & -1 & \dots & -1
\end{array}\right),$ we have that its determinant is $K$.) This gives $p_{1}(a)=p_{1}(x_{1}(a))+\dots+p_{1}(x_{K}(a))$,
and $p_{2}(a)=\sum_{1\leq i<j\leq K}p_{1}(x_{i}(a))p_{1}(x_{j}(a))+\sum_{1\leq i\leq K}p_{2}(x_{i}(a))$.
Therefore, the first term in (\ref{eq: error4sum}) is
\[
E_{C}p_{2}(x)=E_{C}\left[\sum_{1\leq i<j\leq K}p_{1}(x_{i}(A))p_{1}(x_{j}(A))+\sum_{1\leq i\leq K}p_{2}(x_{i}(A))\right]
\]
where C represents the area where
\[
-q\leq\sqrt{K(K-1)}\frac{[u,A_{0}]}{\sqrt{\sum_{i=1}^{K}\left(\begin{array}{c}
[u,A_{i}-A_{0}]\end{array}\right)^{2}}}\leq q.
\]
For the second term of (\ref{eq: error4sum}), we have that

\[
\int\left(\begin{array}{c}
\left[\phi_{\tilde{\Sigma}}(a)\left(p_{1}(a)\right)\right]|_{a_{0,1}=F_{+}}\left(y_{x}\right)\\
-\left[\phi_{\tilde{\Sigma}}(a)\left(p_{1}(a)\right)\right]|_{a_{0,1}=F_{-}}\left(y_{x}^{(-)}\right)
\end{array}\right)da^{\prime}=E\left(\begin{array}{c}
\phi_{\tilde{\sigma}_{0}}(F_{+}-\mu)\left(p_{1}(a)\right)|_{a_{0,1}=F_{+}}\left(y_{x}\right)\\
-\phi_{\tilde{\sigma}_{0}}(F_{-}-\mu)\left(p_{1}(a)\right)|_{a_{0,1}=F_{-}}\left(y_{x}^{(-)}\right)
\end{array}\right)
\]
Here $\tilde{\sigma}_{0}=\left(\sigma_{0}-\sigma_{01}\sigma_{11}^{-1}\sigma_{10}\right)/K$,
$\sigma_{0}$ is the variance of the first dimension of $X$. $\mu=\sigma_{01}\sigma_{11}^{-1}a_{0}^{\prime}$.
The expectation is taken under the limiting normal distribution of
$a^{\prime}$. We explain why the above is true. Observe that $\phi_{\tilde{\sigma}_{0}}\left(\cdot-\mu\right)$
is the conditional density of $A_{0,1}$ given $A^{\prime}$ (under
their joint limiting distribution), which implies $\phi_{\Sigma}(a)=\phi_{\tilde{\sigma}_{0}}(F_{+}-\mu)\phi_{\Sigma^{\prime}}(a^{\prime})$.
Moreover, it follows from the definition of expectation that $\int\left[\phi_{\Sigma^{\prime}}(a^{\prime})g(a^{\prime})\right]da^{\prime}=Eg(A^{\prime})$
where $A^{\prime}$ follows from its limiting normal distribution.
Therefore, we have the preceding displayed equality. Similar argument
holds for the third and last terms of (\ref{eq: error4sum}). Therefore,
the summation of the last three terms of (\ref{eq: error4sum}) is
\begin{align*}
 & E\left(\begin{array}{c}
\phi_{\tilde{\sigma}_{0}}(F_{+}-\mu)\left(p_{1}(a)\right)|_{a_{0,1}=F_{+}}\left(y_{x}\right)\\
-\phi_{\tilde{\sigma}_{0}}(F_{-}-\mu)\left(p_{1}(a)\right)|_{a_{0,1}=F_{-}}\left(y_{x}^{(-)}\right)
\end{array}\right)+E\left(\begin{array}{c}
\phi_{\tilde{\sigma}_{0}}(F_{+}-\mu)\left(\frac{1}{2}y_{xx}\right)\\
-\phi_{\tilde{\sigma}_{0}}(F_{-}-\mu)\left(\frac{1}{2}y_{xx}^{(-)}\right)
\end{array}\right)\\
 & +\frac{1}{2}E\left(\begin{array}{c}
\phi_{\tilde{\sigma}_{0}}(F_{+}-\mu)(-\left(F_{+}-\mu\right)/\tilde{\sigma}_{0})y_{x}^{2}\\
-\phi_{\tilde{\sigma}_{0}}^{\prime}(F_{-}-\mu)(-\left(F_{-}-\mu\right)/\tilde{\sigma}_{0})y_{x}^{(-)2}
\end{array}\right)
\end{align*}
 Here, the expectation is taken under the asymptotic normal distribution
of $A^{\prime}$.
\end{proof}

\begin{proof}[Proof of Proposition \ref{prop: batching small K}]
For batching, the event $P(-q\leq W_{B}\leq q)$ could be written
as (recall that $X_{i},Y_{i},i=1,2$ corresponds to the scaled batch
averages and has standard normal distribution, and are independent)
\[
-\frac{q}{\sqrt{2}}\leq\frac{\frac{X_{1}+X_{2}}{2}+n^{-1/2}\frac{Y_{1}^{2}+Y_{2}^{2}}{2}}{\left|\frac{X_{1}-X_{2}}{2}\right|\left|1+n^{-1/2}\frac{Y_{1}^{2}-Y_{2}^{2}}{X_{1}-X_{2}}\right|}\leq\frac{q}{\sqrt{2}}
\]
the critical points for $(X_{1}+X_{2})/2$ (which corresponds to $F_{+}^{n}$ and $F_{-}^{n}$ in the proof of Theorem \ref{thm: validity}) are given by
\[
\frac{q}{\sqrt{2}}\left|\frac{X_{1}-X_{2}}{2}\right|\left|1+n^{-1/2}\frac{Y_{1}^{2}-Y_{2}^{2}}{X_{1}-X_{2}}\right|-n^{-1/2}\frac{Y_{1}^{2}+Y_{2}^{2}}{2}
\]
and 
\[
-\frac{q}{\sqrt{2}}\left|\frac{X_{1}-X_{2}}{2}\right|\left|1+n^{-1/2}\frac{Y_{1}^{2}-Y_{2}^{2}}{X_{1}-X_{2}}\right|-n^{-1/2}\frac{Y_{1}^{2}+Y_{2}^{2}}{2}
\]

However, when $K=2$, since the probability that $1+n^{-1/2}\frac{Y_1^2-Y_2^2}{X_1-X_2}>0$ is not $o(n^{-1})$, the above can not be approximated by (with error
$o_{p}(n^{-1})$) dropping the absolute value whose resulting value
is
\[
\frac{q}{\sqrt{2}}\left|\frac{X_{1}-X_{2}}{2}\right|\left(1+n^{-1/2}\frac{Y_{1}^{2}-Y_{2}^{2}}{X_{1}-X_{2}}\right)-n^{-1/2}\frac{Y_{1}^{2}+Y_{2}^{2}}{2}
\]
and 
\[
-\frac{q}{\sqrt{2}}\left|\frac{X_{1}-X_{2}}{2}\right|\left(1+n^{-1/2}\frac{Y_{1}^{2}-Y_{2}^{2}}{X_{1}-X_{2}}\right)-n^{-1/2}\frac{Y_{1}^{2}+Y_{2}^{2}}{2}.
\]
On the other hand, dropping the absolute value operation exactly corresponds
to what is done using the current expansion scheme:
\begin{align*}
\left|1+n^{-1/2}\frac{Y_{1}^{2}-Y_{2}^{2}}{X_{1}-X_{2}}\right| & =\sqrt{1+2n^{-1/2}\frac{Y_{1}^{2}-Y_{2}^{2}}{X_{1}-X_{2}}+n^{-1}\left(\frac{Y_{1}^{2}-Y_{2}^{2}}{X_{1}-X_{2}}\right)^{2}}\\
 & =1+n^{-1/2}\frac{Y_{1}^{2}-Y_{2}^{2}}{X_{1}-X_{2}}+0+M\frac{n^{-3/2}}{\left(X_{1}-X_{2}\right)^{3}}.
\end{align*}
Though the residual has a coefficient $n^{-3/2}$, its contribution is actually larger since $\frac{1}{(X_1-X_2)^3}$ does not have finite expectation.
\end{proof}

\begin{proof}[Proof of Theorem \ref{thm: K_large}]
As $K\rightarrow\infty$, $S_{\text{batch}}\rightarrow \sqrt{n\text{Var}\psi(\hat{P}_1)}$. But $\sqrt{nK}\left(\frac{1}{K}\sum_{i}\psi\left(\hat{P}_{i}\right)-\psi\right)\rightarrow\text{sign}(E\psi(\hat{P}_{1})-\psi)\cdot\infty$. Thus $W_B$ either converges to $\infty$ or $-\infty$ which implies $P(-q\leq W_B\leq q)\rightarrow 0$. Similarly, since $$\left|S_{\text{sec}}^2-\frac{1}{K-1}\sum_{i=1}^{K}\left(\psi\left(\hat{P}_{i}\right)-\psi\right)^{2}\right|\leq \frac{K}{K-1}(\psi(\hat{P})-\psi)^2=o_p(1),$$ {by the strong law of large numbers and Slutsky's theorem we have} $S_{\text{sec}}\rightarrow \sqrt{E(\psi(\hat{P}_1)-\psi)^2}$ ({as $K\rightarrow\infty$}). Also note that by the differentiability of $\psi$ implies the asymptotic result $\sqrt{nK}\left(\psi(\hat{P})-\psi\right)\Rightarrow N(0,\sigma^{2})$. Therefore, we get the claim for $P(-q\leq W_S \leq q)$. The claim that $P(-q\leq W_{SJ}\leq q)\rightarrow \Phi(q)-\Phi(-q)$ can be seen by regarding each batch as a multidimensional data and applying the consistency of the usual jackknife (e.g. Theorem 2.3 of \cite{Shao1995}).
\end{proof}

\section{Edgeworth Expansion for Smooth Function Models}
For completeness, we introduce the Edgeworth expansion result in \cite{bhattacharya2010normal}. Let $M_{s}(f):=\begin{cases}
\sup_{x\in R^{k}}(1+\left\Vert x\right\Vert ^{s})^{-1}\left|f(x)\right| & s>0\\
\omega_{f}(R^{k})=\sup_{x,y\in R^{k}}\left|f(x)-f(y)\right| & s=0
\end{cases}.$ Let $\omega_{f}(\epsilon,x):=\sup_{d(x^{\prime},x)\leq\epsilon}f(x^{\prime})-\inf_{d(x^{\prime},x)\leq\epsilon}f(x^{\prime})$,\\ $\bar{\omega}_{f}(\epsilon,\mu):=\int\omega_{f}(\epsilon,x)d\mu(x)$. Heuristically, $\omega_{f}(\epsilon,x)$ measures the oscillation of $f(x)$ in a $\epsilon$-neighborhood of $x$, which is then averaged over $x$ w.r.t. measure $\mu$ to give $\bar{\omega}_{f}(\epsilon,\mu)$.
\begin{thm}\label{thm: Edgeworth}
Let $\left\{X_{{n}}: {n}>1\right\}$ be an i.i.d. sequence of random vectors with values in $\mathbb{R}^{{k}}$ whose common distribution ${Q}_{1}$ satisfies Cramer's condition \eqref{eq: cramer}. Assume that ${Q}_{1}$ has mean zero and a finite $s$-th absolute moment for some integer ${s}>3 .$ Let ${V}$ denote the covariance matrix of ${Q}_{1}$ and $\chi_{\nu}$ its $v t h$ cumulant $(3<|\nu|<{s})$. Then for every real-valued, Borel-measurable function f on $R^{k}$ satisfying
$$
M_{s^{\prime}}(f)<\infty
$$
for some ${s}^{\prime}, 0 \leqslant {~s}^{\prime} \leqslant {s}$, one has
$$
\begin{array}{l}
\left|\int f d\left(Q_{n}-\sum_{r=0}^{s-2} n^{-j / 2} p_{j}\left(-\Phi_{0, V}:\left\{\chi_{\nu}\right\}\right)\right)\right| \\
<M_{s^{\prime}}(f) \delta_{1}(n)+c(s, k) \bar{\omega}_{f}\left(2 e^{-\gamma n}: \Phi_{0, V}\right)
\end{array}
$$
where ${Q}_{{n}}$ is the distribution of ${n}^{-1 / 2}\left({X}_{1}+\cdots+{X}_{{n}}\right), \gamma$ is a suitable positive constant, and
$$
\delta_{1}(n)=o\left(n^{-(s-2) / 2}\right), \quad(n \rightarrow \infty)
$$
Moreover, $c(s, k)$ depends only on s and ${k}$, and the quantities $\gamma, \delta_{1}({n})$ do not depend on ${f}$.
\end{thm}
\end{document}